\documentclass[12pt]{amsart}
\usepackage{amssymb}
\usepackage{amscd}
\usepackage{amsmath}
\usepackage{amsmath}
\usepackage[all]{xy}
\usepackage{mathrsfs}

 \usepackage{color}

\theoremstyle{plain}
\newtheorem{theorem}{Theorem}[section]
\newtheorem{lemma}[theorem]{Lemma}
\newtheorem{corollary}[theorem]{Corollary}
\newtheorem{proposition}[theorem]{Proposition}
\newtheorem{conjecture}[theorem]{Conjecture}

\theoremstyle{definition}
\newtheorem{definition}[theorem]{Definition}

\newtheorem{examples}[theorem]{Examples}
\newtheorem*{notation}{Notation}
\newtheorem{remark}[theorem]{Remark}

  1

\DeclareMathOperator{\Gal}{Gal}

\DeclareMathOperator{\Hom}{Hom}

\DeclareMathOperator{\Pic}{Pic}
\DeclareMathOperator{\Spec}{Spec}
\DeclareMathOperator{\N}{N}

\DeclareMathOperator{\im}{im}
\DeclareMathOperator{\coker}{coker}

\DeclareMathOperator{\Fitt}{Fitt}

\newcommand{\CC}{\mathbb{C}}

\newcommand{\GG}{\mathbb{G}}

\newcommand{\QQ}{\mathbb{Q}}

\newcommand{\RR}{\mathbb{R}}

\newcommand{\ZZ}{\mathbb{Z}}

\newcommand{\G}{\mathcal{G}}

\newcommand{\cN}{\mathcal{N}}
\newcommand{\cO}{\mathcal{O}}

\newcommand{\cY}{\mathcal{Y}}
\newcommand{\cX}{\mathcal{X}}

\newcommand{\frp}{\mathfrak{p}}

\newcommand{\Ann}{\mathrm{Ann}}
\newcommand{\bz}{\mathbb{Z}}

\newcommand{\La}{\Lambda}

\newcommand{\Log}{\mathrm{Log}}
\newcommand{\ord}{\mathrm{ord}}
\newcommand{\Ord}{\mathrm{Ord}}

\newcommand{\Fr}{\mathrm{Fr}}

\begin{document}

%%%%%%%%%%%%%%%%%%%%%%%%%%%%%%%
%%% TITLE, AUTHOR, ABSTRACT %%%
%%%%%%%%%%%%%%%%%%%%%%%%%%%%%%%
\title[]{Iwasawa theory and zeta elements for $\GG_m$}
%\\
%-- Tate motives in the abelian case}

\author{David Burns, Masato Kurihara and Takamichi Sano}
%\thanks{Preliminary version of June, 2015}

\begin{abstract}
We describe an explicit `higher rank' Iwasawa theory for zeta elements associated to the multiplicative group over abelian extensions of general number fields. We then show that this theory leads to a concrete new strategy for proving special cases of the equivariant Tamagawa number conjecture. As a first application of this approach, we use it to prove new cases of the conjecture for Tate motives over natural families of abelian CM-extensions of totally real fields for which the relevant $p$-adic $L$-functions possess trivial zeroes.
\end{abstract}

\address{King's College London,
Department of Mathematics,
London WC2R 2LS,
U.K.}
\email{david.burns@kcl.ac.uk}

\address{Keio University,
Department of Mathematics,
3-14-1 Hiyoshi\\Kohoku-ku\\Yokohama\\223-8522,
Japan}
\email{kurihara@math.keio.ac.jp}

\address{Keio University,
Department of Mathematics,
3-14-1 Hiyoshi\\Kohoku-ku\\Yokohama\\223-8522,
Japan}
\email{tkmc310@a2.keio.jp}

\maketitle

%\section{Introduction} \label{Intro}%This is an updated version of the submitted file dburns.tex
\tableofcontents
\section{Introduction}

In our previous article \cite{bks1}, we showed that a natural theory of
zeta elements
associated to the multiplicative group $\mathbb{G}_m$ over finite abelian
extensions of number fields shed light on the equivariant Tamagawa
number conjecture (or eTNC for short in the remainder of this introduction)
in the setting of untwisted Tate motives.

In particular, in this way we derived a wide range of explicit results
and predictions concerning, amongst other things, families of fine
integral congruence relations between Rubin-Stark elements of different ranks
and also several aspects of the detailed Galois module structures of
ideal class groups and of the natural Selmer groups
(and their homotopy-theoretic transposes) that are associated to
$\mathbb{G}_m$. For details see \cite{bks1}.

The main aims of the present article are now to develop
an explicit Iwasawa theory for these zeta elements, to use this theory to describe a new approach to proving
some important special cases of the eTNC and to describe some initial concrete applications of this approach.

In the next two subsections we discuss briefly the main results that we shall obtain in this direction.

\vspace{-3mm}

\subsection{Iwasawa main conjectures for general number fields}\label{iwc intro}
The first key aspect of our approach is the formulation of
an explicit main conjecture of Iwasawa theory for abelian extensions of {\it general} number fields (we refer to this conjecture as a
`higher rank main conjecture' since the rank of any associated
 Euler system
%(in the sense of Mazur and Rubin in \cite{MR2})
would in most cases be greater than one).

%and also
%an Iwasawa-theoretic version of a strengthening of
%the `refined class number formula for $\mathbb{G}_m$'
%that was recently formulated by the third author in \cite{sano}
%and also independently by Mazur and Rubin in \cite{MRGm}.

%Our strategy for proving the eTNC for untwisted Tate motives is then obtained by combining these ideas with an observation concerning a conjecture of Gross from \cite{Gp} (as generalized by Jaulent in \cite{jaulent} and interpreted by Kolster in \cite{kolster}) and can be seen to provide a natural extension to general number fields of the strategy used successfully by  Greither and the first author in \cite{bg}, by Flach in \cite{fg} and by Bley in \cite{bley} to prove the same case of the eTNC in the restricted context of abelian extensions of either $\QQ$ or of suitable imaginary quadratic fields.

%As a first concrete application of this strategy, we shall then combine it with recent work of Darmon, Dasgupta and Pollack \cite{DDP} and of Ventullo \cite{ventullo} concerning the $p$-adic Gross-Stark Conjecture to derive verifications of the eTNC in a family of interesting new cases.

%Our first subject in this paper is to formulate the Iwasawa main conjecture
%for {\it general} number fields.
%and its $\ZZ_{p}$-extension.

To give a little more detail we fix a finite abelian extension $K/k$ of general number fields and a $\ZZ_{p}$-extension $k_{\infty}$ of $k$ and set  $K_{\infty}=Kk_{\infty}$. In this introduction, we suppose that $k_{\infty}/k$ is the cyclotomic $\ZZ_{p}$-extension but this is purely for simplicity.

Our {\it higher rank main conjecture} is first stated (as Conjecture \ref{IMC}, which for pedagogical reasons we refer to as (hIMC) in this introduction)
in terms of the existence of an Iwasawa-theoretic zeta element
 which effectively plays the role of $p$-adic $L$-functions
for general number fields and has precise prescribed interpolation
properties in terms of the values at zero of the higher derivatives
of abelian $L$-series.

We then subsequently reinterpret this conjecture, occasionally under
suitable hypotheses, in several more explicit ways: firstly,
in terms of the properties of Iwasawa-theoretic `Rubin-Stark elements'
(see Conjecture \ref{IMC rubinstark}), then in terms of the existence
of natural Iwasawa-theoretic measures (see Conjecture \ref{measure conjecture}),
then in terms of the explicit generation, after localization
at height one prime ideals, of the higher exterior powers of
Iwasawa-theoretic unit groups (see Conjecture \ref{IMCexplicit})
and finally in a very classical way in terms of the characteristic ideals
of suitable ideal class groups and concrete torsion modules constructed
from Rubin-Stark elements (see Conjecture \ref{char IMC};
in fact, readers who are familiar with the classical
formulation of main conjectures may wish to look at this formulation of the conjecture first).
%See also Proposition \ref{measureIMC} for its interpretation by measures,
%and Conjecture \ref{IMCexplicit}, Propositions \ref{equiv prop}, \ref{IMC4}
%for the formulation using usual characteristic ideals.

In particular, for the minus part of a CM-abelian extension of a totally real field, we show that (hIMC) is equivalent to the usual main conjecture involving the $p$-adic $L$-function of Deligne-Ribet (see the proof of Theorem \ref{CM theorem}(i)) and that for a real abelian field over $\QQ$ it is equivalent to the standard formulation of a main conjecture involving cyclotomic units.

In general, the conjecture (hIMC) is implied by the validity of the relevant case of the eTNC for extensions $K'/k$ as $K'$ runs over all (sufficiently large) finite extensions of $K$ in $K_{\infty}$ but is usually very much weaker than the eTNC (and hence also than the main conjecture formulated by Fukaya and Kato in \cite{FukayaKato}).

For example, if any $p$-adic prime of $k$ splits completely in $K$,
then our conjectured zeta element encodes no information
concerning the $L$-values of characters of $\Gal(K/k)$.
More precisely, if for any character $\chi$ of $\Gal(K/k)$
there exists a $p$-adic prime of $k$ whose decomposition subgroup
in $\Gal(K/k)$ is contained in the kernel of $\chi$
(in which case one says that the $p$-adic $L$-function
for $\chi$ has a trivial zero at $s=0$), then the zeta element that
we predict to exist has no conjectured interpolation property
involving the leading term of $L_{k,S,T}(\chi,s)$ at $s=0$.

%The Rubin-Stark conjecture predicts the existence of the Rubin-Stark element
%which is a generalization of Stark units, and
%which is an algebraic element related to the $r$-th derivatives of $L$-values (see \cite{R}).
%We interpret the Rubin-Stark element as the canonical image of our zeta element as in
%our previous paper \cite{bks1},
%and reformulate the main conjecture in a very classical way, using the characteristic ideal of
%certain class groups and Rubin-Stark elements (see Conjecture \ref{char IMC}).
%\vspace{5mm}
%We next explain our main theorem in this paper which
%gives {\it the first theoretical result that the eTNC holds even
%when the trivial zero occurs}. Our main theorem is regarded as a generalization of the main results by Greither and the first author \cite{bg} and %Flach \cite{fg}, and of Bley \cite{bley}, in which the eTNC is proved for abelian extensions over $\QQ$ and imaginary quadratic fields respectively.

\vspace{-2mm}

\subsection{eTNC and congruences between Rubin-Stark elements}\label{rs intro}
We now turn to discuss how (hIMC) leads to a concrete strategy to prove some interesting new cases of the eTNC.

Here, a key role is played by
a detailed Iwasawa-theoretic study of the fine congruence relations
between Rubin-Stark elements of differing ranks that were independently formulated in the context of finite abelian extensions by Mazur and Rubin in \cite{MRGm} (where the congruences are referred to as a `refined class number formula for $\mathbb{G}_m$') and by the third author in \cite{sano}. In particular, working in the setting of the extension $K_{\infty}/k$
we formulate an explicit conjecture, denoted for convenience (MRS) here, which, roughly speaking, describes the precise relation between
the natural Rubin-Stark elements for $K_{\infty}/k$ and for $K/k$. (For full details see Conjectures \ref{mrs1} and \ref{mrs2}).

To better understand the context of this conjecture we prove in Theorem \ref{GS thm} that it constitutes a natural generalization of the so-called `Gross-Stark conjecture' formulated by Gross in \cite{Gp}.

It is easy to see that, as already observed above, the eTNC implies
the validity of (hIMC) and, in addition, one of the main results of
our previous work \cite{bks1} allows us to prove in a straightforward way that it also implies the validity of (MRS).

One of the key observations of the present article is that, much more significantly, one can prove under certain natural hypotheses a powerful converse to these implications. To be a little more precise, we shall prove a result of the following sort (for a detailed statement of which see Theorem \ref{mainthm}).

\begin{theorem}\label{IntroTh}
If the Galois coinvariants of a certain natural Iwasawa module is finite
(as has been conjectured to be the case by Gross), then the validity of both
(hIMC) and (MRS) for the extension $K_{\infty}/k$ combine to imply
the validity of the $p$-component of the eTNC for every finite subextension
$F/k$ of $K_\infty/k$.
% (concerning more precise conditions,
%see Theorem \ref{mainthm}).
\end{theorem}

To give a first indication of the usefulness of the above theorem,
we apply it in the case that $k$ is totally real and $K$ is CM and
consider the `minus component'  of the $p$-part of the eTNC.
In this context we write $K^+$ for the maximal totally real subfield of $K$.

We recall that if no $p$-adic place splits in $K/K^{+}$
and the Iwasawa-theoretic $\mu$-invariant of $K_\infty/K$ vanishes,
then the validity of the minus component of the $p$-part of the eTNC is already known
(as far as we are aware, such a result was first implicitly discussed
in the survey article of Flach \cite{flachsurvey}).

However, by combining Theorems \ref{GS thm} and \ref{mainthm} with recent work of Darmon,
Dasgupta and Pollack \cite{DDP} and of Ventullo \cite{ventullo} on the Gross-Stark conjecture, we can now
%check that all the conditions of Theorem \ref{IntroTh} are satisfied, and can
prove the following result (for a precise statement of which see Corollary \ref{MC1}).

\begin{corollary}\label{IntroCor} Let $K$ be a finite CM-extension of a totally real field $k$. Let $p$ be an odd prime for which the Iwasawa-theoretic
$\mu$-invariant of $K_\infty/K$ vanishes and at most one $p$-adic place of $k$ splits in $K/K^{+}$. Then the minus component of the $p$-part of the eTNC for $K/k$ is (unconditionally) valid.
\end{corollary}

%Corollary \ref{IntroCor} (see Corollary \ref{MC1}).}
We remark that this result gives {\it the first verifications} of the (minus component of the $p$-part of the) eTNC for the untwisted Tate motive over abelian CM-extensions
of a totally real field that is not equal to $\mathbb{Q}$ and
 for which {\it the relevant $p$-adic $L$-series possess trivial zeroes}. For details of some concrete applications of Corollary \ref{IntroCor} in this regard, see Examples \ref{RemarkExample}.

In another direction, Corollary \ref{IntroCor} also leads directly to a strong refinement of one of the main results of Greither and Popescu in \cite{GreitherPopescu} (for details of which see Corollary \ref{CMunconditional4}).
%For details see Corollary \ref{MC1}.
%In this way {\it we obtain the first theoretical result that the eTNC is proved even
%when the trivial zero occurs.}

%We also give a remark on the relation with our zeta elements and
%the $L$-values at negative integers in Remark \ref{negative}.

\vspace{-0.5mm}

\subsection{Further developments}
Finally we would like to point out that the ideas presented in this article extend naturally in at least two different directions and that we intend to discuss these developments elsewhere.

Firstly, one can formulate a natural generalization of the theory discussed here in the context of arbitrary Tate motives. In this setting our theory is related to natural generalizations of both the notion of Rubin-Stark element (which specializes in the case of Tate motives of strictly positive weight over abelian extensions of $\QQ$ to recover Soul\'e's construction of cyclotomic elements in higher algebraic $K$-theory) and of the Rubin-Stark conjecture itself. In particular, our approach leads in this context to the formulation of precise conjectural congruence relations between Rubin-Stark elements of differing `weights' which can be seen to constitute a wide-ranging (conjectural) generalization of the classical Kummer congruences involving Bernoulli numbers. For  more details in this regard see \S\ref{negative}.

Secondly, many of the constructions, conjectures and results that are discussed here extend naturally to the setting of non-commutative Iwasawa theory and can then be used to prove the same case of the eTNC that we consider here over natural families of  Galois extensions that are both non-abelian and of degree divisible by a prime $p$ at which the relevant $p$-adic $L$-series possess trivial zeroes.

\begin{notation} For the reader's convenience we start by collecting some basic notation.

For any (profinite) group $G$ we write $\widehat{G}$ for the group of homomorphisms $G \to \CC^\times$ of finite order.

Let $k$ be a number field. For a place $v$ of $k$, the residue field of $v$ is denoted by $\kappa(v)$, and its order is denoted by ${\N}v$. We denote the set of places of $k$ which lie above the infinite place $\infty$ of $\QQ$ (resp. a prime number $p$) by $S_\infty(k)$ (resp. $S_p(k)$). For a Galois extension $L/k$, the set of places of $k$ which ramify in $L$ is denoted by $S_{\rm ram}(L/k)$. For any set $\Sigma$ of places of $k$, we denote by $\Sigma_L$ the set of places of $L$ which lie above places in $\Sigma$.

Let $L/k$ be an abelian extension with Galois group $G$. For a place $v$ of $k$, the decomposition group at $v$ in $G$ is denoted by $G_v$. If $v$ is unramified in $L$, the Frobenius automorphism at $v$ is denoted by $\Fr_v$.

Let $E$ be either a field of characteristic $0$ or $\ZZ_p$. For an abelian group $A$,
we denote $E\otimes_\ZZ A$ by $EA$ or $A_E$.
For a $\ZZ_p$-module $A$ and an extension field $E$ of $\QQ_p$, we also write $EA$ or $A_E$ for $E\otimes_{\ZZ_p} A$.
(This abuse of notation would not make any confusion.)
We use similar notation for complexes. For example, if $C$ is a complex of abelian groups,
then we denote $E\otimes_\ZZ^{\mathbb{L}} C$ by $EC$ or $C_E$.

Let $R$ be a commutative ring, and $M$ be an $R$-module. Let $r$ and $s$ be non-negative integers with $r\leq s$. There is a canonical paring
$$\bigwedge_R^s M \times \bigwedge_R^r \Hom_R(M,R) \to \bigwedge_R^{s-r}M$$
defined by
$$(a_1\wedge\cdots\wedge a_s,\varphi_1\wedge\cdots\wedge \varphi_r)\mapsto
\sum_{\sigma \in {\mathfrak{S}_{s,r}} }{\rm{sgn}}(\sigma)\det(\varphi_i(a_{\sigma(j)}))_{1\leq i,j\leq r} a_{\sigma(r+1)}\wedge\cdots\wedge a_{\sigma(s)},$$
where
$$\mathfrak{S}_{s,r}:=\{ \sigma \in \mathfrak{S}_s  \mid \sigma(1) < \cdots < \sigma(r) \text{ and } \sigma(r+1) <\cdots<\sigma(s) \}.$$
(See \cite[Proposition 4.1]{bks1}.) We denote the image of $(a,\Phi)$ under the above pairing by $\Phi(a)$.

For any $R$-module $M$, we denote the linear dual $\Hom_R(M,R)$ by $M^\ast$.

The total quotient ring of $R$ is denoted by $Q(R)$.

\end{notation}

\section{Zeta elements for $\GG_m$}

In this section, we review the zeta elements for $\GG_m$ that were studied in \cite{bks1}.

\subsection{The Rubin-Stark conjecture} \label{rssection}
We review the formulation of the Rubin-Stark conjecture \cite[Conjecture B$'$]{R}.

Let $L/k$ be a finite abelian extension of number fields with Galois group $G$. Let $S$ be a finite set of places of $k$ which contains $S_\infty(k)\cup S_{\rm ram}(L/k)$. We fix a labeling $S=\{v_0,\ldots,v_n\}$. Take $r\in \ZZ$ so that $v_1,\ldots,v_r$ split completely in $L$. We put $V:=\{v_1,\ldots , v_r\}$.
For each place $v$ of $k$, we fix a place $w$ of $L$ lying above $v$. In particular, for each $i$ with $0\leq i \leq n$,
we fix a place $w_i$ of $L$ lying above $v_i$. Such conventions are frequently used in this paper.

For $\chi \in \widehat G$, let $L_{k,S}(\chi,s)$ denote the usual $S$-truncated $L$-function for $\chi$. We put
$$r_{\chi,S}:=\ord_{s=0}L_{k,S}(\chi,s).$$
Let $\cO_{L,S}$ be the ring of $S_L$ integers of $L$. For any set $\Sigma$ of places of $k$, put $Y_{L,\Sigma}:=\bigoplus_{w\in \Sigma_L}\ZZ w$, the free abelian group on $\Sigma_L$. We define
$$X_{L,\Sigma}:=\{ \sum_{w\in \Sigma_L} a_w w\in Y_{L,\Sigma} \mid \sum_{w\in \Sigma_L}a_w=0\}.$$
By Dirichlet's unit theorem, we know that the homomorphism of $\RR[G]$-modules
$$\lambda_{L,S}: \RR \cO_{L,S}^\times \stackrel{\sim}{\to} \RR X_{L,S}; \ a\mapsto -\sum_{w\in S_L}\log|a|_w w$$
is an isomorphism.

By \cite[Chap. I, Proposition 3.4]{tate} we know that
\begin{eqnarray}
r_{\chi,S}&=&\dim_\CC(e_\chi \CC \cO_{L,S}^\times)=\dim_\CC(e_\chi \CC X_{L,S}) \nonumber \\
&=& \begin{cases}
\# \{ v\in S \mid \chi(G_v)=1\} &\text{ if $\chi \neq 1$}, \\
n(=\#S-1) &\text{ if $\chi=1$},
\end{cases} \nonumber
\end{eqnarray}
where $e_\chi:=\frac{1}{\#G}\sum_{\sigma\in G}\chi(\sigma)\sigma^{-1}$. From this fact, we see that $r \leq r_{\chi,S}$.

Let $T$ be a finite set of places of $k$ which is disjoint from $S$. The $S$-truncated $T$-modified $L$-function is defined by
$$L_{k,S,T}(\chi,s):=(\prod_{v\in T}(1-\chi(\Fr_v){\N}v^{1-s}))L_{k,S}(\chi,s).$$
The $(S,T)$-unit group of $L$ is defined by
$$\cO_{L,S,T}^\times:=\ker(\cO_{L,S}^\times \to \bigoplus_{w \in T_L}\kappa(w)^\times).$$
Note that $\cO_{L,S,T}^\times$ is a subgroup of $\cO_{L,S}^\times$ of finite index. We have
$$r \leq r_{\chi,S}=\ord_{s=0}L_{k,S,T}(\chi,s)=\dim_\CC(e_\chi \CC \cO_{L,S,T}^\times).$$
We put
$$L_{k,S,T}^{(r)}(\chi,0):=\lim_{s\to 0}s^{-r}L_{k,S,T}(\chi,s).$$
We define the $r$-th order Stickelberger element by
$$\theta_{L/k,S,T}^{(r)}:=\sum_{\chi \in \widehat G}L_{k,S,T}^{(r)}(\chi^{-1},0)e_\chi \in \RR[G]. $$
The ($r$-th order) Rubin-Stark element
$$\epsilon_{L/k,S,T}^V \in \RR \bigwedge_{\ZZ[G]}^r \cO_{L,S,T}^\times$$
is defined to be the element which corresponds to
$$\theta_{L/k,S,T}^{(r)}\cdot (w_1-w_0)\wedge\cdots\wedge(w_r-w_0) \in \RR\bigwedge_{\ZZ[G]}^r X_{L,S}$$
under the isomorphism
$$\RR \bigwedge_{\ZZ[G]}^r \cO_{L,S,T}^\times \stackrel{\sim}{\to} \RR\bigwedge_{\ZZ[G]}^r X_{L,S}$$
induced by $\lambda_{L,S}$.

Now assume that $\cO_{L,S,T}^\times$ is $\ZZ$-free. Then, the Rubin-Stark conjecture predicts that the Rubin-Stark element $\epsilon_{L/k,S,T}^V$ lies in the lattice
$$\bigcap_{\ZZ[G]}^r \cO_{L,S,T}^\times:=\{a\in \QQ \bigwedge_{\ZZ[G]}^r \cO_{L,S,T}^\times \mid \Phi(a)\in\ZZ[G]\text{ for all }\Phi\in\bigwedge_{\ZZ[G]}^r\Hom_{\ZZ[G]}(\cO_{L,S,T}^\times,\ZZ[G])\}.$$
(See \cite[Conjecture B$'$]{R}.) In this paper, we consider the `$p$-part' of the Rubin-Stark conjecture for a fixed prime number $p$. We put
$$U_{L,S,T}:=\ZZ_p \cO_{L,S,T}^\times.$$
We also fix an isomorphism $\CC \simeq \CC_p$. From this, we regard
$$\epsilon_{L/k,S,T}^V \in \CC_p \bigwedge_{\ZZ_p[G]}^r U_{L,S,T}.$$
We define
$$\bigcap_{\ZZ_p[G]}^r U_{L,S,T}:=\{a \in \QQ_p\bigwedge_{\ZZ_p[G]}^r U_{L,S,T} \mid \Phi(a) \in \ZZ_p[G] \text{ for all $\Phi \in \bigwedge_{\ZZ_p[G]}^r \Hom_{\ZZ_p[G]}(U_{L,S,T},\ZZ_p[G])$}\}.$$
We easily see that there is a natural isomorphism $\ZZ_p\bigcap_{\ZZ[G]}^r \cO_{L,S,T}^\times\simeq \bigcap_{\ZZ_p[G]}^r U_{L,S,T}$. We often denote $\bigwedge_{\ZZ_p[G]}^r$ and $\bigcap_{\ZZ_p[G]}^r $ simply by $\bigwedge^r$ and $\bigcap^r $ respectively.

We propose the `$p$-component version' of the Rubin-Stark conjecture
%, which should be called the `$p$-part' of the Rubin-Stark conjecture,
as follows.

\begin{conjecture}[{${\rm RS}(L/k,S,T,V)_p$}]
$$\epsilon_{L/k,S,T}^V \in \bigcap^r U_{L,S,T}.$$
\end{conjecture}

\begin{remark}
Concerning known results on the Rubin-Stark conjecture, see \cite[Remark 5.3]{bks1} for example. Note that the Rubin-Stark conjecture is a consequence of the eTNC. This result was first proved by the first author in \cite[Corollary 4.1]{burns}, and later by the present authors \cite[Theorem 5.13]{bks1} in a simpler way.
\end{remark}

\subsection{The eTNC for the untwisted Tate motive}

In this subsection, we review the formulation of the eTNC for the untwisted Tate motive.

Let $L/k,G,S,T$ be as in the previous subsection. Fix a prime number $p$. We assume that $S_p(k) \subset S$.
Consider the complex
$$C_{L,S}:=R\Hom_{\ZZ_p}(R\Gamma_c(\cO_{L,S},\ZZ_p),\ZZ_p)[-2].$$
It is known that $C_{L,S}$ is a perfect complex of $\ZZ_p[G]$-modules, acyclic outside degrees zero and one. We have a canonical isomorphism
$$H^0(C_{L,S})\simeq U_{L,S}(:=\ZZ_p\cO_{L,S}^\times),$$
and a canonical exact sequence
$$0 \to A_S(L) \to H^1(C_{L,S}) \to \cX_{L,S}\to 0,$$
where $A_S(L):=\ZZ_p\Pic(\cO_{L,S})$ and $\cX_{L,S}:=\ZZ_pX_{L,S}$. The complex $C_{L,S}$ is identified with the $p$-completion of the complex obtained from the classical `Tate sequence' (if $S$ is large enough), and also identified with $\ZZ_p R\Gamma((\cO_{L,S})_{\mathcal{W}},\GG_m)$,
where $R\Gamma((\cO_{L,S})_{\mathcal{W}},\GG_m)$ is the `Weil-\'etale cohomology complex' constructed in \cite[\S 2.2]{bks1} (see \cite[Proposition 3.3]{BFongal} and \cite[Proposition 3.5(e)]{pnp}).

By a similar construction with \cite[Proposition 2.4]{bks1}, we construct a canonical complex $C_{L,S,T}$ which lies in the distinguished triangle
$$C_{L,S,T} \to C_{L,S} \to \bigoplus_{w\in T_L}\ZZ_p\kappa(w)^\times[0].$$
(Simply we can define $C_{L,S,T}$ by $\ZZ_p R\Gamma_T((\cO_{L,S})_{\mathcal{W}},\GG_m)$ in the terminology of \cite{bks1}.)
We have
$$H^0(C_{L,S,T})=U_{L,S,T}$$
and the exact sequence
$$0 \to A_S^T(L) \to H^1(C_{L,S,T}) \to \cX_{L,S} \to 0,$$
where $A_S^T(L)$ is the $p$-part of the ray class group of $\cO_{L,S}$ with modulus $\prod_{w\in T_L} w$.

We define the leading term of $L_{k,S,T}(\chi,s)$ at $s=0$ by
$$L_{k,S,T}^\ast(\chi,0):=\lim_{s\to 0}s^{-r_{\chi,S}}L_{k,S,T}(\chi,s).$$
The leading term at $s=0$ of the equivariant $L$-function
$$\theta_{L/k,S,T}(s):=\sum_{\chi \in \widehat G}L_{k,S,T}(\chi^{-1},s)e_\chi$$
is defined by
$$\theta_{L/k,S,T}^\ast(0):=\sum_{\chi \in \widehat G}L_{k,S,T}^\ast(\chi^{-1},0)e_\chi \in \RR[G]^\times.$$

As in the previous subsection we fix an isomorphism $\CC \simeq \CC_p$. We regard $\theta_{L/k,S,T}^\ast(0)\in \CC_p[G]^\times$. The zeta element for $\GG_m$
$$z_{L/k,S,T} \in \CC_p{\det}_{\ZZ_p[G]}(C_{L,S,T})$$
is defined to be the element which corresponds to $\theta_{L/k,S,T}^\ast(0)$ under the isomorphism
\begin{eqnarray}
\CC_p{\det}_{\ZZ_p[G]}(C_{L,S,T}) &\simeq& {\det}_{\CC_p[G]}(\CC_pU_{L,S,T}) \otimes_{\CC_p[G]} {\det}_{\CC_p[G]}^{-1}(\CC_p \cX_{L,S}) \nonumber \\
&\stackrel{\sim}{\to}& {\det}_{\CC_p[G]}(\CC_p \cX_{L,S})\otimes_{\CC_p[G]} {\det}_{\CC_p[G]}^{-1}(\CC_p \cX_{L,S}) \nonumber \\
&\stackrel{\sim}{\to}& \CC_p[G], \nonumber
\end{eqnarray}
where the second isomorphism is induced by $\lambda_{L,S}$, and the last isomorphism is  the evaluation map. Note that determinant modules must be regarded as graded invertible modules, but we omit the grading of any graded invertible modules as in \cite{bks1}.

The eTNC for the pair $(h^0(\Spec L),\ZZ_p[G])$ is formulated as follows.

\begin{conjecture}[{${\rm eTNC}(h^0(\Spec L),\ZZ_p[G])$}]
$$\ZZ_p[G]\cdot z_{L/k,S,T}={\det}_{\ZZ_p[G]}(C_{L,S,T}).$$
\end{conjecture}

\begin{remark}
When $p$ is odd, $k$ is totally real, and $L$ is CM, we say that the minus part of the eTNC (which we denote by
${\rm eTNC}(h^0(\Spec L),\ZZ_p[G]^-)$) is valid if we have the equality
$$e^-\ZZ_p[G]\cdot z_{L/k,S,T}=e^-{\det}_{\ZZ_p[G]}(C_{L,S,T}),$$
where $e^-:=\frac{1-c}{2}$ and $c \in G$ is the complex conjugation.
\end{remark}

\subsection{The eTNC and Rubin-Stark elements}

In this subsection, we interpret the eTNC, using Rubin-Stark elements. The result in this subsection will be used in \S \ref{descent section}.

We continue to use the notation in the previous subsection. Take $\chi \in \widehat G$, and suppose that $r_{\chi,S} <\# S $. Put $L_\chi:=L^{\ker \chi}$ and $G_\chi:=\Gal(L_\chi/k)$. Take $V_{\chi,S}\subset S$ so that all $v\in V_{\chi,S}$ split completely in $L_\chi$ (i.e. $\chi(G_v)=1$) and $\#V_{\chi,S}=r_{\chi,S}$. Note that, if $\chi\neq 1$, we have
$$V_{\chi,S}=\{v\in S \mid \chi(G_v)=1\}.$$
Consider the Rubin-Stark element
$$\epsilon_{L_\chi/k,S,T}^{V_{\chi,S}}\in \CC_p \bigwedge^{r_{\chi,S}}U_{L_\chi,S,T}.$$
Note that a Rubin-Stark element depends on a fixed labeling of $S$, so in this case a labeling of $S$ such that $S=\{ v_0,\ldots,v_n\}$ and $V_{\chi,S}=\{v_1,\ldots,v_{r_{\chi,S}}\}$ is understood to be chosen.

For a set $\Sigma$ of places of $k$ and a finite extension $F/k$, put $\cY_{F,\Sigma}:=\ZZ_p Y_{F,\Sigma}=\bigoplus_{w \in \Sigma_F}\ZZ_p w$ and $\cX_{F,\Sigma}:=\ZZ_pX_{F,\Sigma}=\ker(\cY_{F,\Sigma} \to \ZZ_p)$. The natural surjection
$$\cX_{L_\chi,S} \to \cY_{L_\chi,V_{\chi,S}}$$
induces an injection
$$\cY_{L_\chi,V_{\chi,S}}^\ast \to \cX_{L_\chi,S}^\ast,$$
where $(\cdot)^\ast:=\Hom_{\ZZ_p[G_\chi]}(\cdot, \ZZ_p[G_\chi])$.
Since $\cY_{L_\chi,V_{\chi,S}} \simeq \ZZ_p[G_\chi]^{\oplus r_{\chi,S}}$ and $\dim_{\CC_p}(e_\chi \CC_p\cX_{L,S})=r_{\chi,S}$, the above map induces an isomorphism
$$e_\chi\CC_p\cY_{L_\chi,V_{\chi,S}}^\ast \stackrel{\sim}{\to} e_\chi \CC_p \cX_{L,S}^\ast.$$
From this, we have a canonical identification
$$e_\chi\CC_p( \bigwedge^{r_{\chi,S}}U_{L_\chi,S,T} \otimes\bigwedge^{r_{\chi,S}} \cY_{L_\chi,V_{\chi,S}}^\ast )= e_\chi({\det}_{\CC_p[G]}(\CC_pU_{L,S,T}) \otimes_{\CC_p[G]} {\det}_{\CC_p[G]}^{-1}(\CC_p\cX_{L,S})).$$
Since $\{w_1,\ldots , w_{r_{\chi,S}}\}$ is a basis of $\cY_{L_\chi,V_{\chi,S}}$, we have the (non-canonical) isomorphism
$$\bigwedge^{r_{\chi,S}}U_{L_\chi,S,T} \stackrel{\sim}{\to} \bigwedge^{r_{\chi,S}}U_{L_\chi,S,T}\otimes \bigwedge^{r_{\chi,S}}\cY_{L_\chi,V_{\chi,S}}^\ast; \ a \mapsto a\otimes w_1^\ast \wedge \cdots \wedge w_{r_{\chi,S}}^\ast,$$
where $w_i^\ast$ is the dual of $w_i$. Hence, we have the (non-canonical) isomorphism
$$e_\chi \CC_p \bigwedge^{r_{\chi,S}}U_{L_\chi,S,T} \simeq e_\chi({\det}_{\CC_p[G]}(\CC_pU_{L,S,T}) \otimes_{\CC_p[G]} {\det}_{\CC_p[G]}^{-1}(\CC_p\cX_{L,S})).$$

\begin{proposition} \label{etncbyrs}
Suppose that $r_{\chi,S} <\# S $ for every $\chi \in \widehat G$. Then, ${\rm eTNC}(h^0(\Spec L),\ZZ_p[G])$ holds if and only if there exists a $\ZZ_p[G]$-basis $\mathcal{L}_{L/k,S,T}$ of ${\det}_{\ZZ_p[G]}(C_{L,S,T})$ such that, for every $\chi \in \widehat G$, the image of $e_\chi \mathcal{L}_{L/k,S,T}$ under the isomorphism
$$e_\chi \CC_p{\det}_{\ZZ_p[G]}(C_{L,S,T}) \simeq e_\chi({\det}_{\CC_p[G]}(\CC_pU_{L,S,T}) \otimes_{\CC_p[G]} {\det}_{\CC_p[G]}^{-1}(\CC_p\cX_{L,S})) \simeq e_\chi \CC_p \bigwedge^{r_{\chi,S}}U_{L_\chi,S,T}$$
coincides with $e_\chi \epsilon_{L_\chi/k,S,T}^{V_{\chi,S}}$.
\end{proposition}

\begin{proof}
By the definition of Rubin-Stark elements, we see that the image of $e_\chi\epsilon_{L_\chi/k,S,T}^{V_{\chi,S}}$ under the isomorphism
\begin{eqnarray}
e_\chi \CC_p \bigwedge^{r_{\chi,S}}U_{L_\chi,S,T} &\simeq & e_\chi({\det}_{\CC_p[G]}(\CC_pU_{L,S,T}) \otimes_{\CC_p[G]} {\det}_{\CC_p[G]}^{-1}(\CC_p\cX_{L,S})) \nonumber \\
&\simeq & e_\chi({\det}_{\CC_p[G]}(\CC_p\cX_{L,S}) \otimes_{\CC_p[G]} {\det}_{\CC_p[G]}^{-1}(\CC_p\cX_{L,S}))  \nonumber\\
&\simeq& e_\chi \CC_p[G] \nonumber
\end{eqnarray}
is equal to $e_\chi L_{k,S,T}^\ast(\chi^{-1},0)$. The `only if part' follows by putting $\mathcal{L}_{L/k,S,T}:=z_{L/k,S,T}$. The `if part' follows by noting that $\mathcal{L}_{L/k,S,T}$ must be equal to $z_{L/k,S,T}$.
\end{proof}

\subsection{The canonical projection maps} \label{section explicit}

Let $L/k,G,S,T,V,r$ be as in \S \ref{rssection}. We put
$$e_r:=\sum_{\chi \in \widehat G, \ r_{\chi,S}=r} e_\chi \in \QQ[G].$$
As in Proposition \ref{etncbyrs}, we construct the (non-canonical) isomorphism
$$e_r\CC_p{\det}_{\ZZ_p[G]}(C_{L,S,T}) \simeq e_r \CC_p \bigwedge^{r}U_{L,S,T}.$$

In this subsection, we give an explicit description of the map
$$\pi_{L/k,S,T}^V:{\det}_{\ZZ_p[G]}(C_{L,S,T}) \stackrel{e_r \CC_p \otimes}{\to} e_r\CC_p{\det}_{\ZZ_p[G]}(C_{L,S,T}) \simeq e_r \CC_p \bigwedge^{r}U_{L,S,T} \subset \CC_p \bigwedge^{r}U_{L,S,T}.$$
This map is important since the image of the zeta element $z_{L/k,S,T}$ under this map is the Rubin-Stark element $\epsilon_{L/k,S,T}^V$.
%which was given in \cite{bks1}, and which sends the zeta element to the
%Rubin-Stark element.
%We also give its explicit description.

Firstly, we choose a representative of $C_{L,S,T}$
$$\Pi \stackrel{\psi}{\to} \Pi,$$
where the first term is placed in degree zero, such that $\Pi$ is a free $\ZZ_p[G]$-module with basis $\{ b_1,\ldots, b_d\}$ ($d$ is sufficiently large), and that the natural surjection
$$\Pi \to H^1(C_{L,S,T}) \to \cX_{L,S}$$
sends $b_i$ to $w_i-w_0$ for each $i$ with $1\leq i \leq r$. For the details of this construction, see \cite[\S 5.4]{bks1}. Note that the representative of $R\Gamma_T((\mathcal{O}_{K,S})_{\mathcal{W}},\GG_m)$ chosen in \cite[\S 5.4]{bks1} is of the form
$$P\to F,$$
where $P$ is projective and $F$ is free. By Swan's theorem \cite[(32.1)]{curtisr}, we have an isomorphism $\ZZ_p P\simeq \ZZ_p F$. This shows that we can take the representative of $C_{L,S,T}$ as above.

We define $\psi_i \in \Hom_{\ZZ_p[G]}(\Pi,\ZZ_p[G])$ by
$$\psi_i:=b_i^\ast\circ \psi,$$
where $b_i^\ast$ is the dual of $b_i$. Note that $\bigwedge_{r<i\leq d}\psi_i \in \bigwedge^{d-r} \Hom_{\ZZ_p[G]}(\Pi,\ZZ_p[G])$ defines the homomorphism
$$\bigwedge_{r<i\leq d}\psi_i :\bigwedge^d \Pi \to \bigwedge^r \Pi$$
given by
$$(\bigwedge_{r<i\leq d}\psi_i) (b_1\wedge \cdots \wedge b_d)=\sum_{\sigma \in \mathfrak{S}_{d,r}}{\rm sgn}(\sigma) \det(\psi_i(b_{\sigma(j)}))_{r< i,j \leq d} b_{\sigma(1)} \wedge \cdots \wedge b_{\sigma(r)}$$
(see Notation.)

\begin{proposition} \label{explicit projector}\
\begin{itemize}
\item[(i)] We have
$$\bigcap^r U_{L,S,T}=(\QQ_p\bigwedge^r U_{L,S,T})\cap \bigwedge^r \Pi,$$
where we regard $U_{L,S,T}\subset \Pi$ via the natural inclusion
$$U_{L,S,T}=H^0(C_{L,S,T}) =\ker \psi \hookrightarrow \Pi.$$
%induces an injection
%$$\bigcap^r U_{L,S,T} \hookrightarrow \bigwedge^r \Pi.$$

\item[(ii)] If we regard $\bigcap^r U_{L,S,T} \subset \bigwedge^r \Pi$ by (i), then we have
$$\im (\bigwedge_{r<i\leq d}\psi_i :\bigwedge^d \Pi \to \bigwedge^r \Pi) \subset \bigcap^r U_{L,S,T} .$$

\item[(iii)] The map
$${\det}_{\ZZ_p[G]}(C_{L,S,T}) = \bigwedge^d \Pi \otimes \bigwedge^d \Pi^\ast \to \bigcap^r U_{L,S,T}; \ b_1\wedge \cdots \wedge b_d\otimes b_1^\ast \wedge \cdots \wedge b_d^\ast \mapsto (\bigwedge_{r<i\leq d}\psi_i) (b_1\wedge \cdots \wedge b_d)$$
coincides with $(-1)^{r(d-r)}\pi_{L/k,S,T}^V$. In particular, we have
$$\pi_{L/k,S,T}^V(b_1\wedge \cdots \wedge b_d\otimes b_1^\ast \wedge \cdots \wedge b_d^\ast)=(-1)^{r(d-r)}\sum_{\sigma \in \mathfrak{S}_{d,r}}{\rm sgn}(\sigma) \det(\psi_i(b_{\sigma(j)}))_{r < i,j \leq d} b_{\sigma(1)} \wedge \cdots \wedge b_{\sigma(r)}$$
and
$$\im \pi_{L/k,S,T}^V \subset \{  a\in  \bigcap^r U_{L,S,T} \mid e_r a =a\}.$$

\end{itemize}
\end{proposition}

\begin{proof}
For (i), see \cite[Lemma 4.7(ii)]{bks1}. For (ii) and (iii), see \cite[Lemma 4.3]{bks1}.
\end{proof}

\section{Higher rank Iwasawa theory}\label{hrit sec}

\subsection{Notation}
We fix a prime number $p$. We use the following notation:
\begin{itemize}
\item $k$: number field;
\item $K_\infty/k$: Galois extension such that $\G:=\Gal(K_\infty/k)\simeq \Delta \times \Gamma$, where $\Delta$ is a finite abelian group and $\Gamma\simeq \ZZ_p$;
\item $\Lambda:=\ZZ_p[[\G]]$;
\item Fix an isomorphism $\CC\simeq \CC_p$, and identify $\widehat \Delta$ with $\Hom_\ZZ(\Delta,\overline \QQ_p^\times)$. For $\chi \in \widehat \Delta$, put $\Lambda_\chi:=\ZZ_p[\im \chi][[\Gamma]].$
\end{itemize}

Note that the total quotient ring $Q(\Lambda)$ has the decomposition
$$Q(\Lambda)\simeq \bigoplus_{\chi \in \widehat \Delta/{\sim}_{\QQ_p}} Q(\Lambda_\chi),$$
where the equivalence relation $\sim_{\QQ_p}$ is defined by
$$\chi \sim_{\QQ_p} \chi' \Leftrightarrow \text{there exists $\sigma \in G_{\QQ_p}$ such that $\chi=\sigma \circ \chi'$}.$$

\begin{itemize}
\item $K:=K_\infty^\Gamma$ (so $\Gal(K/k)=\Delta$);
\item $k_\infty:=K_\infty^{\Delta}$ (so $k_\infty/k$ is a $\ZZ_p$-extension with Galois group $\Gamma$);
\item $k_n$: the $n$-th layer of $k_\infty/k$;
\item $K_n$: the $n$-th layer of $K_\infty/K$;
\item $\G_n:=\Gal(K_n/k)$.
\end{itemize}
For each character $\chi \in \widehat{\G}$ we also set
\begin{itemize}
\item $L_\chi:=K_{\infty}^{\ker \chi}$;
\item $L_{\chi,\infty}:=L_\chi \cdot k_\infty$;
\item $L_{\chi,n}$: the $n$-th layer of $L_{\chi,\infty}/L_\chi$;
\item $\G_\chi:=\Gal(L_{\chi,\infty}/k)$;
\item $\G_{\chi,n}:=\Gal(L_{\chi,n}/k)$;
\item $G_\chi:=\Gal(L_\chi/k)$;
\item $\Gamma_\chi:=\Gal(L_{\chi,\infty}/L_\chi)$;
\item $\Gamma_{\chi,n}:=\Gal(L_{\chi,n}/L_\chi)$;
\item $S$: a finite set of places of $k$ which contains $S_\infty(k)\cup S_{\rm ram}(K_\infty/k) \cup S_p(k)$;
\item $T$: a finite set of places of $k$ which is disjoint from $S$;
\item $V_\chi:=\{ v\in S \mid v \text{ splits completely in } L_{\chi,\infty} \}$ (this is a proper subset of $S$);
\item $r_\chi:=\# V_\chi.$
\end{itemize}

For any intermediate field $L$ of $K_\infty/k$, we denote $\varprojlim_{F}U_{F,S,T}$ by $U_{L,S,T}$, where $F$ runs over all intermediate field of $L/k$ which is finite over $k$ and the inverse limit is taken with respect to norm maps. Similarly, $C_{L,S,T}$ is defined to be the inverse limit of $C_{F,S,T}$. We denote $\varprojlim_F \cY_{F,S}$ by $\cY_{L,S}$, where the inverse limit is taken with respect to the maps
$$\cY_{F',S}\to \cY_{F,S}; \ w_{F'} \mapsto  w_F, $$
where $F \subset F'$, $w_{F'} \in S_{F'}$, and $w_F\in S_F$ is the place lying under $w_{F'}$. We use similar notation for $\cX_{L,S}$ etc.
%(Many objects for infinite extensions over $k$ are understood to be defined by the inverse limit with respect to norm maps.)

\subsection{Iwasawa main conjecture I} In this section we formulate
the main conjecture of Iwasawa theory for general number fields, that is a key to our study.

\subsubsection{}For any character $\chi$ in $\widehat{\G}$ there is a natural composite homomorphism
\begin{eqnarray}
\lambda_\chi :{\det}_\Lambda(C_{K_\infty,S,T}) &\to& {\det}_{\ZZ_p[G_\chi]}(C_{L_\chi,S,T}) \nonumber \\
&\hookrightarrow& {\det}_{\CC_p[G_{\chi}]}(\CC_pC_{L_{\chi},S,T}) \nonumber \\
&\stackrel{\sim}{\to}&{\det}_{\CC_p[G_{\chi}]}(\CC_p U_{L_{\chi},S,T}) \otimes_{\CC_p[G_{\chi}]}{\det}_{\CC_p[G_{\chi}]}^{-1}(\CC_p \cX_{L_{\chi},S}) \nonumber \\
& \stackrel{\sim}{\to}&{\det}_{\CC_p[G_{\chi}]}(\CC_p \cX_{L_\chi,S}) \otimes_{\CC_p[G_{\chi}]}{\det}_{\CC_p[G_{\chi}]}^{-1}(\CC_p \cX_{L_{\chi},S}) \nonumber \\
&\simeq& \CC_p[G_{\chi}] \nonumber\\
&\stackrel{\chi}{\to } &\CC_p, \nonumber
\end{eqnarray}
where the fourth map is induced by $\lambda_{L_\chi,S}$, the fifth map is the evaluation, and the last map is induced by $\chi$.

We can now state our higher rank main conjecture of Iwasawa theory in its first form.

\begin{conjecture}[{${\rm IMC}(K_\infty/k,S,T,p)$}] \label{IMC}
There exists a $\Lambda$-basis $\mathcal{L}_{K_\infty/k,S,T}$ of the module ${\det}_\Lambda(C_{K_\infty,S,T})$ for which, at every $\chi \in \widehat \Delta$ and every $\psi\in \widehat{\G_{\chi}}$ for which $r_{\psi,S} = r_\chi$ one has $\lambda_\psi(\mathcal{L}_{K_\infty/k,S,T})=L_{k,S,T}^{(r_\chi)}(\psi^{-1},0).$
\end{conjecture}
% an equality
%$$\psi_{S,T}^\ast(\mathcal{L}_{K_\infty/k,S,T})=L_{k,S,T}^{r_\chi}(\psi^{-1},0).$$

\begin{remark} It is important to note that this conjecture is much weaker than the (relevant case of the) equivariant Tamagawa number conjecture. For example, if $k_\infty/k$ is the cyclotomic $\ZZ_p$-extension, then for any $\psi$ that is trivial on the decomposition group in $\mathcal{G}_\chi$ of any $p$-adic place of $k$ one has $r_{\psi,S} > r_\chi$ and
so there is no interpolation condition at $\psi$ specified above.
When $r_{\chi}=0$, (the $\chi$-component of) the element $\mathcal{L}_{K_\infty/k,S,T}$ is
the $p$-adic $L$-function, and in the general case $r_{\chi}>0$, it plays a role of $p$-adic $L$-functions.
We show later that Conjecture \ref{IMC} can also be naturally interpreted in terms of the existence of suitable Iwasawa-theoretic measures (see Proposition \ref{measureIMC}).\end{remark}

%\begin{remark}
%The element $\mathcal{L}_{K_\infty/k,S,T}$ is unique if it exists. This follows from the fact that $\Lambda^\times=\varprojlim_n \ZZ_p[\G_n]^\times$ and that for each $n$ the map
%$$\ZZ_p[\G_n]^\times \to \prod_{\chi \in \widehat \G_n} \CC_p^\times$$
%is injective.
%\end{remark}

%\begin{remark}
%Take any $\chi \in \widehat \Delta$. If $n$ is sufficiently large, then there exists $\psi \in \widehat \G_{\chi,n}$ such that $L_{k,S,T}^{(r_\chi)}(\psi^{-1},0)\neq 0$. Indeed, consider the subgroup $\bigcap_{v \in S \setminus V_\chi, \chi(v)=1} G_v\subset G(L_{\chi,\infty}/L_\chi)$. This is an open subgroup. Hence, if $n$ is sufficiently large, then we can take a character $\psi \in \widehat \G_{\chi,n}$ which is non-trivial on $ \bigcap_{v \in S \setminus V_\chi, \chi(v)=1} G_v$. Then $\psi$ satisfies $r_{\psi,S}=r_\chi$, and hence $L_{k,S,T}^{(r_\chi)}(\psi^{-1},0)\neq 0$.
%\end{remark}

\subsubsection{}\label{negative}
 In this subsection we assume that $K_{\infty}/K$ is the cyclotomic $\ZZ_{p}$-extension and that $K$ contains a primitive
$p$-th root of unity and we briefly discuss how in this case the element $\mathcal{L}_{K_\infty/k,S,T}$ predicted by Conjecture ${\rm IMC}(K_\infty/k,S,T,p)$ should encode information about the $L$-values at $s=n$ for arbitrary integers $n$.

To do this we use the twisting map
$${\rm tw}: \Lambda \to \Lambda$$
defined by setting ${\rm tw}(\sigma):=\chi_{\rm cyc}(\sigma)\sigma$, where $\sigma \in \G$ and $\chi_{\rm cyc} : \G \to \ZZ_p^\times$ is the cyclotomic character. For an integer $n$, the ring $\Lambda$, which is regarded as a $\Lambda$-algebra via ${\rm tw}^n$, is denoted by $\Lambda(n)$. For a finite extension $L/k$ and a set of places $\Sigma$ of $k$ which contains $S_\infty(k)\cup S_{\rm ram}(L/k) \cup S_p(k)$, we put
$$C_{L,\Sigma}(n):=R\Hom_{\ZZ_p}(R\Gamma_c(\mathcal{O}_{L,\Sigma},\ZZ_p(n)),\ZZ_p)[-2].$$
For a set of places $T$ of $k$ which is disjoint from $\Sigma$, one can construct a canonical complex $C_{L,\Sigma,T}(n)$ which lies in the exact triangle
$$C_{L,\Sigma,T}(n) \to C_{L,\Sigma}(n) \to \bigoplus_{w \in T_L} H^1(\kappa(w),\ZZ_p(1-n))[0]$$
and is such that there exists a canonical isomorphism
$${\det}_\Lambda(C_{K_\infty,S,T})\otimes_{\Lambda} \Lambda(n) \simeq {\det}_\Lambda (C_{K_\infty,S,T}(n)),$$
where $C_{K_\infty,S,T}(n)$ is defined by taking the inverse limit of the complexes $C_{K_m,S,T}(n)$ (see \cite[Proposition 1.6.5(3)]{FukayaKato}).

Assuming the validity of ${\rm IMC}(K_\infty/k,S,T,p)$, we then define $\mathcal{L}_{K_\infty/k,S,T}(n)$ to be the element of ${\det}_\Lambda (C_{K_\infty,S,T}(n))$ which corresponds to $\mathcal{L}_{K_\infty/k,S,T}\otimes 1$ under the above isomorphism and we denote the image of $\mathcal{L}_{K_\infty/k,S,T}(n)$ under the canonical surjection ${\det}_\Lambda(C_{K_\infty,S,T}(n)) \to {\det}_{\ZZ_p[\Delta]}(C_{K,S,T}(n))$
by $\mathcal{L}_{K/k,S,T}(n)$.

Then the conjecture of Fukaya-Kato \cite[Conjecture 2.3.2]{FukayaKato} suggests that $\mathcal{L}_{K/k,S,T}(n)$ is the zeta element for $(h^0(\Spec K)(n),\ZZ_p[\Delta])$, namely, the element which corresponds to the leading term
$$\theta_{K/k,S,T}^\ast(n)=\sum_{\chi \in \widehat \Delta} L_{k,S,T}^\ast(\chi^{-1},n)e_\chi \in \RR[\Delta]^\times \subset \CC_p[\Delta]^\times$$
under the canonical isomorphism
$$\CC_p {\det}_{\ZZ_p[\Delta]}(C_{K,S,T}(n)) \simeq \CC_p\otimes_\QQ \Xi(h^0(\Spec K)(n)) \simeq \CC_p[\Delta],$$
where $\Xi(h^0(\Spec K)(n))$ is the fundamental line for $(h^0(\Spec K)(n),\QQ[\Delta])$ (see \cite[(29)]{BF Tamagawa}), and the first (resp. second) isomorphism is the $p$-adic regulator isomorphism $\tilde \vartheta_p$ in \cite[p.479]{BF Tamagawa2} (resp. the regulator isomorphism $\vartheta_\infty$ in \cite[p.529]{BF Tamagawa}).

In a subsequent paper we shall study these subjects thoroughly. In particular, we generalize the Rubin-Stark conjecture to the setting of Tate motives $h^0(\Spec K)(n)$ for an arbitrary integer $n$ (the original Rubin-Stark conjecture being regarded as the special case of this conjecture
in the case $n=0$).

We also show that the corresponding Rubin-Stark elements for twisted Tate motives are generalizations
of Soul\'e's cyclotomic elements and we find that the above conjectural property of $\mathcal{L}_{K_\infty/k,S,T}$ predicts
the existence of precise congruence relations between the Rubin-Stark elements for $h^0(\Spec K)(n)$ and $h^0(\Spec K)(n')$ for arbitrary integers $n$ and $n'$ which constitute a natural extension of the classical congruences of Kummer. %this congruence is nothing but Kummer's congruence, so
%We can describe this relation in an explicit way.
%(In a special case, this description is given in \cite{sanoexp}.)
%It turns out that this conjectural relation between generalized Rubin-Stark elements
%the congruence is regarded as a generalization of Kummer's congruences.
%[We shall write down the explicit map, and explain the relation between the zeta element
%to $L$-values at negative integers. This is just a special case of Fukaya-Kato.]

\subsubsection{}In the next result we record a useful invariance property of Conjecture \ref{IMC}. In the proof of this result we set
\[ \delta_T := \prod_{v \in T}(1-{\rm Fr}_v^{-1}{\rm N}v)  \in\Lambda\]
where ${\rm Fr}_v$ denotes the arithmetic Frobenius in $\mathcal{G}$ of any place $w$ of $K_\infty$ that lies above $v$. This element belongs to $Q(\Lambda)^\times$ since for each $\chi \in \widehat \Delta$ and each $v\in T$ the image $1-{\rm Fr}_v^{-1}{\N}v $ under the map $\Lambda \stackrel{\chi}{\to} \Lambda_\chi$ is non-zero.

\begin{lemma}\label{independent}  The validity of Conjecture \ref{IMC} is independent of the choice of $T$.
\end{lemma}

\begin{proof} It is enough to consider replacing $T$ by a larger set $T'$. Set $T'' := T'\setminus T$. Then, one finds that there exists an exact triangle
\[ C_{K_\infty,S,T} \to C_{K_\infty,S,T'} \to \bigoplus_{w \in T''_{K_\infty}}(\ZZ_p\kappa(w)^\times )[0]\]
($\kappa(w)^\times$ is defined by the inverse limit) and hence an equality
\[{\rm det}_\Lambda(C_{K_\infty,S,T'}) = {\det}_\Lambda(C_{K_\infty,S,T})\prod_{v \in T''_{K_\infty}}{\rm Fitt}^0_\Lambda(\ZZ_p\kappa(w)^\times) = {\det}_\Lambda(C_{K_\infty,S,T})\delta_{T''},\]
where ${\rm Fitt}^0$ denotes the (initial) Fitting ideal (see \cite{North}).

Given this, it is straightforward to check that an element $\mathcal{L}_{K_\infty/k,S,T}$ validates Conjecture \ref{IMC} with respect to $T$ if and only if the element $\delta_{T''}\cdot \mathcal{L}_{K_\infty/k,S,T}$ validates Conjecture \ref{IMC} with respect to $T'$.
\end{proof}

%\begin{remark}\label{independence 2} The equality $\Lambda\cdot \mathcal{L}_{K_\infty/k,S,T} = {\det}_\Lambda(C_{K_\infty,S,T})$ in Conjecture \ref{IMC}  specifies $\mathcal{L}_{K_\infty/k,S,T}$ up to multiplication by an element of $\Lambda^\times$. Under the assumption that no place in $S$ has a finite non-trivial decomposition group in $\mathcal{G}$, this observation combines with the predicted interpolation property to imply easily that the element $\mathcal{L}_{K_\infty/k,S,T}$ in Conjecture \ref{IMC} is unique. In Theorem \ref{lemisom}(ii) below we will see that this uniqueness property holds without any hypothesis on the decomposition groups of places in $S$. \end{remark}

Following Lemma \ref{independent} we shall assume in the sequel that $T$ contains two places of unequal residue characteristics and hence that each  group $U_{L,S,T}$ is torsion-free.

\subsubsection{}
For each $\Phi$ in $\bigwedge^{r_\chi}\Hom_{\ZZ_p[\G_{\chi,n}]}(U_{L_{\chi,n},S,T}, \ZZ_p[\G_{\chi,n}])$, Conjecture ${\rm RS}(L_{\chi,n}/k,S,T,V_\chi)_p$ implies only that $\Phi(\epsilon_{L_{\chi,n}/k,S,T}^{V_\chi})$ belongs to $\ZZ_p[\G_{\chi,n}]$.

By contrast, if Conjecture \ref{IMC} is valid, then the following result shows that the elements $\Phi(\epsilon_{L_{\chi,n}/k,S,T}^{V_\chi})$ encode significant arithmetic information.

In this result we write $\Fitt^a$ for the $a$-th Fitting ideal (see \cite{North}).

\begin{theorem}\label{imcrs}
Assume that the Iwasawa main conjecture (Conjecture \ref{IMC}) is valid for $(K_\infty/k,S,T)$. Then, for each $\chi \in \widehat \Delta$ and each positive integer $n$, we have
$$\Fitt_{\ZZ_p[\G_{\chi,n}]}^{r_\chi}(H^1(C_{L_{\chi,n},S,T}))= \{ \Phi(\epsilon_{L_{\chi,n}/k,S,T}^{V_\chi}) \mid \Phi \in \bigwedge^{r_\chi}\Hom_{\ZZ_p[\G_{\chi,n}]}(U_{L_{\chi,n},S,T}, \ZZ_p[\G_{\chi,n}])\}.$$
In particular, Conjecture ${\rm RS}(L_{\chi,n}/k,S,T,V_\chi)_p$ is valid.
\end{theorem}

\begin{proof} The explicit definition of the elements $\epsilon_{L_{\chi,n}/k,S,T}^{V_\chi}$ implies directly that the assertion of Conjecture \ref{IMC} is valid if and only if there is a $\Lambda$-basis $\mathcal{L}_{K_\infty/k,S,T}$ of ${\det}_\Lambda(C_{K_\infty,S,T})$ for which, for every character $\chi \in \widehat \Delta$ and every positive integer $n$, the image of $\mathcal{L}_{K_\infty/k,S,T}$ under the map
\begin{eqnarray}
{\det}_{\Lambda}(C_{K_\infty,S,T}) \to {\det}_{\ZZ_p[\G_{\chi,n}]}(C_{L_{\chi,n},S,T})
\stackrel{\pi_{L_{\chi,n}/k,S,T}^{V_\chi}}{\to} e_{r_\chi}\CC_p\bigwedge^{r_\chi}U_{L_{\chi,n},S,T} \nonumber
\end{eqnarray}
is equal to $\epsilon_{L_{\chi,n}/k,S,T}^{V_\chi}$.

Given this equivalence, the claimed result follows directly from Proposition \ref{explicit projector}(iii) and the same argument used to prove \cite[Theorem 7.5]{bks1}.
\end{proof}

\subsubsection{}
For each character $\chi \in \widehat \Delta$, there is a natural ring homomorphism
$$\ZZ_p[[\G_\chi]] = \ZZ_p[[G_\chi \times \Gamma]] \stackrel{\chi}{\to} \ZZ_p[\im \chi][[\Gamma]] =\Lambda_\chi \subset Q(\Lambda_\chi).$$
In the sequel we use this homomorphism to regard $Q(\Lambda_\chi)$ as a $\ZZ_p[[\G_\chi]]$-algebra.

In the next result we describe an important connection between the element $\mathcal{L}_{K_\infty/k,S,T}$ that is predicted to exist by Conjecture \ref{IMC} and the inverse limit (over $n$) of the Rubin-Stark elements $\epsilon_{L_{\chi,n}/k,S,T}^{V_\chi}$. This result shows, in particular, that  the element $\mathcal{L}_{K_\infty/k,S,T}$ in Conjecture \ref{IMC} is unique (if it exists).

In the sequel we set
\[ \bigcap^{r_\chi} U_{L_{\chi,\infty},S,T}:= \varprojlim_n \bigcap^{r_\chi} U_{L_{\chi,n},S,T},\]
where the inverse limit is taken with respect to the map
$$\bigcap^{r_\chi}U_{L_{\chi,m},S,T} \to \bigcap^{r_\chi} U_{L_{\chi,n},S,T}$$
induced by the norm map $U_{L_{\chi,m},S,T} \to U_{L_{\chi,n},S,T}$, where $n \leq m$. Note that Rubin-Stark elements are norm compatible (see \cite[Proposition 6.1]{R} or \cite[Proposition 3.5]{sano}), so if we know that Conjecture ${\rm RS}(L_{\chi,n}/k,S,T,V_\chi)_p$ is valid for all sufficiently large $n$, then we can define the element
$$\epsilon_{L_{\chi,\infty}/k,S,T}^{V_\chi}:=\varprojlim_n \epsilon_{L_{\chi,n}/k,S,T}^{V_\chi}  \in \bigcap^{r_\chi} U_{L_{\chi,\infty},S,T}.$$

\begin{theorem} \label{lemisom} \
\begin{itemize}
\item[(i)] For each $\chi \in \widehat \Delta$, the homomorphism
$${\det}_{\Lambda}(C_{K_\infty,S,T}) \to {\det}_{\ZZ_p[\G_{\chi,n}]}(C_{L_{\chi,n},S,T})
\stackrel{\pi_{L_{\chi,n}/k,S,T}^{V_\chi}}{\to} \bigcap^{r_\chi}U_{L_{\chi,n},S,T} $$
(see Proposition \ref{explicit projector}(iii)) induces an isomorphism of $Q(\Lambda_\chi)$-modules
$$\pi_{L_{\chi,\infty}/k,S,T}^{V_\chi}:{\det}_{\Lambda}(C_{K_\infty,S,T}) \otimes_\Lambda Q(\Lambda_\chi) \simeq (\bigcap^{r_\chi}U_{L_{\chi,\infty},S,T})\otimes_{\ZZ_p[[\G_\chi]]}Q(\Lambda_\chi).$$
\item[(ii)] If Conjecture \ref{IMC} is valid, then we have
$$\pi_{L_{\chi,\infty}/k,S,T}^{V_\chi}(\mathcal{L}_{K_\infty/k,S,T})=\epsilon_{L_{\chi,\infty}/k,S,T}^{V_\chi}.$$
(Note that in this case Conjecture ${\rm RS}(L_{\chi,n}/k,S,T,V_\chi)_p$ is valid for all $n$ by Theorem \ref{imcrs}.)
\end{itemize}
\end{theorem}

\begin{proof} Since the module $A_{S}^T(K_\infty)\otimes_\Lambda Q(\Lambda_\chi)$ vanishes, there are canonical isomorphisms
\begin{eqnarray}\label{first step}
& &{\det}_{\Lambda}(C_{K_\infty,S,T}) \otimes_\Lambda Q(\Lambda_\chi) \\
&\simeq& {\det}_{Q(\Lambda_\chi)}(C_{K_\infty,S,T}\otimes_\Lambda Q(\Lambda_\chi))\nonumber \\
&\simeq& {\det}_{Q(\Lambda_\chi)}(U_{K_\infty,S,T} \otimes_{\Lambda}Q(\Lambda_\chi)) \otimes_{Q(\Lambda_\chi)} {\det}_{Q(\Lambda_\chi)}^{-1}(\cX_{K_\infty,S}\otimes_\Lambda Q(\Lambda_\chi)).\nonumber
\end{eqnarray}
It is also easy to check that there are natural isomorphisms
$$U_{K_\infty,S,T}\otimes_\Lambda Q(\Lambda_\chi)\simeq U_{L_{\chi,\infty},S,T} \otimes_{\ZZ_p[[\G_\chi]]}Q(\Lambda_\chi)$$
and
$$\cX_{K_\infty,S}\otimes_\Lambda Q(\Lambda_\chi) \simeq \cX_{L_{\chi,\infty},S}\otimes_{\ZZ_p[[\G_\chi]]}Q(\Lambda_\chi) \simeq \cY_{L_{\chi,\infty},V_\chi}\otimes_{\ZZ_p[[\G_\chi]]}Q(\Lambda_\chi),$$
and that these are $Q(\Lambda_\chi)$-vector spaces of dimension $r:=r_\chi(=\# V_\chi)$. The isomorphism (\ref{first step}) is therefore a canonical isomorphism of the form
\[{\det}_{\Lambda}(C_{K_\infty,S,T}) \otimes_\Lambda Q(\Lambda_\chi) \simeq
(\bigwedge^r U_{L_{\chi,\infty},S,T}\otimes \bigwedge^r \cY^*_{L_{\chi,\infty},V_\chi})\otimes_{\ZZ_p[[\G_\chi]]} Q(\Lambda_\chi).\]
Composing this isomorphism with the map induced by the non-canonical isomorphism
$$\bigwedge^r \cY_{L_{\chi,\infty},V_\chi}^\ast \stackrel{\sim}{\to} \ZZ_p[[\G_\chi]]; w_1^\ast \wedge \cdots \wedge w_r^\ast \mapsto 1,$$
we have
$${\det}_{\Lambda}(C_{K_\infty,S,T}) \otimes_\Lambda Q(\Lambda_\chi) \simeq (\bigwedge^rU_{L_{\chi,\infty},S,T})\otimes_{\ZZ_p[[\G_\chi]] }Q(\Lambda_\chi).$$
As in the proofs of Proposition \ref{explicit projector}(iii) and of \cite[Lemma 4.3]{bks1}, this isomorphism is induced by $\varprojlim_n \pi_{L_{\chi,n}/k,S,T}^{V_\chi}$. Now the isomorphism in claim (i) is thus obtained directly from Lemma \ref{technical limit} below.

Claim (ii) follows by noting that the image of $\mathcal{L}_{K_\infty/k,S,T}$ under the map
\begin{eqnarray}
{\det}_{\Lambda}(C_{K_\infty,S,T}) \to {\det}_{\ZZ_p[\G_{\chi,n}]}(C_{L_{\chi,n},S,T})
\stackrel{\pi_{L_{\chi,n}/k,S,T}^{V_\chi}}{\to} \bigcap^{r_\chi}U_{L_{\chi,n},S,T} \nonumber
\end{eqnarray}
is equal to $\epsilon_{L_{\chi,n}/k,S,T}^{V_\chi}$ (see the proof of Theorem \ref{imcrs}).
\end{proof}

\begin{lemma}\label{technical limit} With notation as above, there is a canonical identification
\[ (\bigcap^{r} U_{L_{\chi,\infty},S,T}) \otimes_{\ZZ_p[[\G_\chi]]} Q(\Lambda_\chi) = (\bigwedge^r U_{L_{\chi,\infty},S,T})\otimes_{\ZZ_p[[\G_\chi]]} Q(\Lambda_\chi).\]\end{lemma}

\begin{proof}
Take a representative of $C_{L_{\chi,\infty},S,T}$
$$\Pi_{\infty} \to \Pi_{\infty}$$
as in \S \ref{section explicit}. Put $\Pi_n:=\Pi_{\infty}\otimes_{\ZZ_p[[\G_\chi]]} \ZZ_p[\G_{\chi,n}]$. We have
%$U_{L_{\chi,\infty},S,T} \subset \Pi_{L_{\chi,\infty}}$ and
$$\bigcap^r U_{L_{\chi,n},S,T}=(\QQ_p \bigwedge^r U_{L_{\chi,n},S,T} ) \cap \bigwedge^r \Pi_{n}$$
(see Proposition \ref{explicit projector}(i)) and so $\varprojlim_n {\bigcap}_{\ZZ_p[\G_{\chi,n}]}^r U_{L_{\chi,n},S,T}$ can be regarded as a submodule of the free $\ZZ_p[[\G_\chi]]$-module
\[ \varprojlim_n \bigwedge^r\Pi_{n}= \bigwedge^r \Pi_{\infty}.\]
For simplicity, we set
\begin{itemize}
\item $G_n:=\G_{\chi,n}$,
\item $G:=\G_\chi$,
\item $U_n:=U_{L_{\chi,n},S,T}$,
\item $U_\infty:= U_{L_{\chi,\infty},S,T}$,
%\item $\Pi_n:=\Pi_{L_{\chi,n}}$,
%\item $\Pi_\infty:=\Pi_{L_{\chi,\infty}}$,
\item $Q:=Q(\Lambda_\chi)$.
\end{itemize}
We show the equality
$$((\varprojlim_n \QQ_p \bigwedge^r U_n)\cap \bigwedge^r \Pi_\infty)\otimes_{\ZZ_p[[G]]}Q = (\bigwedge^r U_\infty)\otimes_{\ZZ_p[[G]]} Q$$
of the submodules of $(\bigwedge^r \Pi_\infty)\otimes_{\ZZ_p[[G]]} Q$.

It is easy to see that
%$$\im (\bigwedge^r U_\infty \to (\bigwedge^r U_\infty)\otimes_{\ZZ_p[[G]]}Q) \subset (\varprojlim_n \QQ_p \bigwedge^r U_n)\cap \bigwedge^r \Pi_\infty$$
%and hence
$$(\bigwedge^r U_\infty)\otimes_{\ZZ_p[[G]]} Q \subset ((\varprojlim_n \QQ_p \bigwedge^r U_n)\cap \bigwedge^r \Pi_\infty)\otimes_{\ZZ_p[[G]]}Q.$$
Conversely, take $a \in (\varprojlim_n \QQ_p \bigwedge^r U_n)\cap \bigwedge^r \Pi_\infty$ and set
$$M_n:=\coker (U_n \to \Pi_n).$$
Then we have
$$\varprojlim_n M_n \simeq \coker (U_\infty \to \Pi_\infty)=:M_\infty.$$
%(This follows from the fact that $M_n\simeq \ker(\Pi_n \to H^1(C_{L_{\chi,n},S,T}))$, $M_\infty\simeq \ker (\Pi_\infty \to H^1(C_{L_{\chi,\infty},S,T}))$, and the inverse limit is left exact.)
Since
$$\Pi_\infty \otimes_{\ZZ_p[[G]]}Q \simeq (U_\infty\otimes_{\ZZ_p[[G]]}Q)\oplus (M_\infty\otimes_{\ZZ_p[[G]]}Q),$$
we have the decomposition
$$(\bigwedge^r \Pi_\infty)\otimes_{\ZZ_p[[G]]} Q \simeq \bigoplus_{i=0}^r ( \bigwedge^{r-i}U_\infty \otimes \bigwedge^i M_\infty  )\otimes_{\ZZ_p[[G]]}Q.$$
Write
$$a=(a_i)_i \in \bigoplus_{i=0}^r ( \bigwedge^{r-i}U_\infty \otimes \bigwedge^i M_\infty  )\otimes_{\ZZ_p[[G]]}Q.$$
It is sufficient to show that $a_i=0$ for all $i>0$. We may assume that
$$a_i \in \im (  \bigwedge^{r-i}U_\infty  \otimes \bigwedge^i M_\infty  \to ( \bigwedge^{r-i}U_\infty  \otimes \bigwedge^i M_\infty )\otimes_{\ZZ_p[[G]]}Q )$$
for every $i$. Since $a \in \bigwedge^r \Pi_\infty$, we can also write
$$a=(a_{(n)})_n \in \varprojlim_n \bigwedge^r \Pi_n.$$
For each $n$, we have a decomposition
$$\QQ_p \bigwedge^r \Pi_n \simeq \bigoplus_{i=0}^r ( \QQ_p \bigwedge^{r-i}U_n \otimes_{\QQ_p[G_n]} \QQ_p \bigwedge^i M_n ),$$
and we write
$$a_{(n)}=(a_{(n),i})_i \in\bigoplus_{i=0}^r ( \QQ_p \bigwedge^{r-i}U_n \otimes_{\QQ_p[G_n]} \QQ_p \bigwedge^i M_n ).$$
Since $a \in \varprojlim_n \QQ_p \bigwedge^r U_n$, we must have $a_{(n),i}=0$ for all $i>0$. To prove $a_i=0$ for all $i>0$, It is sufficient to show that the natural map
\begin{multline} \label{injective map}
\im (  \bigwedge^{r-i}U_\infty  \otimes \bigwedge^i M_\infty  \to ( \bigwedge^{r-i}U_\infty  \otimes \bigwedge^i M_\infty  )\otimes_{\ZZ_p[[G]]}Q ) \\
\to \varprojlim_n  ( \QQ_p \bigwedge^{r-i}U_n \otimes_{\QQ_p[G_n]} \QQ_p \bigwedge^i M_n )
\end{multline}
is injective. Note that $M_\infty$ is isomorphic to a submodule of $\Pi_\infty$, since $M_\infty \simeq \ker (\Pi_\infty \to H^1(C_{L_{\chi,\infty},S,T}))$. Hence both $U_\infty$ and $M_\infty$ are embedded in $\Pi_\infty$, and we have
\begin{multline*}
\ker(  \bigwedge^{r-i}U_\infty  \otimes \bigwedge^i M_\infty  \to ( \bigwedge^{r-i}U_\infty  \otimes \bigwedge^i M_\infty  )\otimes_{\ZZ_p[[G]]}Q )\\
= \ker(\bigwedge^{r-i}U_\infty  \otimes \bigwedge^i M_\infty  \stackrel{\alpha}{\to} (\bigwedge^r(\Pi_\infty\oplus\Pi_\infty))\otimes_{\ZZ_p[[G]]} \Lambda_\chi).
\end{multline*}
Set $\Lambda_{\chi,n}:=\ZZ_p[\im \chi][\Gamma_{\chi,n}]$. The commutative diagram
$$\xymatrix{
\bigwedge^{r-i}U_\infty  \otimes \bigwedge^i M_\infty  \ar[r]^{\alpha} \ar[d]_{\beta} & (\bigwedge^r(\Pi_\infty\oplus\Pi_\infty))\otimes_{\ZZ_p[[G]]} \Lambda_\chi \ar[d]^{f} \\
\varprojlim_n \QQ_p((\bigwedge^{r-i} U_n \otimes \bigwedge^i M_n)\otimes_{\ZZ_p[G_n]}\Lambda_{\chi,n}) \ar[r]_{g}& \varprojlim_n \QQ_p((\bigwedge^r(\Pi_n\oplus\Pi_n))\otimes_{\ZZ_p[G_n]}\Lambda_{\chi,n}) }
$$
and the injectivity of $f$ and $g$ implies $\ker \alpha =\ker \beta$. Hence we have
\begin{eqnarray*}
&&\ker(  \bigwedge^{r-i}U_\infty  \otimes \bigwedge^i M_\infty  \to ( \bigwedge^{r-i}U_\infty  \otimes \bigwedge^i M_\infty  )\otimes_{\ZZ_p[[G]]}Q )\\
&=&\ker \alpha \\
&=& \ker \beta.
%\\
%&=& \ker( \bigwedge^{r-i}U_\infty  \otimes \bigwedge^i M_\infty \to \varprojlim_n  ( \QQ_p \bigwedge^{r-i}U_n \otimes_{\QQ_p[G_n]} \QQ_p \bigwedge^i M_n).
\end{eqnarray*}
This shows the injectivity of (\ref{injective map}).
\end{proof}

By Theorem \ref{lemisom}, we can formulate the following conjecture, which is equivalent to Conjecture \ref{IMC} under the assumption that Conjecture ${\rm RS}(L_{\chi,n}/k,S,T,V_\chi)_p$ is valid for all $\chi \in \widehat \Delta$ and $n$.

\begin{conjecture} \label{IMC rubinstark}
Assume that Conjecture ${\rm RS}(L_{\chi,n}/k,S,T,V_\chi)_p$ is valid for all $\chi \in \widehat \Delta$ and $n$. Define $\mathcal{L}_{K_\infty/k,S,T}\in {\det}_\Lambda(C_{K_\infty,S,T})\otimes_\Lambda Q(\Lambda)$ by
\begin{eqnarray}
\mathcal{L}_{K_\infty/k,S,T}&:=&(\pi_{L_{\chi,\infty}/k,S,T}^{V_\chi,-1}(\epsilon_{L_{\chi,\infty}/k,S,T}^{V_\chi}))_\chi \nonumber \\
& \in&  \bigoplus_{\chi \in {\widehat \Delta}/{\sim}_{\QQ_p}}({\det}_\Lambda(C_{K_\infty,S,T})\otimes_\Lambda Q(\Lambda_\chi)) \nonumber \\
&=&{\det}_\Lambda(C_{K_\infty,S,T})\otimes_\Lambda Q(\Lambda). \nonumber
\end{eqnarray}
Then, we have
$$\Lambda\cdot \mathcal{L}_{K_\infty/k,S,T} = {\det}_\Lambda(C_{K_\infty,S,T}).$$

\end{conjecture}

\subsection{Iwasawa main conjecture II} In this section we reinterpret Conjecture \ref{IMC} in terms of the existence of suitable Iwasawa-theoretic measures.
%This interpretation will then prove to be useful in later sections.

To do this we assume to be given, for each $\chi$ in $\widehat \Delta/{\sim}_{\QQ_p}$, a homomorphism of $\ZZ_p[[\mathcal{G}_\chi]]$-modules
\[ \varphi_\chi: \bigwedge^{r_\chi}\cX_{L_{\chi,\infty},S} \to \bigcap^{r_\chi} U_{L_{\chi,\infty},S,T}\]
for which $\ker(\varphi_\chi)$ is a torsion $\ZZ_p[[\mathcal{G}_\chi]]$-module.
The Rubin-Stark Conjecture implies the existence of a canonical such homomorphism $\varphi_\chi$. Indeed, if we assume Conjecture ${\rm RS}(L_{\chi,n}/k,S,T,V_\chi)_p$ for all $n$, then we can define a homomorphism
$$\bigwedge^{r_\chi}\cX_{L_{\chi,\infty},S} \to \bigwedge^{r_\chi}\cY_{L_{\chi,\infty},V_\chi} \to \bigcap^{r_\chi}U_{L_{\chi,\infty},S,T},$$
where the first map is the natural surjection, and the second is given by
$$w_1\wedge\cdots \wedge w_{r_\chi} \mapsto \epsilon_{L_{\chi,\infty}/k,S,T}^{V_\chi}.$$
Using Lemma \ref{technical limit}, one sees that the kernel of this homomorphism is torsion, and that this homomorphism is canonical.
For each character $\psi\in \widehat{\mathcal{G}_\chi}$ this homomorphism induces, upon taking coinvariance, a homomorphism of $\ZZ_p[G_\psi]$-modules
\[\varphi_{(\psi)}: \bigwedge^{r_\chi} \cX_{L_\psi,S} \to \bigcap^{r_\chi}U_{L_{\psi},S,T}. \]
Consider the endomorphism
$$e_\psi\CC_p \bigwedge^{r_\chi}\cX_{L_\psi,S} \stackrel{\varphi_{(\psi)}}{\to} e_\psi\CC_p \bigwedge^{r_\chi}U_{L_\psi,S,T} \stackrel{\lambda_{L_\psi,S}^{-1}}{\to} e_\psi\CC_p\bigwedge^{r_\chi}\cX_{L_\psi,S}.$$
We denote the determinant of this endomorphism by $\mathcal{L}_\varphi(\psi)$.
%and we use this to define a $\CC_p$-valued $\mathcal{L}$-invariant by setting
%
%\[ \mathcal{L}_\varphi(\psi):= {\rm det}_{\CC_p}(e_\psi((\CC_p\otimes_\RR{\bigwedge}^{r_\chi}_{\RR[G(L_\psi/k)]}\lambda_{L_\psi,S}^{-1})\circ(\CC_p\otimes_{\ZZ_p}\varphi_{(\psi)}))).\]

In addition, since each $\Lambda$-module $\ker(\varphi_\chi)$ is torsion, the collection $\varphi = (\varphi_\chi)_\chi$ combines with the canonical isomorphisms (\ref{first step}) and the result of Lemma \ref{technical limit} to give a composite isomorphism of $Q(\Lambda)$-modules
%
%\begin{multline} \mu_\varphi: {\det}_{\Lambda}(C_{K_\infty,S,T}) \otimes_\Lambda Q(\Lambda) \\ \simeq {\det}_{Q(\Lambda)}(U_{K_\infty,S,T} %\otimes_{\Lambda}Q(\Lambda)) \otimes_{Q(\Lambda)} {\det}_{Q(\Lambda)}^{-1}(\cX_{K_\infty,S}\otimes_\Lambda Q(\Lambda)) \simeq Q(\Lambda).
%\end{multline*}
\[ \mu_\varphi: {\det}_{\Lambda}(C_{K_\infty,S,T}) \otimes_\Lambda Q(\Lambda) \simeq Q(\Lambda).\]

%\begin{align*}&{\det}_{\Lambda}(C_{K_\infty,S,T}) \otimes_\Lambda Q(\Lambda)\\
% \simeq& {\det}_{Q(\Lambda)}(U_{K_\infty,S,T} \otimes_{\Lambda}Q(\Lambda)) \otimes_{Q(\Lambda)} %{\det}_{Q(\Lambda)}^{-1}(\cX_{K_\infty,S}\otimes_\Lambda Q(\Lambda))\\
% \simeq& {\bigotimes}_{\chi\in \widehat \Delta/{\sim}_{\QQ_p}}(({\bigcap}_{\ZZ_p[[\mathcal{G}_\chi]]}^{r_\chi} %U_{L_{\chi,\infty},S,T})\otimes_{\ZZ_p[[\mathcal{G}_\chi]]} ({\bigwedge}_{\ZZ_p[[\mathcal{G}_\chi]]}^{r_\chi} %\cX_{L_{\chi,\infty},S,T})^\ast)\otimes_{\ZZ_p[[\mathcal{G}_\chi]]}Q(\Lambda_\chi)\\
% \simeq& {\bigotimes}_{\chi\in \widehat \Delta/{\sim}_{\QQ_p}}(({\bigcap}_{\ZZ_p[[\mathcal{G}_\chi]]}^{r_\chi} %U_{L_{\chi,\infty},S,T})\otimes_{\ZZ_p[[\mathcal{G}_\chi]]} ({\bigcap}_{\ZZ_p[[\mathcal{G}_\chi]]}^{r_\chi} %U_{L_{\chi,\infty},S,T})^\ast\otimes_{\ZZ_p[[\mathcal{G}_\chi]]}Q(\Lambda_\chi)\\
% \simeq& Q(\Lambda)\end{align*}
%
%where the second isomorphism is induced by Lemma \ref{technical limit}, the third by the maps $\varphi_\chi$ and the last by the relevant evaluation %map.  to

For each $\psi$ in $\widehat{\mathcal{G}}$ we write $\mathfrak{q}_\psi$ for the kernel of the ring homomorphism $\psi: \Lambda \to \CC_p$. This is a height one prime ideal of $\Lambda$ and for any element $\lambda$ of the localization $\Lambda_{\mathfrak{q}_\psi}$ there exists a non-zero divisor $\lambda'$ of $\Lambda$ for which both $\lambda'\lambda\in \Lambda$ and $\psi(\lambda')\not= 0$. In particular, in any such case the value of the quotient $\psi(\lambda'\lambda)/\psi(\lambda')$ is independent of the choice of $\lambda'$ and will be denoted in the sequel by $\int\!\psi\,{\rm d}\lambda$.

\begin{conjecture}\label{measure conjecture}
For any collection $\varphi = (\varphi_\chi)_\chi$ as above there exists a $\Lambda$-basis $\lambda_\varphi\in Q(\Lambda)$ of $\mu_\varphi({\rm det}_{\Lambda}(C_{K^\infty,S,T}))$ such that for every $\chi\in \widehat \Delta/{\sim}_{\QQ_p}$ and every $\psi\in \widehat{\mathcal{G}_\chi}$ for which both $r_{\psi,S} = r_\chi$ and $\mathcal{L}_\varphi(\psi)\not= 0$ one has $\lambda_\varphi\in \Lambda_{\mathfrak{q}_\psi}$ and
\begin{equation}\label{integral} \int\! \psi{\rm d}\lambda_\varphi = \mathcal{L}_\varphi(\psi)\cdot L^{(r_\chi)}_{k,S,T}(\psi^{-1},0).\end{equation}
\end{conjecture}

\begin{proposition}\label{measureIMC} Conjecture \ref{IMC} is equivalent to Conjecture \ref{measure conjecture}.
\end{proposition}

\begin{proof} %It is clear that any element $\lambda_\varphi$ of $Q(\Lambda)$ which satisfies (\ref{integral}) must belong to $\Lambda_{\mathfrak{q}_\psi}$.
Fix a $\Lambda$-basis $\mathcal{L}$ of ${\rm det}_{\Lambda}(C_{K^\infty,S,T})$. Then it is enough to prove that this element satisfies the interpolation conditions of Conjecture \ref{IMC} if and only if for any choice of data $\varphi$ as above, the element $\lambda_\varphi := \mu_\varphi(\mathcal{L})$ belongs to $\Lambda_\mathfrak{q_\psi}$ and satisfies the interpolation property (\ref{integral}).

%If $\psi$ is any character in $\widehat{G_\chi}$ for which $r_{S,\psi} \not= r_\chi$, then $r_{S,\psi} > r_\chi$ so that $L^{r_\chi}_{S,T}(\psi^{-1},0)$ vanishes and the interpolation properties at $\psi$ of $z$ and $\lambda$ clearly coincide.

Set $r := r_\chi$ and $V:=V_\chi$. Then it is enough for us to fix a character $\psi\in \widehat{\G_\chi}$ for which $r_{\psi,S} = r$ and to show both that there exists a homomorphism $\varphi_\chi$ for which the map
$$e_\psi\CC_p \bigwedge^r\cX_{L_\psi,S} \stackrel{\varphi_{(\psi)}}{\to} e_\psi \CC_p \bigwedge^r U_{L_\psi,S,T}$$
is injective (and hence $\mathcal{L}_\varphi(\psi)\not= 0$) and also that for any such $\varphi_\chi$ there exists a commutative diagram of the form
\begin{equation}\label{key CD}\begin{CD} {\rm det}_{\Lambda}(C_{K^\infty,S,T}) @> \mu_\varphi >> \mu_\varphi({\rm det}_{\Lambda}(C_{K^\infty,S,T}))\\
@V \lambda_\psi VV @VV {x\mapsto \int\! \psi {\rm d}x} V \\
\CC_p @> \times \mathcal{L}_\varphi(\psi) >> \CC_p.\end{CD}\end{equation}
%

%Secondly, for any $\psi\in \widehat{G_\chi}$ for which $r_{S,\psi} = r_\chi$, there exists a homomorphism $\varphi_\chi$ for which $\mathcal{L}_\varphi(\psi)\not= 0$.

Set $\mathfrak{q}:= \mathfrak{q}_\psi$. Then, since $\mathfrak{q}$ is a height one prime ideal of $\Lambda$ which does not contain $p$ the localization $\Lambda_\mathfrak{q}$ is a discrete valuation ring. In addition, $\psi$ induces isomorphisms $\Lambda/\mathfrak{q} \simeq \ZZ_p[\im \psi]$ and $\Lambda_\mathfrak{q}/\mathfrak{q}\Lambda_\mathfrak{q} \simeq \QQ_p(\psi):=\QQ_p(\im \psi)$ and, if for each $\ZZ_p[G_\psi]$-module $M$ we set $M_\psi := e_\psi(\QQ_p(\psi)\otimes_{\ZZ_p} M)$, then we have an isomorphism of $\QQ_p(\psi)$-vector spaces
\begin{multline}\label{key iso0} H^1(C_{L_{\chi,\infty},S,T})_\mathfrak{q}/\mathfrak{q}\Lambda_\mathfrak{q} = \QQ_p\otimes_{\ZZ_p}H^1(C_{L_{\chi,\infty},S,T})/\mathfrak{q} \simeq  H^1(C_{L_\psi,S,T})_\psi \\ = (\cX_{L_\psi,S})_\psi = (\cY_{L_\psi,V})_\psi.\end{multline}
%
%\begin{multline*}\label{key iso0} H^1(C_{L_{\chi,\infty},S,T})_\mathfrak{q}/(\mathfrak{q}\cdot\Lambda_\mathfrak{q}) = H^1(C_{L_{\chi,\infty},S,T})/\mathfrak{q} \cong  %H^1(C_{L_\psi,S,T})_\psi \\ = (\cX_{L_\psi,S})_\psi = (\cY_{L_\psi,V_\chi})_\psi.\end{multline*}
%
Here the second equality follows from the exact sequence
\begin{equation}\label{loc exact} 0 \to A_{S}^T(L_{\chi,\infty}) \to H^1(C_{L_{\chi,\infty},S,T}) \stackrel{\pi}{\to} \cX_{L_{\chi, \infty},S}\to 0 \end{equation}
and the third from the assumption $r_{\psi,S} = r$.

Since $\cY_{L_{\chi,\infty},V,\mathfrak{q}}$ is a free $\Lambda_\mathfrak{q}$-module there is a direct sum decomposition $\cX_{L_{\chi, \infty},S,\mathfrak{q}} = \cX_{L_{\chi, \infty},S\setminus V,\mathfrak{q}}\oplus \cY_{L_{\chi,\infty},V,\mathfrak{q}}$. Thus, since the composite isomorphism (\ref{key iso0}) factors through the map $\pi$ in (\ref{loc exact}), the module $\cX_{L_{\chi, \infty},S\setminus V,\mathfrak{q}}/\mathfrak{q}\Lambda_\mathfrak{q}$ vanishes, and hence also (by Nakayama's Lemma) the module $\cX_{L_{\chi, \infty},S\setminus V,\mathfrak{q}}$ vanishes.

The $\Lambda_\mathfrak{q}$-module $\cX_{L_{\chi,\infty},S,\mathfrak{q}} = \cY_{L_{\chi,\infty},V,\mathfrak{q}}$ is therefore free of rank $r$ and, given this, the $\mathfrak{q}$-localisation of the tautological exact sequence
\begin{eqnarray}
0 \to U_{L_{\chi,\infty},S,T} \to \Pi_{\infty} \to \Pi_{\infty} \to H^1(C_{L_{\chi,\infty},S,T})\to 0 \nonumber
\end{eqnarray}
implies $U_{L_{\chi,\infty},S,T,\mathfrak{q}}$ is also a free $\Lambda_\mathfrak{q}$-module of rank $r$. In particular, since the $\Lambda_\mathfrak{q}$-modules $\cX_{L_{\chi,\infty},S,\mathfrak{q}}$ and $U_{L_{\chi,\infty},S,T,\mathfrak{q}}$ are isomorphic we may choose a homomorphism of $\Lambda$-modules $\varphi'_\chi: \cX_{L_{\chi,\infty},S} \to U_{L_{\chi,\infty},S,T}$ with the property that $\ker(\varphi'_\chi)_\mathfrak{q}$ vanishes. It is then easily checked that the $r$-th exterior power of $\varphi'_\chi$ induces a homomorphism of the required sort $\varphi_\chi$ for which the induced map
$$e_\psi\CC_p \bigwedge^r\cX_{L_\psi,S} \stackrel{\varphi_{(\psi)}}{\to} e_\psi \CC_p \bigwedge^r U_{L_\psi,S,T}$$
is injective.

To prove the existence of a commutative diagram (\ref{key CD}) we note first that, since the $\Lambda_\mathfrak{q}$-module $\cX_{L_{\chi,\infty},S,\mathfrak{q}}=\cY_{L_{\chi,\infty},V,\mathfrak{q}}$ is free, the exact sequence (\ref{loc exact}) splits and so the isomorphism (\ref{key iso0}) combines with Nakayama's Lemma to imply  $A_{S}^T(L_{\chi,\infty})_\mathfrak{q}$ vanishes. It follows that  $H^0(C_{L_{\chi,\infty},S,T})_\mathfrak{q} = U_{L_{\chi,\infty},S,T,\mathfrak{q}}$ and $H^1(C_{L_{\chi,\infty},S,T})_\mathfrak{q} = \cX_{L_{\chi, \infty},S,\mathfrak{q}}$ are both free $\Lambda_\mathfrak{q}$-modules of rank $r$. This gives a canonical isomorphism of $\Lambda_\mathfrak{q}$-modules
\[ {\rm det}_{\Lambda}(C_{K^\infty,S,T})_\mathfrak{q} \simeq (\bigwedge^{r}_{\Lambda_\mathfrak{q}}U_{L_{\chi,\infty},S,T,\mathfrak{q}})\otimes_{\Lambda_\mathfrak{q}} (\bigwedge^{r}_{\Lambda_\mathfrak{q}}\cX_{L_{\chi,\infty},S,\mathfrak{q}})^*.\]
and by combining this isomorphism with the natural projection map
\[ (\bigwedge^{r}_{\Lambda_\mathfrak{q}}U_{L_{\chi,\infty},S,T,\mathfrak{q}})\otimes_{\Lambda_\mathfrak{q}} (\bigwedge^{r}_{\Lambda_\mathfrak{q}}\cX_{L_{\chi,\infty},S,\mathfrak{q}})^* \to
 (\bigwedge^{r}_{\QQ_p(\psi)}(U_{L_{\psi},S,T})_\psi) \otimes_{\QQ_p(\psi)} (\bigwedge^{r}_{\QQ_p(\psi)}(\cX_{L_\psi,S})_\psi)^*\]
we obtain the horizontal arrow in the following diagram.

\begin{equation*}
\xymatrix{& & \CC_p \\
\\
{\rm det}_{\Lambda}(C_{K^\infty,S,T}) \ar[r] \ar[rruu]^{(x\mapsto \int \! \psi{\rm d}x)\circ\mu_\varphi\,\,} \ar[rrdd]_{\lambda_\psi}
& ({\bigwedge}^{r}_{\QQ_p(\psi)}(U_{L_{\psi},S,T})_\psi)\otimes_{\QQ_p(\psi)}({\bigwedge}^{r}_{\QQ_p(\psi)}(\cX_{L_\psi,S})_\psi)^* \ar[ruu]_{\theta_1} \ar[rdd]^{\theta_2}\\
\\
& & \CC_p\ar[uuuu]_{\times\mathcal{L}_\varphi(\psi)}}
\end{equation*}
Here $\theta_1$ and $\theta_2$ are the maps induced by $\varphi_{(\psi)}^{-1}$ and $\lambda_{L_\psi,S}$ and the respective evaluation maps. The commutativity of the right hand triangle is then clear and the commutativity of the two remaining triangles follows by an explicit comparison of the definitions of the maps involved. Since the whole diagram  commutes this then gives a commutative diagram of the form (\ref{key CD}), as required.\end{proof}

\subsection{Iwasawa main conjecture III}\label{IMC3}

In this subsection, we work under the following simplifying assumptions:
$$
(\ast) \begin{cases}
p \text{ is odd}, \\
\text{for every $\chi \in \widehat \Delta$, $V_\chi$ contains no finite places.}
\end{cases}
$$
We note that the second assumption here is satisfied whenever $k_\infty/k$ is the cyclotomic $\ZZ_p$-extension.

%Under these assumptions, we will show that the Rubin-Stark Conjecture leads to a canonical choice of the data $\varphi$ in Proposition \ref{measureIMC} and that this in turn leads to a much more explicit version of Conjecture \ref{IMC}.

\subsubsection{}We start by quickly reviewing some basic facts concerning the height one prime ideals of $\Lambda$.

We say that a height one prime ideal $\frp$ of $\Lambda$ is `regular' (resp. `singular') if one has $p \notin \frp$ (resp. $p \in \frp$). We will often abbreviate `height one regular (resp. singular) prime ideal' to `regular (resp. singular) prime'.

If $\frp$ is regular, then $\Lambda_\frp$ is identified with the localization of $\Lambda[1/p]$ at $\frp \Lambda[1/p]$. Since we have the decomposition
\begin{eqnarray}
\Lambda\left[\frac1p\right]=\bigoplus_{\chi \in\widehat \Delta/{\sim}_{\QQ_p}} \Lambda_\chi\left[\frac 1p\right],\label{decomposition}
\end{eqnarray}
we see that $\Lambda_\frp$ is equal to the localization of some $\Lambda_\chi[1/p]$ at $\frp \Lambda_\chi[1/p]$. This shows that $Q(\Lambda_\frp)=Q(\Lambda_\chi)$. This $\chi \in {\widehat \Delta}/{\sim}_{\QQ_p}$ is uniquely determined by $\frp$, so we denote it by $\chi_\frp$. Since $\Lambda_\chi[1/p]$ is a regular local ring, we also see that $\Lambda_\frp$ is a one-dimensional regular local ring i.e. discrete valuation ring.

Next, suppose that $\frp$ is a singular prime. We have the decomposition
$$\Lambda=\bigoplus_{\chi \in \widehat \Delta'/{\sim}_{\QQ_p}} \ZZ_p[\im \chi][\Delta_p][[\Gamma]],$$
where $\Delta_p$ is the Sylow $p$-subgroup of $\Delta$, and $\Delta'$ is the unique subgroup of $\Delta$ which is isomorphic to $\Delta/\Delta_p$. From this, we see that $\Lambda_\frp$ is identified with the localization of some $\ZZ_p[\im \chi][\Delta_p][[\Gamma]]$ at $\frp \ZZ_p[\im \chi][\Delta_p][[\Gamma]]$. By \cite[Lemma 6.2(i)]{bg}, we have
$$\frp \ZZ_p[\im \chi][\Delta_p][[\Gamma]]=(\sqrt{p\ZZ_p[\im\chi][\Delta_p]}),$$
where we denote the radical of an ideal $I$ by $\sqrt{I}$. This shows that there is a one-to-one correspondence between the set of all singular primes of $\Lambda$ and the set $\widehat \Delta'/{\sim}_{\QQ_p}$. We denote by $\chi_\frp \in \widehat \Delta'/{\sim}_{\QQ_p}$ the character corresponding to $\frp$. The next lemma shows that
$$Q(\Lambda_\frp)=\bigoplus_{\chi \in \widehat \Delta/{\sim}_{\QQ_p}, \chi \mid_{\Delta'}=\chi_\frp} Q(\Lambda_\chi). $$

\begin{lemma} \label{lemmasingular}
Let $E/\QQ_p$ be a finite unramified extension, and $\cO$ be its ring of integers. Let $P$ be a finite abelian group whose order is a power of $p$. Put $\Lambda:=\cO[P][[\Gamma]]$ and $\frp:=\sqrt{p \cO[P]}\Lambda$. ($\frp$ is the unique singular prime of $\Lambda$.) Then we have
$$Q(\Lambda_\frp)=Q(\Lambda)=\bigoplus_{\chi \in \widehat P/{\sim}_{E}} Q(\cO[\im \chi][[\Gamma]]).$$
%In particular, $Q(\Lambda_\frp) $ contains $Q(\cO[[T]])$.
\end{lemma}

\begin{proof}
%The last assertion follows from the fact that $\frp=(p, \mathfrak{q}_1)\Lambda.$
Note that $p$ is not a zero divisor of $\Lambda$, so we have
$$Q(\Lambda_\frp)=Q\left(\Lambda_\frp \left[\frac1p\right]\right).$$
We have the decomposition
$$\Lambda_\frp \left[\frac1p\right]=\bigoplus_{\chi \in \widehat P/{\sim}_{E}}e_\chi \Lambda_\frp \left[\frac1p\right],$$
where $e_\chi:=\sum_{\chi' \sim_E \chi}e_{\chi'}$.
%One can check that $e_\chi \Lambda_\frp[1/p]$ is non-zero if and only if $\mathfrak{q}_\chi \subset \frp$. Indeed, for every $a \in \Lambda \setminus \frp$, we see that $(\# \Delta e_\chi) a =(\# \Delta e_\chi) \N_{E(\im \chi)/E}(\chi(a)) \neq 0$ in $\Lambda$ if and only if $\chi(a) \neq 0$ in $\cO[\im \chi][[\Gamma]]$. Hence we have
%$$\Lambda_\frp \left[\frac1p\right]=\bigoplus_{\chi \in \widehat P/{\sim}_{E} , \mathfrak{q}_\chi \subset \frp}e_\chi \Lambda_\frp \left[\frac1p\right].$$
It is easy to see that each $e_\chi \Lambda_\frp [1/p]$ is a domain. Therefore we have
$$Q\left(\Lambda_\frp \left[\frac1p\right]\right)=\bigoplus_{\chi \in \widehat P/{\sim}_{E} }Q\left( e_\chi \Lambda_\frp \left[\frac1p\right]\right).$$
For $\chi \in \widehat P/{\sim}_E$, put $\mathfrak{q}_\chi:=\ker(\Lambda \stackrel{\chi}{\to} \cO[\im \chi][[\Gamma]])$. Note that $\sqrt{p \cO[P]}=(p, I_\cO(P))$, where $I_\cO(P)$ is the kernel of the augmentation map $\cO[P] \to \cO$. This can be shown as follows. Note that any prime ideal of $\cO/p\cO[P]$ is the kernel of some surjection $f: \cO/p\cO[P]\to R$ with some finite domain $R$. It is well-known that every finite domain is a field, so we must have $R\simeq \cO/p\cO$, and $f$ is the augmentation map $\cO/p\cO[P] \to \cO/p\cO\simeq R$. This shows that $\ker f$ is the unique prime ideal of $\cO/p\cO[P]$. Hence we have $\sqrt{p \cO[P]}=(p, I_\cO(P))$. From this, we also see that
$$\sqrt{p \cO[P]}=\ker(\cO[P] \stackrel{\chi}{\to} \cO[\im \chi] \to \cO[\im \chi]/\pi_\chi\cO[\im\chi]\simeq \cO/p\cO)$$
holds for any $\chi \in \widehat P/{\sim}_E$, where $\pi_\chi \in \cO[\im \chi]$ is a uniformizer. This shows that $\mathfrak{q}_\chi \subset \frp$. Hence, we know that $\Lambda_{\mathfrak{q}_\chi}$ is the localization of $\Lambda_{\mathfrak{p}}[1/p]$ at $\mathfrak{q}_\chi \Lambda_{\frp}[1/p]$. One can check that $\Lambda_{\mathfrak{q}_\chi}=Q(e_\chi \Lambda_\frp[1/p])$. Since we have $\Lambda_{\mathfrak{q}_\chi}=Q(\cO[\im \chi][[\Gamma]])$, the lemma follows.
\end{proof}

For a height one prime ideal $\frp$ of $\Lambda$, define a subset $\Upsilon_\frp \subset \widehat \Delta/{\sim}_{\QQ_p}$ by
$$\Upsilon_\frp:=\begin{cases}
\{ \chi_\frp \} &\text{ if $\frp$ is regular,}\\
\{ \chi \in \widehat \Delta/{\sim}_{\QQ_p} \mid \chi\mid_{\Delta'}=\chi_\frp \} &\text{ if $\frp$ is singular.}
\end{cases}
$$
The above argument shows that
$$Q(\Lambda_\frp)=\bigoplus_{\chi \in \Upsilon_\frp}Q(\Lambda_\chi).$$
%Note that, if $\frp$ is regular, then $\#\Upsilon_\frp=1$. Note also that, if $\frp$ is regular, then $\Upsilon_\frp$ is non-empty since $\chi_\frp \in \Upsilon_\frp$.

%\begin{remark}
%If $\Delta$ has no $p$-torsion, then, for every height one prime ideal $\frp$ of $\Lambda$, the ring $\Lambda_\frp$ is a discrete valuation ring and $\# \Upsilon_\frp=1$.
%
%In general, we have $\# \Upsilon_\frp >1$ for a singular prime $\frp$. To see this, suppose that $\Delta$ is cyclic of order $p$. Note that in this case we have $\# \widehat\Delta/{\sim}_{\QQ_p}=2$. Let $\tau$ be a generator of $\Delta$. The unique singular prime of $\Lambda=\ZZ_p[\Delta][[\Gamma]]$ is $\frp=\sqrt{p\ZZ_p[\Delta]}\Lambda=(p, 1-\tau) \Lambda$. We claim that $\# \Upsilon_\frp=2$. To see this, it is sufficient to show $\ker(\ZZ_p[\Delta] \stackrel{\chi}{\to} \ZZ_p[\mu_p])\Lambda \subset \frp$ for any non-trivial character $\chi \in \widehat\Delta$. Take any $\sum_{i=0}^{p-1}a_i \tau^i \in \ker(\ZZ_p[\Delta] \stackrel{\chi}{\to} \ZZ_p[\mu_p])$. Put $\zeta:=\chi(\tau)$ (this is a primitive $p$-th root of unity). Since $\sum_{i=0}^{p-1}\zeta^i=0$, we have
%$$0=\sum_{i=0}^{p-1}a_i \zeta^i= \sum_{i=1}^{p-1} (a_i-a_0) \zeta^i.$$
%This shows that $a_0=\cdots =a_{p-1}=:a$. Hence we have $\sum_{i=0}^{p-1}a_i\tau^i=a\N_\Delta$. Since we have
%$$\N_\Delta=p-(1-\tau)-(1-\tau^2)-\cdots -(1-\tau^{p-1}) \in \frp,$$
%the claim follows.
%\end{remark}

To end this section we recall a useful result concerning $\mu$-invariants.

\begin{lemma} \label{lemmamu}
Let $M$ be a finitely generated torsion $\Lambda$-module. Let $\frp$ be a singular prime of $\Lambda$. Then the following are equivalent:
\begin{itemize}
\item[(i)] The $\mu$-invariant of the $\ZZ_p[[\Gamma]]$-module $e_{\chi_\frp} M$ vanishes.
\item[(ii)] For any $\chi \in \Upsilon_\frp$, the $\mu$-invariant of the $\ZZ_p[\im \chi][[\Gamma]]$-module $M\otimes_{\ZZ_p[\Delta']} \ZZ_p[\im \chi]$ vanishes.
\item[(iii)] $M_\frp=0$.
\end{itemize}
\end{lemma}

\begin{proof}
See \cite[Lemma 5.6]{flachsurvey}.
\end{proof}

\subsubsection{}In the rest of this subsection we assume the condition $(\ast)$.

\begin{lemma}
Let $\frp$ be a singular prime of $\Lambda$. Then $V_\chi$ is independent of $\chi \in \Upsilon_\frp$. In particular, for any $\chi \in \Upsilon_\frp$, the $Q(\Lambda_\frp)$-module $U_{K_\infty,S,T}\otimes_\Lambda Q(\Lambda_\frp)$ is free of rank $r_\chi$.
\end{lemma}

\begin{proof}
It is sufficient to show that $V_\chi=V_{\chi_\frp}$ for any $\chi \in \Upsilon_\frp$. Note that the extension degree $[L_{\chi,\infty}:L_{\chi_\frp,\infty}]=[L_\chi:L_{\chi_\frp}]$ is a power of $p$. Since $p$ is odd by the assumption $(\ast)$, we see that an infinite place of $k$ which splits completely in $L_{\chi_\frp,\infty}$ also splits completely in $L_{\chi,\infty}$. By the assumption $(\ast)$, we know every places in $V_{\chi_\frp}$ is infinite. Hence we have $V_\chi=V_{\chi_\frp}$.
%We may assume that $\# \Delta$ is a power of $p$. It is sufficient to show that $V_\chi=V_1$ i.e.
%$$\{ v\in S \mid \text{$v$ splits completely in $L_{\chi,\infty}$}\}=\{ v \in S \mid \text{$v$ splits completely in $k_\infty$}\}$$
%for any $\chi\in \widehat \Delta/{\sim}_{\QQ_p}$. Note that $L_{\chi,\infty}/k$ is a pro-$p$-extension. Since $p$ is odd by assumption, we see that all infinite places split completely in $L_{\chi,\infty}$ and in $k_\infty$. By the assumption that $V_1$ contains no finite places, we must have $V_\chi=S_\infty(k)=V_1$. This shows the proposition.
\end{proof}

The above result motivates us, for any height one prime ideal $\frp$ of $\Lambda$, to define $V_\frp:=V_\chi$ and $r_\frp:=r_\chi$ by choosing some $\chi \in \Upsilon_\frp$.

Assume that Conjecture ${\rm RS}(L_{\chi,n}/k,S,T,V_\chi)_p$ holds for all $\chi \in \widehat \Delta$ and $n$. We then define the `$\frp$-part' of the Rubin-Stark element
$$\epsilon_{K_\infty/k,S,T}^\frp \in (\bigwedge^{r_\frp} U_{K_\infty,S,T})\otimes_\Lambda Q(\Lambda_\frp)$$
as the image of
$$(\epsilon_{L_{\chi,\infty}/k,S,T}^{V_\chi})_{\chi \in \Upsilon_\frp} \in \bigoplus_{\chi \in \Upsilon_\frp} \bigcap^{r_\frp }U_{L_{\chi,\infty},S,T}$$
under the natural map
$$\bigoplus_{\chi \in \Upsilon_\frp} \bigcap^{r_\frp }U_{L_{\chi,\infty},S,T} \to \bigoplus_{\chi \in \Upsilon_\frp} (\bigcap^{r_\frp }U_{L_{\chi,\infty},S,T})\otimes_{\ZZ_p[[\G_\chi]]}Q(\Lambda_\chi) = (\bigwedge^{r_\frp} U_{K_\infty,S,T})\otimes_\Lambda Q(\Lambda_\frp).$$
(see Lemma \ref{technical limit}.)

We can now formulate a much more explicit main conjecture. %In fact we conjecture that something much stronger is true.

\begin{conjecture} \label{IMCexplicit} If condition ($*$) is valid, then for every height one prime ideal $\frp$ of $\Lambda$ there is an equality
$$\Lambda_\frp\cdot\epsilon^\frp_{K_{\infty}/k,S,T}={\rm Fitt}_\Lambda^0(A_S^T(K_\infty)){\rm Fitt}_\Lambda^0(\cX_{K_\infty,S\setminus V_\frp})\cdot ({\bigwedge}_\Lambda^{r_\frp}U_{K_\infty,S,T} )_\frp.$$
\end{conjecture}

\begin{remark} \label{remarkIMC} At every height one prime ideal $\frp$ there is an equality
$${\rm Fitt}_\Lambda^0(A_S^T(K_\infty))_\frp {\rm Fitt}_\Lambda^0(\cX_{K_\infty,S\setminus V_\frp})_\frp={\rm Fitt}_\Lambda^{r_\frp}(H^1(C_{K_\infty,S,T}))_\frp.$$
If $\frp$ is regular, then $\Lambda_\frp$  is a discrete valuation ring and this equality follows directly from the exact sequence
$$0 \to A_S^T(K_\infty) \to H^1(C_{K_\infty,S,T}) \to \cX_{K_\infty,S}\to 0.$$
If $\frp$ is singular, then the equality is valid since the result of Lemma \ref{lemmamu} implies $(\cX_{K_\infty,S\setminus V_\frp})_\frp$ vanishes and so  $H^1(C_{K_\infty,S,T})_\frp$ is isomorphic to the direct sum $A_S^T(K_\infty)_\frp \oplus (\cY_{K_\infty,V_\frp})_\frp$.

Conjecture \ref{IMCexplicit} is thus valid if and only if for every height one prime $\frp$ one has
$$\Lambda_\frp\cdot\epsilon^\frp_{K_{\infty}/k,S,T}={\rm Fitt}_\Lambda^{r_\frp}(H^1(C_{K_\infty,S,T}))\cdot ({\bigwedge}_\Lambda^{r_\frp}U_{K_\infty,S,T} )_\frp .$$
\end{remark}

\begin{remark} \label{remarkIMCsingular} If the prime $\frp$ is singular, then $(\cX_{K_\infty,S\setminus V_\frp})_\frp$ vanishes and one has ${\rm Fitt}_\Lambda^0(A_S^T(K_\infty))_\frp = \Lambda_\frp$ if and only if the $\mu$-invariant of the $\ZZ_p[[\Gamma]]$-module $ e_{\chi_\frp} A_{S}^T(K_\infty)$ vanishes (see Lemma \ref{lemmamu}). Thus, for any such $\frp$ Conjecture \ref{IMCexplicit} implies that the invariant of $ e_{\chi_\frp} A_{S}^T(K_\infty)$ vanishes if and only if one has% an equality
\begin{equation}\label{easy singular} \Lambda_\frp\cdot\epsilon^\frp_{K_{\infty}/k,S,T}=({\bigwedge}_\Lambda^{r_\frp}U_{K_\infty,S,T} )_\frp.\end{equation}
In a similar way, one finds that the vanishing of the $\mu$-invariant of $e_{\chi_\frp} A_{S}^T(K_\infty)$ implies that Conjecture \ref{IMCexplicit} for $\frp$ is itself equivalent to the equality (\ref{easy singular}). \end{remark}

\begin{remark}\label{noT version} For every height one prime ideal $\frp$ of $\Lambda$, put $\epsilon_{K_\infty/k,S}^\frp:=\delta_T^{-1} \cdot \epsilon_{K_\infty/k,S,T}^\frp$. Then Lemma \ref{independent} implies that the equality of Conjecture \ref{IMCexplicit} is valid at $\frp$ if and only if one has
\begin{eqnarray*}
\Lambda_\frp \cdot \epsilon_{K_\infty/k,S}^\frp=\Fitt_\Lambda^0(A_S(K_\infty))\Fitt_\Lambda^0(\cX_{K_\infty,S\setminus S_\infty}) \cdot ({\bigwedge}_{\Lambda}^{r} U_{K_\infty,S})_\frp. \label{withoutT}
\end{eqnarray*}
\end{remark}

\subsubsection{}Before comparing Conjecture \ref{IMCexplicit} to the more general Conjecture \ref{IMC} we show that the assumed validity of the $p$-part of the Rubin-Stark conjecture already gives strong evidence in favour of Conjecture \ref{IMCexplicit}.

We note, in particular, that if $\frp$ is a singular prime of $\Lambda$ (and an appropriate $\mu$-invariant vanishes), then the inclusion proved in the following result constitutes `one half' of the equality (\ref{easy singular}) that is equivalent in this case to Conjecture \ref{IMCexplicit}.

\begin{proposition} \label{propositionfree}
Let $\frp$ be a height one prime ideal of $\Lambda$. When $\frp$ is singular, assume that the $\mu$-invariant of $e_{\chi_\frp} A_{S}^T(K_\infty)$ (as $\ZZ_p[[\Gamma]]$-module) vanishes. Then the following claims are valid. 
\begin{itemize}
\item[(i)] The $\Lambda_\frp$-module $(U_{K_\infty,S,T})_\frp$ is free of rank $r_\frp$. 
\item[(ii)] If Conjecture ${\rm RS}(L_{\chi,n}/k,S,T,V_\chi)_p$ is valid for every $\chi$ in $\widehat \Delta$ and every natural number $n$, then there is an inclusion 
$$\Lambda_\frp\cdot \epsilon^\frp_{K_{\infty}/k,S,T} \subset ({\bigwedge}_\Lambda^{r_\frp}U_{K_\infty,S,T})_\frp.$$
\end{itemize}
\end{proposition}

\begin{proof}
As in the proof of Lemma \ref{technical limit}, we choose a representative of $C_{K_\infty,S,T}$
$$\Pi_\infty \stackrel{\psi_\infty}{\to} \Pi_\infty. $$
We have the exact sequence
\begin{eqnarray}
0 \to U_{K_\infty,S,T} \to \Pi_\infty \stackrel{\psi_\infty}{\to} \Pi_\infty \to H^1(C_{K_\infty,S,T}) \to 0. \label{tate infty}
\end{eqnarray}
If $\frp$ is regular, then $\Lambda_\frp$ is a discrete valuation ring and the exact sequence (\ref{tate infty}) implies that the $\Lambda_\frp$-modules $(U_{K_\infty,S,T})_\frp$ and $\im(\psi_{\infty})_\frp$ are free. Since $U_{K_\infty,S,T}\otimes_\Lambda Q(\Lambda_\frp)$ is isomorphic to $\cY_{K_\infty,V_\frp} \otimes_\Lambda Q(\Lambda_\frp)$, we also know that the rank of $(U_{K_\infty,S,T})_\frp$ is $r_\frp$.

Suppose next that $\frp$ is singular. Since the $\mu$-invariant of $e_{\chi_\frp}\cX_{K_\infty,S\setminus V_\frp}$ vanishes, we apply Lemma \ref{lemmamu} to deduce that $(\cX_{K_\infty,S})_\frp=(\cY_{K_\infty,V_\frp})_\frp$. In a similar way, the assumption that the $\mu$-invariant of $e_{\chi_\frp} A_{S}^T(K_\infty)$ vanishes implies that $A_S^T(K_\infty)_\frp=0$. Hence we have $H^1(C_{K_\infty,S,T})_\frp=(\cY_{K_\infty,V_\frp})_\frp$.
By assumption $(\ast)$, we know that $\cY_{K_\infty,V_\frp}$ is projective as $\Lambda$-module. This implies that $H^1(C_{K_\infty,S,T})_\frp=(\cY_{K_\infty,V_\frp})_\frp$ is a free $\Lambda_\frp$-module of rank $r_\frp$. By choosing splittings of the sequence (\ref{tate infty}), we then easily deduce that the $\Lambda_\frp$-modules $(U_{K_\infty,S,T})_\frp$ and $\im(\psi_{\infty})_\frp$ are free and that the rank of $(U_{K_\infty,S,T})_\frp$ is equal to $r_\frp$.

At this stage we have proved that, for any height one prime ideal $\frp$ of $\Lambda$, the $\Lambda_\frp$-module $(U_{K_\infty,S,T})_\frp$ is both free of rank $r_\frp$ (as required to prove claim (i)) and also a direct summand of $(\Pi_\infty)_\frp$, and hence that
\begin{equation}\label{intersect}({\bigwedge}_\Lambda^{r_\frp}U_{K_\infty,S,T})_\frp=({\bigwedge}_\Lambda^{r_\frp}U_{K_\infty,S,T}\otimes_\Lambda Q(\Lambda_\frp)) \cap ({\bigwedge}_\Lambda^{r_\frp} \Pi_\infty)_\frp.\end{equation}

Now we make the stated assumption concerning the validity of the $p$-part of the Rubin-Stark conjecture. This implies, by the proof of Theorem  \ref{lemisom}(i), that for each $\frp$ the element $\epsilon^\frp_{K_\infty/k,S,T}$ lies in both $({\bigwedge}_\Lambda^{r_\frp} \Pi_\infty)_\frp$ and
$$\bigoplus_{\chi\in\Upsilon_\frp}({\bigwedge}_\Lambda^{r_\chi}U_{K_\infty,S,T} )\otimes_\Lambda Q(\Lambda_\chi)=({\bigwedge}_\Lambda^{r_\frp}U_{K_\infty,S,T} )\otimes_\Lambda Q(\Lambda_\frp),$$
and hence, by (\ref{intersect}) that it belongs to $({\bigwedge}_\Lambda^{r_\frp}U_{K_\infty,S,T})_\frp$, as required to prove claim (ii).
\end{proof}

In the next result we compare Conjecture \ref{IMCexplicit} to the more general Conjecture \ref{IMC}.

\begin{proposition}\label{equiv prop}
Assume that Conjecture ${\rm RS}(L_{\chi,n}/k,S,T,V_\chi)_p$ holds for all characters $\chi$ in $\widehat \Delta$ and all sufficiently large $n$ and that for each character $\chi$ in $\widehat \Delta'/{\sim}_{\QQ_p}$ the $\mu$-invariant of the $\ZZ_p[[\Gamma]]$-module $e_\chi A_{S}^T(K_\infty)$ vanishes. Then Conjectures \ref{IMC} and \ref{IMCexplicit} are equivalent.
\end{proposition}

\begin{proof} Since ${\det}_\Lambda(C_{K_\infty,S,T})$ is an invertible $\Lambda$-module the equality $\Lambda\cdot \mathcal{L}_{K_\infty/k,S,T}={\det}_\Lambda(C_{K_\infty,S,T})$ in Conjecture \ref{IMC} is valid if and only if at every height one prime ideal $\frp$ of $\Lambda$ one has
\begin{eqnarray}
\Lambda_\frp \cdot \mathcal{L}_{K_\infty/k,S,T}={\det}_\Lambda(C_{K_\infty,S,T})_\frp \label{localimc}
\end{eqnarray}
(see \cite[Lemma 6.1]{bg}).

If $\frp$ is regular, then one easily sees that this equality is valid if and only if the equality
$$\Lambda_\frp\cdot\epsilon^\frp_{K_{\infty}/k,S,T}={\rm Fitt}_\Lambda^{r_\frp}(H^1(C_{K_\infty,S,T}))\cdot ({\bigwedge}_\Lambda^{r_\frp}U_{K_\infty,S,T} )_\frp$$
is valid, by using Theorem \ref{lemisom}(ii).

If $\frp$ is singular, then the assumption of vanishing $\mu$-invariants and the argument in the proof of Proposition \ref{propositionfree}(i) shows that the $\Lambda_\frp$-modules $(U_{K_\infty,S,T})_\frp$ and $H^1(C_{K_\infty,S,T})_\frp$ are both free of rank $r_\frp$. Noting this, we see that (\ref{localimc}) holds if and only if one has
$$\Lambda_\frp\cdot\epsilon^\frp_{K_{\infty}/k,S,T}=({\bigwedge}_\Lambda^{r_\frp}U_{K_\infty,S,T} )_\frp$$
and so in this case the claimed result follows from Remark \ref{remarkIMCsingular}.
\end{proof}

%\begin{remark} It seems to us likely that Conjectures \ref{IMC} and \ref{IMCexplicit} should be equivalent without any hypothesis on $\mu$-invariants.\end{remark}

%$\subsection{Regular primes} In this section we use Brauer's Induction Theorem for characters to prove the following useful reduction result.
%\begin{proposition}\label{regular} To prove the equality of Conjecture \ref{IMCexplicit} for all extensions $K_\infty/k$ and for all regular primes $\frp$ of $\ZZ_p[[\Gal(K_\infty/k)]]$ it suffices to consider the case that $\Gal(K_\infty/k)\cong \Delta\times \Gamma$ with $\Delta$
%cyclic and of order prime to $p$ and that $\frp$ corresponds to a character $\chi_\frp$ that is faithful. \end{proposition}
%\begin{proof} {\em To be added..} \end{proof}

\subsubsection{}
In our earlier paper \cite{bks1} we defined canonical Selmer modules
$\mathcal{S}_{S,T}(\GG_{m/F})$ and $\mathcal{S}^{{\rm tr}}_{S,T}(\GG_{m/F})$ for $\mathbb{G}_m$ over number fields $F$ that are of finite degree over $\QQ$. For any intermediate field $L$ of $K_{\infty}/k$, we now set %$\mathcal{S}_{p,S,T}(\GG_{m/L})$ and $\mathcal{S}_{p,S,T}^{{\rm tr}}(\GG_{m/L})$ by
$$
\mathcal{S}_{p,S,T}(\GG_{m/L}) :=\varprojlim_{F} \mathcal{S}_{S,T}(\GG_{m/F}) \otimes \ZZ_{p}, \hspace{3mm}
\mathcal{S}^{{\rm tr}}_{p,S,T}(\GG_{m/L}):=\varprojlim_{F} \mathcal{S}^{{\rm tr}}_{S,T}(\GG_{m/F})\otimes \ZZ_{p}
$$
where in both limits $F$ runs over all finite extensions of $k$ in $L$ and the transition morphisms are the natural corestriction maps. 

We note in particular that, by its very definition, $\mathcal{S}^{{\rm tr}}_{p,S,T}(\GG_{m/L})$ coincides with $H^{1}(C_{L,S,T})$. In addition, this  definition implies that for any subset $V$ of $S$ comprising places that split completely in $L$ the kernel of the natural (composite) projection map
$$
\mathcal{S}^{{\rm tr}}_{p,S,T}(\GG_{m/L})_{V}:=\ker(\mathcal{S}^{{\rm tr}}_{p,S,T}(\GG_{m/L})
\to \cX_{L,S} \to \cY_{L,V})
$$
lies in a canonical exact sequence of the form 
\begin{equation}\label{canonical it ses}
0 \to A_S^T(L) \to \mathcal{S}^{{\rm tr}}_{p,S,T}(\GG_{m/L})_{V} \to
\cX_{L,S \setminus V} \to 0.\end{equation}

We now interpret our Iwasawa main conjecture in terms of characteristic ideals.

\begin{conjecture} \label{char IMC}
 Assume Conjecture ${\rm RS}(L_{\chi,n}/k,S,T,V_\chi)_p$ holds for all
$\chi \in \widehat \Delta$ and all non-negative integers $n$.
%and that for each character $\chi \in \widehat \Delta'/{\sim}_{\QQ_p}$
%the $\mu$-invariant of the $\ZZ_p[[\Gamma]]$-module $e_\chi A_{S}^T(K_\infty)$ vanishes.
Then for any $\chi \in \widehat \Delta$ there are equalities 
\begin{align}\label{IMC41}
{\rm char}_{\Lambda_{\chi}}\! (
(\bigcap^{r_\chi}U_{L_{\chi,\infty},S,T}/\langle \epsilon_{L_{\chi,\infty}/k,S,T}^{V_{\chi}} \rangle)^{\chi} )
&\!=\! {\rm char}_{\Lambda_{\chi}}\!(\mathcal{S}^{{\rm tr}}_{p,S,T}(\GG_{m/ L_{\chi,\infty}})^{\chi}_{V_{\chi}}) \\
&\!=\! {\rm char}_{\Lambda_{\chi}}\!(A_S^T(L_{\chi, \infty})^{\chi})
{\rm char}_{\Lambda_{\chi}}\!((\cX_{L_{\chi, \infty},S\setminus V_\chi})^{\chi}).\notag
\end{align}
Here, for any $\ZZ_{p}[[\G_\chi]]$-module $M$ we write $M^{\chi}$ for the $\Lambda_{\chi}$-module $M \otimes_{\ZZ_{p}[G_\chi]} \ZZ_{p}[\im \chi]$ and
${\rm char}_{\Lambda_{\chi}}(M^{\chi})$ for its characteristic
ideal in $\Lambda_{\chi}$. In addition, the second displayed equality is a direct consequence of the appropriate case of the exact sequence (\ref{canonical it ses}).
\end{conjecture}

\begin{proposition}\label{IMC4}
Assume that Conjecture ${\rm RS}(L_{\chi,n}/k,S,T,V_\chi)_p$ is valid for all characters $\chi$ in $\widehat \Delta$ and all sufficiently large natural numbers $n$ and that for each character $\chi \in \widehat \Delta'/{\sim}_{\QQ_p}$ the $\mu$-invariant of the $\ZZ_p[[\Gamma]]$-module $e_\chi A_{S}^T(K_\infty)$ vanishes.
Then Conjectures \ref{IMC} is equivalent to Conjecture \ref{char IMC}.
\end{proposition}

\begin{proof}
Note that by our assumption $\mu=0$ we have
$({\bigcap}^{r_\frp}U_{K_\infty,S,T})_\frp=
({\bigwedge}^{r_\frp}U_{K_\infty,S,T})_\frp$
for any height one prime $\frp$, using (\ref{intersect}).
Therefore, Conjecture \ref{IMCexplicit} implies the equality 
(\ref{IMC41}) for any $\chi$.

On the other hand, for a height one regular prime $\frp$,
we can regard $\frp$ to be a prime of $\Lambda_{\chi}$
for some $\chi$, so the equality (\ref{IMC41}) implies the equality
in Conjecture \ref{IMCexplicit}.
For a singular prime $\frp$,
by Lemma \ref{lemmamu},
(\ref{IMC41}) for any $\chi$ implies
$({\bigwedge}^{r_\frp}U_{K_\infty,S,T})_\frp/
\langle \epsilon^{\frp}_{K_{\infty}/k,S,T} \rangle=0$,
thus Conjecture \ref{IMCexplicit}. 

The proposition therefore follows from Proposition \ref{equiv prop}.
\end{proof}

\subsection{The case of CM-fields}\label{CM field subsec}

In this section, we use the following strengthening of the condition  ($\ast$) used above.
$$
(\ast\ast) \begin{cases}
p \text{ is odd}, \\
k \text{ is totally real and $K$ is either totally real or a CM-field},\\
k_\infty/k \text{ is the cyclotomic $\ZZ_p$-extension.}
\end{cases}
$$

Under this hypothesis Iwasawa has conjectured that for every $\chi \in \widehat \Delta'/{\sim}_{\QQ_p}$ the $\mu$-invariant of the $\ZZ_p[[\Gamma]]$-module $e_\chi A_{S}^T(K_\infty)$ vanishes and, if this is true, then Proposition \ref{equiv prop} implies that the Conjectures \ref{IMC} and \ref{IMCexplicit} are equivalent.

In addition, in this case we can use the main results of
%Greither and Popescu in \cite{GreitherPopescu}
Wiles \cite{Wiles} and of B\"uy\"ukboduk in \cite{buyuk} to give the following concrete evidence in support of these conjectures.

In the following we denote $S_\infty(k)$ and $S_p(k)$ simply by $S_\infty$ and $S_p$ respectively.

\begin{theorem}\label{CM theorem} Assume the condition ($\ast\ast$).
\begin{itemize}
\item[(i)] If $K$ is a CM-field and
the $\mu$-invariant of $K_\infty/K$ vanishes, then the minus part of Conjecture \ref{IMC} is valid for $(K_\infty/k,S,T)$.
\item[(ii)] Suppose that $\chi$ is an even character.
%Let $\frp$ be a regular height one prime of $\Lambda$ for which the character
%$\chi:=\chi_\frp \in\widehat \Delta/{\sim}_{\QQ_p}$ is even.
Then the equality of Conjecture \ref{char IMC} is valid for $\chi$ whenever all of
the following conditions are satisfied:
\begin{itemize}
\item[(a)] all $v  \in S_p$ are unramified in $L_{\chi}$,
\item[(b)] $k/\QQ$ is unramified at $p$,
\item[(c)] every $v \in S \setminus S_\infty$ satisfies $\chi(G_v)\neq 1$,
\item[(d)] the order of $\chi$ is prime to $p$,
\item[(e)] with $T$ chosen as in \cite[Remark 3.1]{buyuk}, the Rubin-Stark conjecture holds for $(F/k,S,T,S_\infty)$  for all $F$ in the set $\mathcal{K}$ of finite abelian extensions of $k$ that is defined in \cite[\S 3]{buyuk} (where our $L_\chi$ corresponds to the field $L$),
\item[(f)] the Leopoldt conjecture holds for $L_{\chi,n}$ for all positive integer $n$.
\end{itemize}
\end{itemize}
\end{theorem}

%\begin{remark} By a standard argument (using Brauer's Induction Theorem for characters) one can show that the equality of Conjecture \ref{IMCexplicit} is valid for all extensions $K_\infty/k$ and for all regular primes $\frp$ of $\ZZ_p[[\Gal(K_\infty/k)]]$ if and only if it is valid in all cases for which $\Gal(K_\infty/k)\simeq \Delta\times \Gamma$ with $\Delta$ cyclic and of order prime to $p$ and $\frp$ corresponds to a character $\chi_\frp$ that is faithful. This shows, in particular, that the condition imposed on $\chi$ in Theorem \ref{CM theorem}(ii)(d) is not restrictive.\end{remark}

\subsubsection{}We obtain Theorem \ref{CM theorem}(i) as a straightforward consequence of
the main conjecture proved by Wiles \cite{Wiles}.
In fact,
for an odd character $\chi$, one has $r_{\chi}=0$ and the Rubin-Stark elements
are Stickelberger elements.
Therefore, $\epsilon_{L_{\chi,\infty}/k,S,T}^{V_{\chi}}$ is the $p$-adic $L$-function
of Deligne-Ribet.

We shall prove the equality (\ref{IMC41}) in Conjecture \ref{char IMC}
for each odd $\chi \in \widehat \Delta$.
We fix such a character $\chi$, and
may take $K=L_{\chi}$ and $S=S_{\infty}(k) \cup S_{{\rm ram}}(K_{\infty}/k) \cup S_{p}(k)$.
Let $S'_{p}$ be the set of $p$-adic primes which split completely in $K$.
If $v \in S\setminus V_\chi$ is prime to $p$, it is ramified in $L_{\chi}=K$, so we have
${\rm char}_{\Lambda_{\chi}}(\cX_{L_{\chi, \infty},S\setminus V_\chi}^{\chi})
=
{\rm char}_{\Lambda_{\chi}}(\cY_{L_{\chi, \infty},S'_{p}}^{\chi})$.
Let $A^T(L_{\chi, \infty})$ be the inverse limit of the $p$-component of the $T$-ray class group
of the full integer ring of $L_{\chi,n}$.
By sending the prime $w$ above $v$ in $S'_{p}$ to the class of $w$, we obtain a homomorphism
$\cY_{L_{\chi, \infty},S'_{p}}^{\chi} \longrightarrow A^T(L_{\chi, \infty})^{\chi}$,
which is known to be injective.
Since the sequence
\[ \cY_{L_{\chi, \infty},S}^{\chi} \longrightarrow A^T(L_{\chi, \infty})^{\chi}
\longrightarrow A_{S}^T(L_{\chi, \infty})^{\chi} \longrightarrow 0\]
is exact and
the kernel of $\cY_{L_{\chi, \infty},S}^{\chi} \longrightarrow
\cY_{L_{\chi, \infty},S'_{p}}^{\chi}$ is finite,
we have
$$
{\rm char}_{\Lambda_{\chi}}(A_S^T(L_{\chi, \infty})^{\chi})
{\rm char}_{\Lambda_{\chi}}((\cY_{L_{\chi, \infty},S})^{\chi})
=
{\rm char}_{\Lambda_{\chi}}(A^T(L_{\chi, \infty})^{\chi}).
$$
Therefore, by noting $\chi \neq 1$, the equality (\ref{IMC41}) in Conjecture \ref{char IMC} becomes
$${\rm char}_{\Lambda_{\chi}}(A^T(L_{\chi, \infty})^{\chi})=\theta_{L_{\chi, \infty}/k,S,T}^{\chi}(0)
\Lambda_{\chi},
$$
where $\theta_{L_{\chi, \infty}/k,S,T}^{\chi}(0)$ is the $\chi$-component of
$\epsilon_{L_{\chi,\infty}/k,S,T}^{\emptyset}$, which is the Stickelberger element in this
case.
The above equality is nothing but the usual main conjecture proved by Wiles \cite{Wiles},
so we have proved (i).

\subsubsection{}We now derive Theorem \ref{CM theorem}(ii) from the main result of B\"uy\"ukboduk
in \cite{buyuk}.
To do this we assume condition ($\ast\ast$) and (without loss of generality) that
$K$ is totally real.
%We interpret the equality of Conjecture \ref{IMCexplicit} at regular primes to results of Buyukboduk \cite{buyuk}.
%Following Theorem \ref{regular} we also assume that $\Delta$ is cyclic and of order prime to $p$ and that $\chi_\frp$ is faithful.
%{\it the following remark should be in the proof of Proposition \ref{independent}: Note that $\delta_T \in Q(\Lambda)^\times$. Indeed, for each $\chi %\in \widehat \Delta$ and $v\in T$, we easily see that the image of $1-{\rm Fr}_v^{-1} {\N}v $ under the map
%$$\Lambda \stackrel{\chi}{\to} \Lambda_\chi$$
%is non-zero. }

Set $r:=[k:\QQ]=\# S_\infty$. Since $K$ is totally real,
one has $V_\chi=S_\infty$ and $r_\chi=r$.
%for every height one prime ideal $\frp$ of $\Lambda$ and
%the $Q(\Lambda)$-module $U_{K_\infty}\otimes_\Lambda Q(\Lambda)$ is free of rank $r$.
%It is easy to see $U_{K_\infty}\otimes_\Lambda Q(\Lambda) = U_{K_\infty,S}\otimes_\Lambda Q(\Lambda)$. Therefore, following Remark \ref{noT version}, the equality of Conjecture \ref{IMCexplicit} is equivalent to
%\begin{eqnarray}
%\Lambda_\frp \cdot \epsilon_{K_\infty/k,S}^\frp=\Fitt_{\Lambda}^0(A(K_\infty))\Fitt_\Lambda^0(\cX_{K_\infty,S \setminus (S_\infty \cup S_p)})\cdot ({\bigwedge}_\Lambda^r U_{K_\infty})_\frp. \label{withoutS}
%\end{eqnarray}
By our assumptions (c) and (d), the $\chi$-component of $\cX_{L_{\chi, \infty},S\setminus S_\infty}$
vanishes.
Therefore, the equality (\ref{IMC41}) becomes
$${\rm char}_{\Lambda_{\chi}} (
(\bigcap^{r}U_{L_{\chi,\infty},S,T}/\langle \epsilon_{L_{\chi,\infty}/k,S,T}^{V_{\chi}} \rangle)^{\chi} )
=
{\rm char}_{\Lambda_{\chi}}(A_S^T(L_{\chi, \infty})^{\chi}).$$

Since $K$ is totally real and $p$ is odd, we may assume that $T$ is empty.
Note that, since $L_{\chi,\infty}/L_{\chi}$ is the cyclotomic $\ZZ_p$-extension, the weak Leopoldt conjecture holds,
%for the extension $L_{\chi,\infty}/L_{\chi}$,
and we have the canonical exact sequence
\begin{eqnarray}
0 \to U_{L_{\chi,\infty}} \to U_{L_{\chi,\infty}}^{\rm sl} \to \Gal(M/L_{\chi,\infty}) \to A(L_{\chi,\infty}) \to 0, \label{semilocal}
\end{eqnarray}
where $U_{L_{\chi,\infty}}^{\rm sl}$ is the semi-local unit of $L_{\chi,\infty}$ at $p$, and $M$ is the maximal abelian $p$-extension of $L_{\chi,\infty}$ unramified outside $p$.
By our assumptions (c) and (d) again,  $U_{L_{\chi,\infty},S}^{\chi}= U_{L_{\chi,\infty}}^{\chi}$
and $A(L_{\chi,\infty})^{\chi}=A_{S}(L_{\chi,\infty})^{\chi}$.
Therefore, what we have to prove is
$${\rm char}_{\Lambda_{\chi}} (
({\bigwedge}^r U_{L_{\chi,\infty}}^{\rm sl}/\langle \epsilon_{L_{\chi,\infty}/k,S}^{V_{\chi}} \rangle)^{\chi})
={\rm char}_{\Lambda_{\chi}} (\Gal(M/L_{\chi,\infty})^{\chi}).$$
This is nothing but \cite[Theorem A]{buyuk}.
%
%
%Since $\Gal(M(K_\infty)/K_\infty)$ is a torsion $\Lambda$-module, we have an isomorphism
%$${\rm loc}_p: U_{K_\infty,S} \otimes_\Lambda Q(\Lambda) = U_{K_\infty}\otimes_\Lambda Q(\Lambda) %\stackrel{\sim}{\to}  U_{K_\infty}^{\rm sl}\otimes_\Lambda Q(\Lambda).$$
%The isomorphism
%$$({\bigwedge}_\Lambda^rU_{K_\infty,S}) \otimes_\Lambda Q(\Lambda) \stackrel{\sim}{\to}  %({\bigwedge}_\Lambda^rU_{K_\infty}^{\rm sl})\otimes_\Lambda Q(\Lambda)$$
%induced by ${\rm loc}_p$ is also denoted by ${\rm loc}_p$.
%When $\frp$ is regular, we see by (\ref{semilocal}) that (\ref{withoutS}) is equivalent to
%
%\begin{equation}\label{localIMC}
%\Lambda_\frp \cdot {\rm loc}_p(\epsilon_{K_\infty/k,S}^\frp)=\Fitt_\Lambda^0(\Gal(M(K_\infty)/K_\infty))\cdot\Fitt_\Lambda^0(\cX_{K_\infty,S \setminus (S_\infty \cup S_p)}) \cdot ({\bigwedge}_\Lambda^r U_{K_\infty}^{\rm sl})_\frp.
%\end{equation}
%
%\begin{eqnarray}
%\Lambda_\frp \cdot {\rm loc}_p(\epsilon_{K_\infty/k,S}^\frp)&=&\Fitt_\Lambda^0(G(M(K_\infty)/K_\infty)) \label{localIMC} \\
%&\cdot& \Fitt_\Lambda^0(\cX_{K_\infty,S \setminus (S_\infty \cup S_p)}) \cdot ({\bigwedge}_\Lambda^r U_{K_\infty}^{\rm sl})_\frp. \nonumber
%\end{eqnarray}
%When $\frp$ is singular, under the assumption that the $\mu$-invariant of %$e_{\chi_\frp}\Gal(M(K_\infty)/K_\infty)$ vanishes, we see that the equalities (\ref{withoutS}) and (\ref{localIMC}) are equivalent.
%
%It is thus enough to show that, under the stated hypotheses, the equality (\ref{localIMC}) follows from \cite[Theorem A]{buyuk}.
%
Note that all of the hypotheses (a)-(f) occur as assumptions in the latter result. Indeed, (a) and (b) are (A1) and (A2) in \cite{buyuk} respectively, (c) is (A3) and the assumption on $S$ in \cite{buyuk}, and (d)-(f) are assumed in his main result.
%We may thus apply \cite[Theorem A]{buyuk} to the data  $(L_{\chi,\infty}/L_\chi/k,S,T,p)$ to deduce that
%$$\Lambda_\frp \cdot {\rm loc}_p(\epsilon_{K_\infty/k,S}^\frp)=\Fitt_\Lambda^0(\Gal(M(K_\infty)/K_\infty)) \cdot ({\bigwedge}_\Lambda^r U_{K_\infty}^{\rm sl})_\frp. $$
%Comparing this with (\ref{localIMC}), we see that it is sufficient to show
%$$\Fitt_\Lambda^0(\cX_{K_\infty,S \setminus (S_\infty \cup S_p)})_\frp=\Lambda_\frp. $$
%The assumption (c) implies that $\chi$ is non-trivial and so
%$$\Fitt_\Lambda^0(\cX_{K_\infty,S\setminus(S_\infty\cup S_p)})_\frp=(\prod_{v \in S \setminus (S_\infty \cup S_p)} (1-\Fr_v))\cdot \Lambda_\frp$$
%(see \cite[Lemma 5.5]{flachsurvey}).
%We know again by the assumption (c) that $1-\chi({\rm Fr}_v) \in \QQ_p(\im \chi)^\times \subset \Lambda_\chi[1/p]^\times$ for every $v \in S \setminus (S_\infty \cup S_p)$. (If $v$ is ramified in $L_\chi$, regard $\chi({\rm Fr}_v)=0$.) Since $\Lambda_\frp$ is the localization of $ \Lambda_\chi[1/p]$, we also have $1-\chi({\rm Fr}_v) \in \Lambda_\frp^\times$. This shows that $\Fitt_\Lambda^0(\cX_{K_\infty,S\setminus(S_\infty\cup S_p)})_\frp=\Lambda_\frp$.
This completes the proof of Theorem \ref{CM theorem}(ii).

\subsection{Consequences for number fields of finite degree} In this subsection we assume the condition ($\ast\ast$) stated at the beginning of \S\ref{CM field subsec} and also that $K$ is a CM-field of finite degree over $\QQ$. We shall describe unconditional results for $K$ which follow the validity of Theorem \ref{CM theorem}(i).

To do this we set $\Lambda:=\ZZ_{p}[[\Gal(K_{\infty}/k)]]$ and for any $\Lambda$-module $M$ we denote by $M^-$ the minus part consisting of elements on which
the complex conjugation acts as $-1$ (namely, $M^-=e^-M$). We note, in particular, that $\theta_{K_{\infty}/k,S,T}(0)$ belongs to $\Lambda^-$.

We also write $x \mapsto x^{\#}$ for the $\ZZ_{p}$-linear involutions of both $\Lambda$ and the group rings $\ZZ_p[G]$ for finite quotients $G$ of $\Gal(K_{\infty}/k)$ which is induced by inverting elements of $\Gal(K_{\infty}/k)$.

\begin{corollary} \label{CMunconditional1}
If the $p$-adic $\mu$-invariant of $K_{\infty}/K$ vanishes, then one has 
$$\Fitt_{\Lambda^-}(\mathcal{S}^{{\rm tr}}_{p,S,T}(\GG_{m/K_{\infty}})^-)
= \Lambda\cdot \theta_{K_{\infty}/k,S,T}(0)$$
and
$$\Fitt_{\Lambda^-}(\mathcal{S}_{p,S,T}(\GG_{m/K_{\infty}})^-)
= \Lambda\cdot \theta_{K_{\infty}/k,S,T}(0)^{\#}.$$
\end{corollary}

\begin{proof} Since one has $r_{\chi}=0$ for any odd character $\chi$, the first displayed equality is equivalent to Conjecture \ref{IMC} in this case and is therefore valid as a consequence of Theorem \ref{CM theorem}.

The second displayed equality is then obtained directly by applying the general result of \cite[Lemma 2.8]{bks1} to the first equality.
\end{proof}

\begin{corollary} \label{CMunconditional2} Let $L$ be an intermediate CM-field of $K_\infty/k$ which is finite over $k$, and set $G:=\Gal(L/k)$. If the $p$-adic $\mu$-invariant of $K_{\infty}/K$ vanishes, then there are equalities
$$\Fitt_{\ZZ_{p}[G]^-}(\mathcal{S}^{{\rm tr}}_{p,S,T}(\GG_{m/L})^-)
= \ZZ_{p}[G]\cdot\theta_{L/k,S,T}(0)$$
and
$$\Fitt_{\ZZ_{p}[G]^-}(\mathcal{S}_{p,S,T}(\GG_{m/L})^-)
=\ZZ_{p}[G]\cdot\theta_{L/k,S,T}(0)^{\#}.$$
\end{corollary}

\begin{proof}
This follows by combining Corollary \ref{CMunconditional1} with the general result of Lemma \ref{fitt descent} below and standard properties of Fitting ideals.
\end{proof}

\begin{lemma}\label{fitt descent}
Suppose that $L/k$ is a Galois extension of finite number fields with Galois group $G$.
Then there are natural isomorphisms
$$
\mathcal{S}^{{\rm tr}}_{S,T}(\GG_{m/L})_{G} \stackrel{\sim}{\rightarrow}
\mathcal{S}^{{\rm tr}}_{S,T}(\GG_{m/k})$$
and
$$
\mathcal{S}_{S,T}(\GG_{m/L})_{G} \stackrel{\sim}{\rightarrow}
\mathcal{S}_{S,T}(\GG_{m/k}).$$
\end{lemma}

\begin{proof} The `Weil-\'etale cohomology complex' $R\Gamma_T((\mathcal{O}_{L,S})_{\mathcal{W}},\GG_m)$ is perfect and so there exist projective $\ZZ[G]$-modules $P_1$ and $P_2$, and a homomorphism of $G$-modules 
$P_1 \to P_2$ whose cokernel identifies with $\mathcal{S}^{{\rm tr}}_{S,T}(\GG_{m/L})$ and is such that the cokernel of the induced map 
$P_1^G \to P_2^G$ identifies with $\mathcal{S}^{{\rm tr}}_{S,T}(\GG_{m/k})$ (see \cite[\S 5.4]{bks1}).

The first isomorphism is then obtained by noting that the norm map induces an isomorphism of modules $(P_{2})_{G} \stackrel{\sim}{\rightarrow} P_{2}^G$. 

The second claimed isomorphism can also be obtained in a similar way, noting that $\mathcal{S}_{S,T}(\GG_{m/L})$ is obtained as the cohomology in the highest (non-zero) degree of a perfect complex (see \cite[Proposition 2.4]{bks1}).\end{proof}

We write $\mathcal{O}_L$ for the ring of integers of $L$ and
${\rm Cl}^{T}(L)$ for the ray class group of $\mathcal{O}_{L}$ with modulus $\Pi_{w \in T_{L}} w$.
We denote the Sylow $p$-subgroup of ${\rm Cl}^{T}(L)$ by $A^{T}(L)$ and write $(A^{T}(L)^-)^\vee$ for the Pontrjagin dual of the minus part of
$A^{T}(L)$. 

The next corollary of Theorem \ref{CM theorem}(i) that we record coincides with one of the main results of Greither and Popescu in \cite{GreitherPopescu}.

\begin{corollary} \label{CMunconditional3}
Let $L$ be an intermediate CM-field of $K_\infty/k$ which is finite over $k$, and set $G := \Gal(L/k)$. If the $p$-adic $\mu$-invariant for $K_{\infty}/K$ vanishes,
then one has 
$$\theta_{L/k,S,T}(0)^{\#} \in \Fitt_{\ZZ_{p}[G]^-}((A^{T}(L)^-)^\vee).$$
\end{corollary}

\begin{proof}
The canonical exact sequence
$$0 \to {\rm Cl}^{T}(L)^{\vee} \to \mathcal{S}_{S_{\infty}(k),T}(\GG_{m/L}) \to
\Hom(\mathcal{O}_{L}^{\times}, \ZZ) \to 0$$
from \cite[Proposition 2.2]{bks1} implies that the natural map 
$$\mathcal{S}_{p,S_{\infty}(k),T}(\GG_{m/L})^{-} \simeq (A^{T}(L)^-)^\vee$$
is bijective.

In addition, from \cite[Proposition 2.4(ii)]{bks1}, we know that the canonical homomorphism
$$\mathcal{S}_{S,T}(\GG_{m/L}) \rightarrow
\mathcal{S}_{S_{\infty}(k),T}(\GG_{m/L})$$
is surjective.

The stated claim therefore follows directly from the second equality in Corollary \ref{CMunconditional2}.
\end{proof}

\begin{remark}\

\noindent{}(i) Our derivation of the equality in Corollary \ref{CMunconditional3} differs from that given in \cite{GreitherPopescu} in that we avoid any use of the Galois modules related to 1-motives that are constructed in loc. cit.

\noindent{}(ii) The Brumer-Stark conjecture 
predicts $\theta_{L/k,S_{\rm ram}(L/k),T}(0)$ belongs to 
the annihilator $\Ann_{\ZZ_{p}[G]^-}(A^{T}(L))$ and if no $p$-adic place of $L^{+}$ splits in $L$, then
Corollary \ref{CMunconditional3} implies a stronger version of this conjecture.
\end{remark}

We have assumed throughout \S\ref{hrit sec} that the set $S$ contains all $p$-adic places of $k$ and so the Stickelberger element  $\theta_{L/k,S,T}(0)$ in Corollary \ref{CMunconditional3} can be imprimitive. In particular, if any $p$-adic prime of $k$ splits completely in $L$, then $\theta_{L/k,S,T}(0)$ vanishes and the assertion of Corollary \ref{CMunconditional3} is trivially valid. 

However, by applying Corollary \ref{IntroCor} in this context, we can now prove the following non-trivial result.

\begin{corollary} \label{CMunconditional4} Let $L$ be an intermediate CM-field of $K_\infty/k$ which is finite over $k$, and set $G := \Gal(L/k)$. 
If the $p$-adic $\mu$-invariant for $K_{\infty}/K$ vanishes and at most one $p$-adic place of $k$ splits in $L/L^{+}$, then one has 
$$\theta_{L/k,S_{\rm ram}(L/k),T}(0) \in \Fitt_{\ZZ_{p}[G]^-}((A^{T}(L)^-)^\vee).$$
\end{corollary}

\begin{proof} This follows immediately by combining \cite[Corollary 1.14]{bks1} with Corollary \ref{IntroCor}.
\end{proof}

\section{Iwasawa-theoretic Rubin-Stark congruences}

In this section, we formulate an Iwasawa-theoretic version of the conjecture proposed by Mazur and Rubin \cite{MRGm} and by the third author \cite{sano} (see also \cite[Conjecture 5.4]{bks1}). This conjecture is a natural generalization of the Gross-Stark conjecture \cite{Gp}, and plays a key role in the descent argument that we present in the next section.

%We assume that there exists $v_0 \in S$ such that $v_0$ does not split completely in the maximal unramified subextension in $K_\infty/k$.
%This assumption ensures that, for every nontrivial character $\chi: \G \to \overline \QQ_p^\times$ of finite order, we have $\chi(v_0)\neq 1$.

We use the notation as in the previous section.

\subsection{Statement of the congruences} \label{formulate mrs}
We first recall the formulation of the conjecture of Mazur and Rubin and of the third author.

Take a character $\chi \in \widehat \G$. Take a proper subset $V'\subset S$ so that all $v\in V'$ splits completely in $L_\chi$ (i.e. $\chi(G_v)=1$) and that $V_\chi\subset V'$. Put $r':=\#V'$. We recall the formulation of the conjecture of Mazur and Rubin and of the third author for $(L_{\chi,n}/L_\chi/k,S,T,V_\chi,V')$. For simplicity, put
\begin{itemize}
\item $L_n:=L_{\chi,n}$;
\item $L:=L_\chi$;
\item $\G_{n}:=\G_{\chi,n}=\Gal(L_{\chi,n}/k)$;
\item $G:=G_\chi=\Gal(L_\chi/k)$;
\item $\Gamma_{n}:=\Gamma_{\chi,n}=\Gal(L_{\chi,n}/L_\chi)$;
\item $V:=V_\chi=\{ v\in S \mid \text{$v$ splits completely in $L_{\chi,\infty}$}\}$;
\item $r:=r_\chi=\# V_\chi$.
\end{itemize}
Put $e:=r'-r$. Let $I(\Gamma_n)$ denote the augmentation ideal of $\ZZ_p[\Gamma_n]$.
It is shown in \cite[Lemma 2.11]{sano} that there exists a canonical injection
$$\bigcap^{r} U_{L,S,T} \hookrightarrow \bigcap^r U_{L_n,S,T}$$
which induces the injection
$$\nu_n: (\bigcap^r U_{L,S,T})\otimes_{\ZZ_p} I(\Gamma_n)^e/I(\Gamma_n)^{e+1} \hookrightarrow (\bigcap^r U_{L_n,S,T})\otimes_{\ZZ_p} \ZZ_p[\Gamma_n]/I(\Gamma_n)^{e+1}.$$
Note that this injection does not coincides with the map induced by the inclusion $U_{L,S,T}\hookrightarrow U_{L_n,S,T}$, and we have
$$\nu_n({\N_{L_n/L}^r(a)})=\N_{L_n/L} a$$
for all $a \in \bigcap^rU_{L_n,S,T}$ (see \cite[Remark 2.12]{sano}). Let $I_n$ be the kernel of the natural map $\ZZ_p[\G_n]\to \ZZ_p[G]$.
For $v\in V'\setminus V$, let ${\rm rec}_w: L^\times \to \Gamma_n$ denote the local reciprocity map at $w$ (recall that $w$ is the fixed place lying above $v$). Define
$${\rm Rec}_w:=\sum_{\sigma \in G} ({\rm rec}_w(\sigma(\cdot))-1)\sigma^{-1} \in \Hom_{\ZZ[G]}(L^\times, I_n/I_n^2).$$
It is shown in \cite[Proposition 2.7]{sano} that $\bigwedge_{v \in V'\setminus V}{\rm Rec}_w$ induces a homomorphism
$${\rm Rec}_n:\bigcap^{r'} U_{L,S,T} \to \bigcap^{r} U_{L,S,T} \otimes_{\ZZ_p} I(\Gamma_n)^e/I(\Gamma_n)^{e+1}. $$
Finally, define
$$\mathcal{N}_n : \bigcap^r U_{L_n,S,T} \to \bigcap^{r} U_{L_n,S,T}\otimes_{\ZZ_p} \ZZ_p[\Gamma_n]/I(\Gamma_n)^{e+1}$$
by
$$\mathcal{N}_n(a):=\sum_{\sigma \in \Gamma_n}\sigma a\otimes \sigma^{-1}.$$

We now state the formulation of \cite[Conjecture 3]{sano} (or \cite[Conjecture 5.2]{MRGm}).

\begin{conjecture}[{${\rm MRS}(L_n/L/k,S,T,V,V')_p$}] \label{mrs1}
Assume Conjectures ${\rm RS}(L_n/k,S,T,V)_p$ and ${\rm RS}(L/k,S,T,V')_p$. Then we have
\begin{eqnarray}
\mathcal{N}_n(\epsilon_{L_n/k,S,T}^V)=(-1)^{re} \nu_n({\rm Rec}_n(\epsilon_{L/k,S,T}^{V'})) \text{ in }\bigcap^{r} U_{L_n,S,T}\otimes_{\ZZ_p} \ZZ_p[\Gamma_n]/I(\Gamma_n)^{e+1}.\nonumber \label{mrs eq}\end{eqnarray}
(Note that the sign in the right hand side depends on the labeling of $S$. We follow the convention in \cite[\S 5.3]{bks1}. )
\end{conjecture}
%As in \cite{bks1}, we refer this conjecture as ${\rm MRS}(L_n/L/k,S,T,V,V')$.

Note that \cite[Conjecture ${\rm MRS}(K/L/k,S,T,V,V')$]{bks1} is slightly stronger than the above conjecture (see \cite[Remark 5.7]{bks1}).

We shall next give an Iwasawa theoretic version of the above conjecture.
Note that, since the inverse limit $\varprojlim_n I(\Gamma_{n})^e/I(\Gamma_{n})^{e+1}$ is isomorphic to $\ZZ_p$, the map
$$\varprojlim_n {\rm Rec}_n: \bigcap^{r'}U_{L,S,T} \to \bigcap^{r}U_{L,S,T} \otimes_{\ZZ_p}\varprojlim_n I(\Gamma_{n})^e/I(\Gamma_{n})^{e+1} $$
uniquely extends to give a $\CC_p$-linear map
$$\CC_p \bigwedge^{r'}U_{L,S,T} \to \CC_p(\bigwedge^{r}U_{L,S,T} \otimes_{\ZZ_p}\varprojlim_n I(\Gamma_{n})^e/I(\Gamma_{n})^{e+1}) $$
which we denote by ${\rm Rec}_{\infty}$.

\begin{conjecture}[{${\rm MRS}(K_\infty/k,S,T,\chi,V')$}] \label{mrs2}
Assume that Conjecture ${\rm RS}(L_n/k,S,T,V)_p$ is valid for all $n$. Then, there exists a (unique)
$$\kappa=(\kappa_n)_n \in \bigcap^r U_{L,S,T} \otimes_{\ZZ_p} \varprojlim_n I(\Gamma_n)^e/I(\Gamma_n)^{e+1}$$
such that
$$\nu_n(\kappa_n)=\mathcal{N}_n(\epsilon_{L_n/k,S,T}^V)$$
for all $n$ and that
$$e_\chi \kappa=(-1)^{re}e_\chi{\rm Rec}_\infty(\epsilon_{L/k,S,T}^{V'})\text{ in }\CC_p(\bigwedge^{r}U_{L,S,T} \otimes_{\ZZ_p}\varprojlim_n I(\Gamma_{n})^e/I(\Gamma_{n})^{e+1}).$$
\end{conjecture}

\begin{remark}
Clearly the validity of Conjecture ${\rm MRS}(L_n/L/k,S,T,V,V')_p$ for all $n$ implies the validity of ${\rm MRS}(K_\infty/k,S,T,\chi,V')$. A significant advantage of the above formulation of Conjecture ${\rm MRS}(K_\infty/k,S,T,\chi,V')$ is that we do not need to assume that Conjecture ${\rm RS}(L/k,S,T,V')_p$ is valid.
\end{remark}

\begin{proposition} \label{prop mrs}\
\begin{itemize}
\item[(i)] If $V=V'$, then ${\rm MRS}(K_\infty/k,S,T,\chi,V')$ is valid.
\item[(ii)] If $V \subset V'' \subset V'$, then ${\rm MRS}(K_\infty/k,S,T,\chi,V')$ implies ${\rm MRS}(K_\infty/k,S,T,\chi,V'')$.
\item[(iii)] Suppose that $\chi(G_v)=1$ for all $v\in S$ and $\# V'=\# S-1$. Then, for any $V''\subset S$ with $V \subset V''$ and $\# V''=\#S -1$, ${\rm MRS}(K_\infty/k,S,T,\chi,V')$ and ${\rm MRS}(K_\infty/k,S,T,\chi,V'')$ are equivalent.
\item[(iv)] If $v \in V'\setminus V$ is a finite place which is unramified in $L_\infty$, then ${\rm MRS}(K_\infty/k,S\setminus \{v\},T,\chi,V'\setminus \{v\})$ implies ${\rm MRS}(K_\infty/k,S,T,\chi,V')$.
\item[(v)] If $\#V'\neq \# S-1$ and $v\in S \setminus V'$ is a finite place which is unramified in $L_\infty$, then ${\rm MRS}(K_\infty/k,S\setminus\{v\},T,\chi,V')$ implies ${\rm MRS}(K_\infty/k,S,T,\chi,V')$.
\end{itemize}
\end{proposition}

\begin{proof}
Claim (i) follows from the `norm relation' of Rubin-Stark elements, see \cite[Remark 3.9]{sano} or \cite[Proposition 5.7]{MRGm}. Claim (ii) follows from \cite[Proposition 3.12]{sano}. Claim (iii) follows from \cite[Lemma 5.1]{sanotjm}. Claim (iv) follows from the proof of \cite[Proposition 3.13]{sano}. Claim (v) follows by noting that
$$\epsilon_{L_n/k,S,T}^V=(1-{\rm Fr}_v^{-1})\epsilon_{L_n/k,S\setminus\{v\},T}^V$$
and
$$\epsilon_{L/k,S,T}^{V'}=(1-{\rm Fr}_v^{-1})\epsilon_{L/k,S\setminus\{v\},T}^{V'}.$$
\end{proof}

\begin{corollary} \label{cor unram}
If every place $v$ in $V'\setminus V$ is both non-archimedean and unramified in $L_\infty$, then ${\rm MRS}(K_\infty/k,S,T,\chi,V')$ is valid.
\end{corollary}

\begin{proof}
By Proposition \ref{prop mrs}(iv), we may assume $V=V'$. By Proposition \ref{prop mrs}(i), ${\rm MRS}(K_\infty/k,S,T,\chi,V')$ is valid in this case.
\end{proof}

Consider the following condition:
$${\rm NTZ}(K_\infty/k,\chi) \quad \text{$\chi(G_\mathfrak{p})\neq 1$ for all $\mathfrak{p} \in S_p(k)$ which ramify in $L_{\chi,\infty}$}.$$
This condition is usually called `no trivial zeros'.

\begin{corollary} \label{nontrivzeros}
Assume that $\chi$ satisfies ${\rm NTZ}(K_\infty/k,\chi)$. Then ${\rm MRS}(K_\infty/k,S,T,\chi,V')$ is valid.
\end{corollary}

\begin{proof}
In this case we see that every $v \in V'\setminus V$ is finite and unramified in $L_{\infty}$.
\end{proof}

\subsection{Connection to the Gross-Stark conjecture} \label{GS} In this subsection we help set the context for Conjecture ${\rm MRS}(K_\infty/k,S,T,\chi,V')$ by showing that it specializes to recover the Gross-Stark Conjecture (as stated in Conjecture \ref{gross stark conj} below).

To do this we assume throughout that $k$ is totally real, $k_\infty/k$ is the cyclotomic $\ZZ_p$-extension and $\chi$ is totally odd. We also set $V':=\{v \in S \mid \chi(G_v)=1\}$ (and note that this is a proper subset of $S$ since $\chi$ is totally odd) and we assume that every $v\in V'$ lies above $p$ (noting that this assumption is not restrictive as a consequence of Proposition \ref{prop mrs}(iv)).

We shall now show that this case of ${\rm MRS}(K_\infty/k,S,T,\chi,V')$ is equivalent to the Gross-Stark conjecture.

As a first step, we note that in this case $V$ is empty (that is, $r=0$) and so one knows that Conjecture ${\rm RS}(L_{n}/k,S,T,V)_p$ is valid for all $n$ (by \cite[Theorem 3.3]{R}). In fact, one has  $\epsilon_{L_{n}/k,S,T}^{V}=\theta_{L_{n}/k,S,T}(0) \in \ZZ_p[\G_n]$ and, by \cite[Proposition 5.4]{MRGm}, the assertion of Conjecture ${\rm MRS}(K_\infty/k,S,T,\chi,V')$ is equivalent to the following claims: one has
\begin{eqnarray}
\theta_{L_n/k,S,T}(0) \in I_n^{r'} \label{vanishing order stickelberger}
\end{eqnarray}
for all $n$ and
\begin{eqnarray}
e_\chi\theta_{L_\infty/k,S,T}(0)= e_\chi {\rm Rec}_\infty(\epsilon_{L/k,S,T}^{V'}) \text{ in }\CC_p[G]\otimes_{\ZZ_p} \varprojlim_n I(\Gamma_n)^{r'}/I(\Gamma_n)^{r'+1}, \label{mrsgs}
\end{eqnarray}
where we set
$$\theta_{L_\infty/k,S,T}(0):=\varprojlim_n \theta_{L_n/k,S,T}(0) \in \varprojlim_n I_n^{r'}/I_n^{r'+1} \simeq \ZZ_p[G]\otimes_{\ZZ_p} \varprojlim_n I(\Gamma_n)^{r'}/I(\Gamma_n)^{r'+1}.$$
We also note that the validity of (\ref{vanishing order stickelberger}) follows as a consequence of our Iwasawa main conjecture (Conjecture \ref{IMC}) by using Proposition \ref{explicit projector}(iii) and the result of \cite[Lemma 5.19]{bks1} (see the argument in \S \ref{proof main result}).

To study (\ref{mrsgs}) we set $\chi_1:=\chi |_\Delta \in \widehat \Delta$ and regard (as we may) the product $\chi_2:=\chi \chi_1^{-1}$ as a character of $\Gamma=\Gal(k_\infty/k)$.

Note that $\Gal(L_\infty/k)=G_{\chi_1}\times \Gamma_{\chi_1}$. Fix a topological generator $\gamma \in \Gamma_{\chi_1}$, and identify $\ZZ_p[{\rm im}(\chi_1)][[\Gamma_{\chi_1}]]$ with the ring of power series $\ZZ_p[{\rm im}(\chi_1)][[t]]$ via the correspondence $\gamma=1+t$.

We then define $g_{L_\infty/k,S,T}^{\chi_1}(t)$ to be the image of $\theta_{L_\infty/k,S,T}(0)$ under the map
$$\ZZ_p[[\Gal(L_\infty/k)]]=\ZZ_p[G_{\chi_1}][[ \Gamma_{\chi_1}]] \to \ZZ_p[{\rm im}(\chi_1)][[\Gamma_{\chi_1}]]=\ZZ_p[{\rm im}(\chi_1)][[t]]$$
induced by $\chi_1$. We recall that the $p$-adic $L$-function of Deligne-Ribet is defined by
$$L_{k,S,T,p}(\chi^{-1}\omega,s):=g_{L_\infty/k,S,T}^{\chi_1}(\chi_2(\gamma)\chi_{\rm cyc}(\gamma)^s-1),$$
where $\chi_{\rm cyc}$ is the cyclotomic character, and we note that one can show $L_{k,S,T,p}(\chi^{-1}\omega,s)$ to be independent of the choice of $\gamma$.

The validity of (\ref{vanishing order stickelberger}) implies an inequality
\begin{eqnarray}
\ord_{s=0}L_{k,S,T,p}(\chi^{-1}\omega,s) \geq r'. \label{p adic L order}
\end{eqnarray}
It is known that (\ref{p adic L order}) is a consequence of the Iwasawa main conjecture (in the sense of Wiles \cite{Wiles}), which is itself known to be valid when $p$ is odd. In addition, Spiess has recently proved that (\ref{p adic L order}) is valid, including the case $p=2$, by using Shintani cocycles \cite{spiess}. In all cases, therefore, we can define
$$L_{k,S,T,p}^{(r')}(\chi^{-1}\omega,0):=\lim_{s\to 0} s^{-r'}L_{k,S,T,p}(\chi^{-1}\omega,s) \in \CC_p. $$

For $v \in V'$, define
$${\rm Log}_w: L^\times \to \ZZ_p[G]$$
by
$${\rm Log}_w(a):=-\sum_{\sigma\in G}\log_p({\N}_{L_w/\QQ_p}(\sigma a))\sigma^{-1},$$
where $\log_p: \QQ_p^\times \to \ZZ_p$ is Iwasawa's logarithm (in the sense that $\log_{p}(p)=0$). We set
$${\rm Log}_{V'}:=\bigwedge_{v\in V'}{\rm Log}_w: \CC_p \bigwedge^{r'} U_{L,S,T} \to \CC_p[G].$$
We shall denote the map $\CC_p[G]\to \CC_p$ induced by $\chi$ also by $\chi$.

For $v\in V'$, we define
$${\rm Ord}_w: L^\times \to \ZZ[G]$$
by
$${\rm Ord}_w(a):=\sum_{\sigma\in G}\ord_w(\sigma a)\sigma^{-1},$$
and set
$${\rm Ord}_{V'}:= \bigwedge_{v\in V'}{\rm Ord}_w: \CC_p \bigwedge^{r'} U_{L,S,T} \to \CC_p[G].$$
On the $\chi$-component, ${\rm Ord}_{V'}$ induces an isomorphism
$$\chi\circ {\rm Ord}_{V'}: e_\chi \CC_p \bigwedge^{r'} U_{L,S,T} \stackrel{\sim}{\to} \CC_p. $$
Taking a non-zero element $x\in e_\chi \CC_p \bigwedge^{r'} U_{L,S,T}$, we define the $\mathcal{L}$-invariant by
$$\mathcal{L}(\chi):=\frac{\chi(\Log_{V'}(x))}{\chi(\Ord_{V'}(x))} \in \CC_p. $$

Since $e_\chi \CC_p \bigwedge^{r'} U_{L,S,T}$ is a one dimensional $\CC_p$-vector space, we see that $\mathcal{L}(\chi)$ does not depend on the choice of $x$.

Then the Gross-Stark conjecture is stated as follows.

\begin{conjecture}[{${\rm GS}(L/k,S,T,\chi)$}]\label{gross stark conj}
$$L_{k,S,T,p}^{(r')}(\chi^{-1}\omega,0)=\mathcal{L}(\chi) L_{k,S\setminus V',T}(\chi^{-1},0).$$
%$$L_{k,S,T,p}^{(r')}(\chi^{-1}\omega,0)=\chi({\rm Log}_{V'}(\epsilon_{L/k,S,T}^{V'})).$$
\end{conjecture}

\begin{remark}
%The Gross-Stark conjecture can be rephrased in terms of `$\mathcal{L}$-invariants'. For $v\in V'$, define
%$${\rm Ord}_w: L^\times \to \ZZ[G]$$
%by
%$${\rm Ord}_w(a):=\sum_{\sigma\in G}\ord_w(\sigma a)\sigma^{-1},$$
%and set
%$${\rm Ord}_{V'}:= \bigwedge_{v\in V'}{\rm Ord}_w: \CC_p \bigwedge^{r'} U_{L,S,T} \to \CC_p[G].$$
%On the $\chi$-component, ${\rm Ord}_{V'}$ induces the isomorphism
%$$\chi\circ {\rm Ord}_{V'}: e_\chi \CC_p \bigwedge^{r'} U_{L,S,T} \stackrel{\sim}{\to} \CC_p, $$
%and this isomorphism sends $e_\chi\epsilon_{L/k,S,T}^{V'}$ to $L_{k,S\setminus V',T}(\chi^{-1},0)$ (see \cite[Proposition 5.2]{R}). Taking a non-zero element $x\in e_\chi \CC_p \bigwedge^{r'} U_{L,S,T}$, we define the $\mathcal{L}$-invariant by
%$$\mathcal{L}(\chi):=\frac{\chi(\Log_{V'}(x))}{\chi(\Ord_{V'}(x))} \in \CC_p. $$
%Since $e_\chi \CC_p \bigwedge^{r'} U_{L,S,T}$ is a one dimensional $\CC_p$-vector space, we see that $\mathcal{L}(\chi)$ does not depend on the choice of $x$. Letting $x=e_\chi \epsilon_{L/k,S,T}^{V'}$, we obtain
%$$\chi({\rm Log}_{V'}(\epsilon_{L/k,S,T}^{V'}))=\mathcal{L}(\chi) L_{k,S\setminus V',T}(\chi^{-1},0).$$
%Thus we see that Conjecture ${\rm GS}(L/k,S,T,\chi)$ is equivalent to the equality
%$$L_{k,S,T,p}^{(r')}(\chi^{-1}\omega,0)=\mathcal{L}(\chi) L_{k,S\setminus V',T}(\chi^{-1},0).$$
This formulation constitutes a natural higher rank generalization of the form of the Gross-Stark conjecture that is considered by Darmon, Dasgupta and Pollack (see \cite[Conjecture 1]{DDP}).
\end{remark}

Letting $x=e_\chi \epsilon_{L/k,S,T}^{V'}$, we obtain
$$\chi({\rm Log}_{V'}(\epsilon_{L/k,S,T}^{V'}))=\mathcal{L}(\chi) L_{k,S\setminus V',T}(\chi^{-1},0).$$
Thus we see that Conjecture ${\rm GS}(L/k,S,T,\chi)$ is equivalent to the equality
$$L_{k,S,T,p}^{(r')}(\chi^{-1}\omega,0)=\chi({\rm Log}_{V'}(\epsilon_{L/k,S,T}^{V'})).$$

Concerning the relation between ${\rm Rec}_\infty$ and ${\rm Log}_{V'}$, we note the fact
$$\chi_{\rm cyc}({\rm rec}_w(a))={\N}_{L_w/\QQ_p}(a)^{-1},$$
where $v \in V'$ and $a\in L^\times$.

Given this fact, it is straightforward to check (under the validity of (\ref{vanishing order stickelberger})) that Conjecture ${\rm GS}(L/k,S,T,\chi)$ is equivalent to (\ref{mrsgs}).

At this stage we have therefore proved the following result.

\begin{theorem}\label{GS thm}
Suppose that $k$ is totally real, $k_\infty/k$ is the cyclotomic $\ZZ_p$-extension,
and $\chi$ is totally odd. Set $V':=\{v\in S \mid \chi(G_v)=1\}$ and assume that every $v\in V'$ lies above $p$. Assume also that (\ref{vanishing order stickelberger}) is valid.
%Assume that the $\chi$-part of the Gross-Stark conjecture
%holds for $(L/k,S,T)$.
Then Conjecture ${\rm GS}(L/k,S,T,\chi)$ is equivalent to Conjecture
${\rm MRS}(K_\infty/k,S,T,\chi,V')$.
\end{theorem}

\subsection{A proof in the case $k=\QQ$}

In \cite[Corollary 1.2]{bks1} the known validity of the eTNC for Tate motives over abelian fields is used to prove that Conjecture ${\rm MRS}(K/L/k,S,T,V,V')$ is valid in the case $k=\QQ$.

In this subsection, we shall give a much simpler proof of the latter result which uses only Theorem \ref{GS thm}, the known validity of the Gross-Stark conjecture over abelian fields and a classical result of Solomon \cite{solomon}.

We note that for any $\chi$ and $n$ the Rubin-Stark conjecture is known to be true for $(L_{\chi,n}/\QQ,S,T,V_\chi)$. (In this setting the Rubin-Stark element is given by a cyclotomic unit (resp. the Stickelberger element) when $r_\chi=1$ (resp. $r_\chi=0$).)

\begin{theorem} \label{solomonFG}
Suppose that $k=\QQ$. Then, ${\rm MRS}(K_\infty/k,S,T,\chi,V')$ is valid.
\end{theorem}

\begin{proof}
 By Proposition \ref{prop mrs}(ii), we may assume that $V'$ is maximal, namely,
 $$r'= {\rm min} \{\#\{ v\in S \mid \chi(G_v)=1\}, \# S-1\}.$$
By Corollary \ref{nontrivzeros}, we may assume that $\chi(p)=1$.

Suppose first that $\chi$ is odd. Since Conjecture ${\rm GS}(L/\QQ,S,T,\chi)$ is valid (see \cite[\S4]{Gp}), Conjecture ${\rm MRS}(K_\infty/\QQ,S,T,\chi,V')$ follows from Theorem \ref{GS thm}.

Suppose next that $\chi=1$. In this case we have $r'=\#S-1$. We may assume $p \notin V'$ by Proposition \ref{prop mrs}(iii). In this case every $v \in V'\setminus V$ is unramified in $L_{\infty}$. Hence, the theorem follows from Corollary \ref{cor unram}.

Finally, suppose that $\chi \neq 1$ is even. By Proposition \ref{prop mrs}(iv) and (v), we may assume
$$S=\{\infty,p\}\cup S_{\rm ram}(L/\QQ) \text{ and }V'=\{\infty,p\}.$$
We label $S=\{v_0,v_1,\ldots\}$ so that $v_1=\infty$ and $v_2=p$.

Fix a topological generator $\gamma$ of $\Gamma=\Gal(L_\infty/L)$. Then we construct an element $\kappa(L,\gamma)\in \varprojlim_n L^\times/(L^\times)^{p^n}$ as follows. Note that $\N_{L_n/L}(\epsilon_{L_n/\QQ,S,T}^V)$ vanishes since $\chi(p)=1$. So we can take $\beta_n\in L_n^\times$ such that $\beta_n^{\gamma-1}=\epsilon_{L_n/\QQ,S,T}^V$ (Hilbert's theorem 90). Define
$$\kappa_n:=\N_{L_n/L}(\beta_n) \in L^\times/(L^\times)^{p^n}.$$
This element is independent of the choice of $\beta_n$, and for any $m>n$ the natural map
$$L^\times/(L^\times)^{p^m} \to L^\times/(L^\times)^{p^n}$$
sends $\kappa_m$ to $\kappa_n$. We define
\[ \kappa(L,\gamma):=(\kappa_n)_n\in \varprojlim_n L^\times/(L^\times)^{p^n}.\]
Then, by Solomon \cite[Proposition 2.3(i)]{solomon}, we know that
$$\kappa(L,\gamma)\in \ZZ_p \otimes_\ZZ \mathcal{O}_{L}\left[\frac1p\right]^\times \hookrightarrow \varprojlim_n L^\times/(L^\times)^{p^n}.$$

Fix a prime $\mathfrak{p}$ of $L$ lying above $p$. Define
$${\rm Ord}_\frp : L^\times \to \ZZ_p[G]$$
by ${\rm Ord}_\frp(a):=\sum_{\sigma \in G}{\rm ord}_\frp(\sigma a)\sigma^{-1}$. Similarly, define
$${\rm Log}_\frp: L^\times \to \ZZ_p[G]$$
by ${\rm Log}_\frp(a):=-\sum_{\sigma \in G}\log_p(\iota_\frp (\sigma a))\sigma^{-1}$, where $\iota_\frp: L \hookrightarrow L_\frp =\QQ_p$ is the natural embedding.

Then by the result of Solomon \cite[Theorem 2.1 and Remark 2.4]{solomon}, one deduces
\begin{eqnarray}
{\rm Ord}_\frp(\kappa(L,\gamma))=-\frac{1}{\log_p(\chi_{\rm cyc}(\gamma))} {\rm Log}_\frp(\epsilon_{L/\QQ,S\setminus \{p\},T}^V). \nonumber
\end{eqnarray}
%where $\chi_{\rm cyc}$ is the cyclotomic character.
From this, we have
\begin{eqnarray}
{\rm Ord}_\frp(\kappa(L,\gamma))\otimes (\gamma-1)=-{\rm Rec}_\frp(\epsilon_{L/\QQ,S\setminus \{ p\} ,T}^V) \text{ in }\ZZ_p[G]\otimes_{\ZZ_p}I(\Gamma)/I(\Gamma)^2,\label{solomon eq}
\end{eqnarray}
where $I(\Gamma)$ is the augmentation ideal of $\ZZ_p[[\Gamma]]$.

We know that $e_\chi \CC_p U_{L,S}$ is a two-dimensional $\CC_p$-vector space. Lemma \ref{lemma basis} below shows that $\{ e_\chi \epsilon_{L/\QQ,S\setminus\{p\},T}^V, e_\chi \kappa(L,\gamma)\}$ is a $\CC_p$-basis of this space. For simplicity, set $\epsilon_L^V:=\epsilon_{L/\QQ,S\setminus\{p\},T}^V$. Note that the isomorphism
$${\rm Ord}_\frp: e_\chi \CC_p \bigwedge^2 U_{L,S} \stackrel{\sim}{\to} e_\chi \CC_p U_L$$
sends $e_\chi \epsilon_L^V\wedge \kappa(L,\gamma)$ to $-\chi({\rm Ord}_\frp(\kappa(L,\gamma)))e_\chi \epsilon_L^V$. Since we have
$${\rm Ord}_\frp(e_\chi \epsilon_{L/\QQ,S,T}^{V'})=- e_\chi \epsilon_L^V$$
(see \cite[Proposition 5.2]{R} or \cite[Proposition 3.6]{sano}), we have
$$e_\chi \epsilon_{L/\QQ,S,T}^{V'}=- \chi({\rm Ord}_\frp(\kappa(L,\gamma)))^{-1} e_\chi \epsilon_L^V \wedge \kappa(L,\gamma).$$
Hence we have
\begin{eqnarray}
{\rm Rec}_\frp(e_\chi \epsilon_{L/\QQ,S,T}^{V'})&=& \chi({\rm Ord}_\frp(\kappa(L,\gamma)))^{-1}e_\chi \kappa(L,\gamma)\cdot {\rm Rec}_\frp(\epsilon_L^V) \nonumber \\
&=&- e_\chi \kappa(L,\gamma)\otimes (\gamma-1), \nonumber
\end{eqnarray}
where the first equality follows by noting that ${\rm Rec}_\frp(\kappa(L,\gamma))=0$ (since $\kappa(L,\gamma)$ lies in the universal norm by definition), and the second by (\ref{solomon eq}).

Now, noting that
$$\nu_n: U_{L,S,T}\otimes_{\ZZ_p}I(\Gamma_n)/I(\Gamma_n)^2 \hookrightarrow U_{L_n,S,T}\otimes_{\ZZ_p} \ZZ_p[\Gamma_n]/I(\Gamma_n)^2$$
is induced by the inclusion map $L\hookrightarrow L_n$, and that
$$\mathcal{N}_n(\epsilon_{L_n/\QQ,S,T}^V)=\kappa_n \otimes (\gamma-1),$$
it is easy to see that the element $\kappa:=\kappa(L,\gamma)\otimes(\gamma-1)$ has the properties in the statement of Conjecture ${\rm MRS}(K_\infty/\QQ,S,T,\chi,V')$.

This completes the proof the claimed result.
\end{proof}

\begin{lemma} \label{lemma basis}
Assume that $k=\QQ$ and $\chi \neq 1$ is even such that $\chi(p)=1$. Assume also that
$S=\{\infty,p\}\cup S_{\rm ram}(L/\QQ)$.
Then, $\{ e_\chi \epsilon_{L/\QQ,S\setminus\{p\},T}^V, e_\chi \kappa(L,\gamma)\}$ is a $\CC_p$-basis of $e_\chi \CC_p U_{L,S}$.
\end{lemma}

\begin{proof}
This result follows from \cite[Remark 4.4]{solomon2}. But we give a sketch of another proof, which is essentially given by Flach in \cite{flachsurvey}.

In the next section, we define the `Bockstein map'
$$\beta: e_\chi \CC_p U_{L,S} \to e_\chi \CC_p(\cX_{L,S}\otimes_{\ZZ_p}I(\Gamma)/I(\Gamma)^2).$$
We see that $\beta$ is injective on $e_\chi \CC_p U_L$, and that
$$\ker \beta\simeq U_{L_\infty,S} \otimes_\Lambda \CC_p,$$
where we put $\Lambda:=\ZZ_p[[\G]]$ and $\CC_p$ is regarded as a $\Lambda$-algebra via $\chi$. Hence we have
$$e_\chi \CC_p U_{L,S}= e_\chi \CC_p U_{L} \oplus (U_{L_\infty,S}\otimes_\Lambda \CC_p).$$
Since $e_\chi\epsilon_{L/\QQ,S\setminus\{p\},T}^V$ is non-zero, this is a basis of $e_\chi \CC_p U_{L,S\setminus\{p\}}=e_\chi \CC_p U_{L}$. We prove that $e_\chi \kappa(L,\gamma)$ is a basis of $U_{L_\infty,S}\otimes_\Lambda \CC_p$.

By using the exact sequence
$$0 \to U_{L_\infty,S} \stackrel{\gamma-1}{\to} U_{L_\infty,S} \to U_{L,S},$$
we see that there exists a unique element $\alpha \in U_{L_\infty,S}$ such that $(\gamma-1)\alpha=\epsilon_{L_\infty/\QQ,S,T}^V$. By the cyclotomic Iwasawa main conjecture over $\QQ$, we see that $\alpha$ is a basis of $U_{L_\infty,S}\otimes_\Lambda \Lambda_{\frp_\chi}$, where $\frp_\chi:=\ker(\chi:\Lambda \to \CC_p)$. The image of $\alpha$ under the map
$$U_{L_\infty,S}\otimes_\Lambda \Lambda_{\frp_\chi} \stackrel{\chi}{\to} U_{L_\infty,S}\otimes_\Lambda \CC_p \hookrightarrow e_\chi \CC_p U_{L,S}$$
is equal to $e_\chi \kappa(L,\gamma)$.\end{proof}

\section{A strategy for proving the eTNC} \label{descent section}

\subsection{Statement of the main result and applications}

In the sequel we fix an intermediate field $L$ of $K_\infty/k$ which is finite over $k$ and set $G:=\Gal(L/k)$.
In this section we always assume the following conditions to be satisfied:
\begin{itemize}
\item[(R)] for every $\chi \in \widehat G$, one has $r_{\chi,S}<\# S$;
\item[(S)] no finite place of $k$ splits completely in $k_\infty$.
\end{itemize}

\begin{remark} Before proceeding we note that condition (R) is very mild since it is automatically satisfied when the class number of $k$ is equal to one and, for any $k$, is satisfied when $S$ is large enough. We also note that condition (S) is satisfied when, for example, $k_\infty/k$ is the cyclotomic $\ZZ_p$-extension.
\end{remark}

The following result is one of the main results of this article and, as we will see, it provides an effective strategy for proving the special case of the eTNC that we are considering here.

\begin{theorem} \label{mainthm}
Assume the following conditions:
\begin{itemize}
%\item[(RS)] for every $\chi \in \widehat G$, Conjecture ${\rm RS}(L_{\chi,n}/k,S,T,V_\chi)_p$ is valid for all $n$;
\item[(hIMC)] The main conjecture ${\rm IMC}(K_\infty/k,S,T)$ is valid;
\item[(F)] for every $\chi$ in $\widehat G$, the module of $\Gamma_{\chi}$-coinvariants of $A_S^T(L_{\chi,\infty})$ is finite;
\item[(MRS)] for every $\chi$ in $\widehat G$, Conjecture ${\rm MRS}(K_\infty/k,S,T,\chi,V_\chi')$ is valid for a maximal set $V_\chi'$ (so that  $\# V_\chi'={\rm min}\{ \#\{ v\in S \mid  \chi(G_v)=1\}, \#S-1\}).$
\end{itemize}
Then, the conjecture ${\rm eTNC}(h^0(\Spec L),\ZZ_p[G])$ is valid.
\end{theorem}

\begin{remark}
 We note that the set $V_\chi'$ in condition (MRS) is not uniquely determined when every place $v$ in $S$ satisfies $\chi(G_v)=1$, but that the validity of the conjecture ${\rm MRS}(K_\infty/k,S,T,\chi,V_\chi')$ is independent of the choice of $V_\chi'$ (by Proposition \ref{prop mrs}(iii)).
\end{remark}

\begin{remark} One checks easily that the condition (F) is equivalent to the finiteness of the the module of $\Gamma_\chi$-coinvariants of $A_S(L_{\chi,\infty})$. Hence, taking account of an observation of Kolster in \cite[Theorem 1.14]{kolster}, condition (F) can be regarded as a natural generalization of the Gross conjecture \cite[Conjecture 1.15]{Gp}. In particular, we recall that condition (F) is satisfied in each of the following cases:
\begin{itemize}
\item $L$ is abelian over $\QQ$ (due to Greenberg, see \cite{greenberg}),
\item $k_\infty/k$ is the cyclotomic $\ZZ_p$-extension and $L$ has unique $p$-adic place (in this case `$\delta_L=0$' holds obviously, see \cite{kolster}),
\item $L$ is totally real and the Leopoldt conjecture is valid for $L$ at $p$ (see \cite[Corollary 1.3]{kolster}).
\end{itemize}
\end{remark}

\begin{remark}
The condition (MRS) is satisfied for $\chi$ in $\widehat G$ when the condition ${\rm NTZ}(K_\infty/k,\chi)$ is satisfied (see Corollary \ref{nontrivzeros}).
\end{remark}

As an immediate corollary of Theorem \ref{mainthm}, we obtain a new proof of a theorem that was first proved by Greither and the first author \cite{bg} for $p$ odd, and by Flach \cite{fg} for $p=2$.

\begin{corollary} \label{burnsgreither}
If $k=\QQ$, then the conjecture ${\rm eTNC}(h^0(\Spec L),\ZZ_p[G])$ is valid.
\end{corollary}

\begin{proof}
As we mentioned above, the conditions (R), (S) and (F) are all satisfied in this case. In addition, the condition (hIMC) is a direct consequence of  the classical Iwasawa main conjecture solved by Mazur and Wiles (see \cite{bg} and \cite{fg}) and the condition (MRS) is satisfied by Theorem \ref{solomonFG}.
\end{proof}

We also obtain a result over totally real fields.

\begin{corollary}
Suppose that $p$ is odd, $k$ is totally real, $k_\infty/k$ is the cyclotomic $\ZZ_p$-extension, and $K$ is CM. Assume that (F) is satisfied, that the $\mu$-invariant of $K_\infty/K$ vanishes, and that for every odd character $\chi \in \widehat G$ Conjecture ${\rm GS}(L_\chi/k,S,T,\chi)$ is valid. Then, Conjecture ${\rm eTNC}(h^0(\Spec L), \ZZ_p[G]^-)$ is valid.
%In particular, if at most one $p$-adic prime $\mathfrak{p}$ of $k$ satisfies $\chi(\mathfrak{p})=1$ for each odd character $\chi\in \widehat G$, then Conjecture ${\rm eTNC}(h^0(\Spec L),\ZZ_p[G]^-)$ is (unconditionally) valid.
\end{corollary}

\begin{proof}
We take $S$ so that condition (R) is satisfied. Then the minus-part of condition (hIMC) is satisfied by Theorem \ref{CM theorem}(i) and the minus part of condition (MRS) by Theorem \ref{GS}.
\end{proof}

When at most one $p$-adic place $\mathfrak{p}$ of $k$ satisfies $\chi(G_\mathfrak{p})=1$, Dasgupta, Darmon and Pollack proved the validity of Conjecture ${\rm GS}(L_\chi/k,S,T,\chi)$ under some assumptions including Leopoldt's conjecture (see \cite{DDP}). Recently in the same case Ventullo asserts in \cite{ventullo} that Conjecture ${\rm GS}(L_\chi/k,S,T,\chi)$ is unconditionally valid. In this case condition (F) is also valid by the argument of Gross in \cite[Proposition 2.13]{Gp}. Hence we get the following

\begin{corollary} \label{MC1}
Suppose that $p$ is odd, $k$ is totally real, $k_\infty/k$ is the cyclotomic $\ZZ_p$-extension, and $K$ is CM. Assume that the $\mu$-invariant of $K_\infty/K$ vanishes, and that for each odd character $\chi \in \widehat G$ there is at most one $p$-adic place $\frp$ of $k$ which satisfies $\chi(G_\frp)=1$. Then, Conjecture ${\rm eTNC}(h^0(\Spec L), \ZZ_p[G]^-)$ is valid.
\end{corollary}

\begin{examples} \label{RemarkExample}
%Corollary \ref{MC1} clearly implies Corollary \ref{IntroCor} in \S 1.
It is not difficult to find many concrete families of examples which satisfy all of the hypotheses of Corollary \ref{MC1} and hence to deduce the validity of ${\rm eTNC}(h^0(\Spec L), \ZZ_p[G]^-)$ in some new and interesting cases.
 In particular, we shall now describe several families of examples in which the extension $k/\QQ$ is not abelian (noting that if $L/\QQ$ is abelian and $k \subset L$, then ${\rm eTNC}(h^0(\Spec L), \ZZ_p[G])$ is already known to be valid).
\

\noindent{}(i) The case $p=3$. As a simple example, we consider the case that $k/\QQ$ is a $S_{3}$-extension.
 To do this we fix an irreducible cubic polynomial $f(x)$ in $\ZZ[x]$ with discriminant $27d$
 where $d$ is strictly positive and congruent to $2$ modulo $3$. (For example, one can take $f(x)$ to be $x^3-6x-3$, $x^3-15x-3$, etc.)
 The minimal splitting field $k$ of $f(x)$ over $\QQ$ is then totally real (since $27d>0$) and an $S_{3}$-extension of $\QQ$ (since $27d$ is not a square). Also, since the discriminant of $f(x)$ is divisible by $27$ but not $81$, the prime $3$ is totally ramified in $k$.
 Now set $p := 3$ and  $K:=k(\mu_{p})=k(\sqrt{-p})=k(\sqrt{-d})$. Then the prime above $p$ splits in $K/k$ because $-d \equiv 1$ (mod $3$). In addition, as $K/\QQ(\sqrt{d}, \sqrt{-p})$ is a cyclic cubic extension, the $\mu$-invariant of $K_{\infty}/K$ vanishes and so the extension $K/k$ satisfies all the conditions of Corollary \ref{MC1} (with $p=3$).
\

\noindent{}(ii) The case $p>3$. In this case one can construct a suitable field $K$ in the following way. Fix  a primitive $p$-th root of unity $\zeta$, an integer $i$ such that $1 \leq i \leq (p-3)/2$ and an integer $b$ which is prime to $p$, and then set
\[ a:=(1+b(\zeta-1))^{2i+1}/(1+b(\zeta^{-1}-1)^{2i+1}).\]
Write $\ord_{\pi}$ for
the normalized additive valuation of $\QQ(\mu_{p})$ associated to the prime element $\pi=\zeta-1$. Then, since $\ord_{\pi}(a-1)=2i+1<p$, $(\pi)$ is totally ramified in $\QQ(\mu_{p}, \sqrt[p]{a})/\QQ(\mu_{p})$. Also, since $\rho(a)=a^{-1}$ where $\rho$ is the complex conjugation,
$\QQ(\mu_{p}, \sqrt[p]{a})$ is the composite of a cyclic extension of
$\QQ(\mu_{p})^{+}$ of degree $p$ and $\QQ(\mu_{p})$. This shows that $\QQ(\mu_{p}, \sqrt[p]{a})$ is a CM-field and, since $1 < 2i+1 <p$, the extension $\QQ(\mu_{p}, \sqrt[p]{a})^{+}/\QQ$ is non-abelian.
We now take a negative integer $-d$ which is a quadratic residue modulo $p$, let $K$ denote the CM-field $\QQ(\mu_{p}, \sqrt[p]{a}, \sqrt{-d})$ and
 set $k:=K^+$. Then $p$ is totally ramified in $k/\QQ$ and the $p$-adic prime of
$k$ splits in $K$. In addition, $k/\QQ$ is not abelian and the $\mu$-invariant of $K_{\infty}/K$ vanishes
since $K/\QQ(\mu_{p}, \sqrt{-d})$ is cyclic of degree $p$. This shows that the extension $K/k$ satisfies all of the hypotheses of Corollary \ref{MC1}.
\

\noindent{}(iii) In both of the cases (i) and (ii) described above, $p$ is totally ramified in the extension $k_{\infty}/\QQ$ and so Corollary \ref{MC1} implies that ${\rm eTNC}(h^0(\Spec K_{n}), \ZZ_p[G]^-)$ is valid for any non-negative integer $n$. In addition, if $F$ is any real abelian field of degree prime to $[k: \QQ]$ in which
$p$ is totally ramified, the minus component of the $p$-part of eTNC for $FK_{n}/k$ holds for any non-negative integer $n$.
\end{examples}

\begin{remark} Finally we note that, by using similar methods
to the proofs of the above corollaries it is also possible to deduce the main result of Bley \cite{bley} as a consequence of Theorem \ref{mainthm}. In this case $k$ is imaginary quadratic, the validity of (hIMC) can be derived from Rubin's result in \cite{rubinIMC} (as explained in \cite{bley}), and the conjecture (MRS) from Bley's result \cite{bleysolomon}, which is itself an analogue of Solomon's theorem \cite{solomon} for elliptic units, by using the same argument as Theorem \ref{solomonFG}.
\end{remark}

\subsection{A computation of Bockstein maps}

Fix a character $\chi \in \widehat G$.
%In the following we assume that (R) is satisfied.
For simplicity, we set
\begin{itemize}
\item $L_n:=L_{\chi,n}$;
\item $L:=L_\chi$;
%\item $\G_{n}:=\G_{\chi,n}=\Gal(L_{\chi,n}/k)$;
%\item $\G:=\G_\chi=\Gal(L_{\chi,\infty}/k)$;
%\item $G_\chi=\Gal(L_\chi/k)$;
%\item $\Gamma_{n}:=\Gamma_{\chi,n}=\Gal(L_{\chi,n}/L_\chi)$;
%\item $\Gamma:=\Gamma_\chi=\Gal(L_{\chi,\infty}/L_\chi)$;
\item $V:=V_\chi=\{ v\in S \mid \text{$v$ splits completely in $L_{\chi,\infty}$}\}$;
\item $r:=r_\chi=\# V_\chi$;
\item $V':=V_{\chi}'$ (as in (MRS) in Theorem \ref{mainthm});
\item $r':=r_{\chi,S}=\# V' $;
%\item $W:=V'\setminus V$;
\item $e:=r'-r$.
\end{itemize}

As in \S \ref{formulate mrs}, we label $S=\{v_0,v_1,\ldots\}$ so that $V=\{v_1,\ldots,v_r\}$ and $V'=\{ v_1,\ldots ,v_{r'}\}$, and fix a place $w$ lying above each $v \in S$.
Also, as in \S \ref{section explicit}, it will be useful to fix a representative of $C_{K_\infty,S,T}$:
$$\Pi_{K_\infty} \to \Pi_{K_\infty},$$
where the first term is placed in degree zero, and $\Pi_{K_\infty}$ is a free $\Lambda$-module with basis $\{b_1,\ldots,b_d\}$. This representative is chosen so that the natural surjection
$$\Pi_{K_\infty} \to H^1(C_{K_\infty,S,T}) \to \cX_{K_\infty,S}$$
sends $b_i$ to $w_i-w_0$ for every $i$ with $1\leq i \leq r'$.

We define a height one regular prime ideal of $\Lambda$ by setting
\[ \mathfrak{p}:=\ker(\Lambda \stackrel{\chi}{\to} \QQ_p(\chi):=\QQ_p(\im \chi)).\]
Then the localization $R:=\Lambda_{\mathfrak{p}}$ is a discrete valuation ring and we write $P$ for its maximal ideal. We see that $\chi$ induces an isomorphism
$$E:=R/P \stackrel{\sim}{\to} \QQ_p(\chi).$$
We set $C:=C_{K_\infty,S,T}\otimes_\Lambda R$ and $\Pi:=\Pi_{K_\infty}\otimes_{\Lambda}R$.

\begin{lemma} \label{uniformizer}
Let $\gamma$ be a topological generator of $\Gamma=\Gal(K_\infty/K)$. Let $n$ be an integer which satisfies $\gamma^{p^n} \in \Gal(K_\infty/L)$. Then $\gamma^{p^n}-1$ is a uniformizer of $R$.
\end{lemma}

\begin{proof}
%We may assume that $n$ is minimal. Indeed, if $\gamma^{p^{n+m}}\in \Gal(K_\infty/L)$,  then we have
%$$\gamma^{p^{n+m}}-1=(\gamma^{p^n}-1)(\gamma^{p^n(p^m-1)}+\gamma^{p^n(p^m-2)}+\cdots + 1),$$
%and $\gamma^{p^n(p^m-1)}+\gamma^{p^n(p^m-2)}+\cdots + 1 \in R^\times$. This shows that $(\gamma^{p^{n+m}}-1)=(\gamma^{p^n}-1)$ as ideals of $R$.

Regard $\chi \in \widehat \G$, and put $\chi_1:=\chi|_\Delta \in \widehat \Delta$. We identify $R$ with the localization of $\Lambda_{\chi_1}[1/p]=\ZZ_p[\im \chi_1][[\Gamma]][1/p]$ at $\mathfrak{q}:=\ker(\Lambda_{\chi_1}[1/p] \stackrel{\chi|_\Gamma}{\to} \QQ_p(\chi))$.

Then the lemma follows by noting that the localization of $\Lambda_{\chi_1}[1/p]/(\gamma^{p^n}-1)=\ZZ_p[\im \chi_1][\Gamma_n][1/p]$ at $\mathfrak{q}$ is identified with $\QQ_p(\chi)$.
\end{proof}

\begin{lemma} \label{keylemma}
Assume that the condition (F) is satisfied.
\begin{itemize}
\item[(i)] $H^0(C)$ is isomorphic to $U_{K_\infty,S,T}\otimes_\Lambda R$, and $R$-free of rank $r$.
\item[(ii)] $H^1(C)$ is isomorphic to $\cX_{K_\infty,S}\otimes_\Lambda R$.
\item[(iii)] The maximal $R$-torsion submodule
$H^1(C)_{\rm tors}$ of $H^1(C)$ is isomorphic to $\cX_{K_\infty,S\setminus V}\otimes_\Lambda R$, and annihilated by $P$. (So $H^1(C)_{\rm tors}$ is an $E$-vector space.)
\item[(iv)] $H^1(C)_{\rm tf}:=H^1(C)/H^1(C)_{\rm tors}$ is isomorphic to $\cY_{K_\infty,V}\otimes_\Lambda R$ and is therefore $R$-free of rank $r$.
\item[(v)] $\dim_E(H^1(C)_{\rm tors})=e$.
\end{itemize}
\end{lemma}

\begin{proof}
Since $U_{K_\infty,S,T}\otimes_{\Lambda}R=H^0(C)$ is regarded as a submodule of $\Pi$, we see that $U_{K_\infty,S,T}\otimes_{\Lambda}R$ is $R$-free. Put $\chi_1:=\chi |_\Delta \in \widehat \Delta$. Note that $L_{\infty}:=L_{\chi,\infty}=L_{\chi_1,\infty}$, and that the quotient field of $R$ is $Q(\Lambda_{\chi_1})$. As in the proof of Theorem \ref{lemisom}, we have
$$U_{K_\infty,S,T}\otimes_\Lambda Q(\Lambda_{\chi_1}) \simeq \cY_{L_\infty,V}\otimes_{\ZZ_p[[\G_\chi]]} Q(\Lambda_{\chi_1}).$$
These are $r$-dimensional $Q(\Lambda_{\chi_1})$-vector spaces. This proves (i).

To prove (ii), it is sufficient to show that $A_S^T(K_\infty)\otimes_{\Lambda}R=0$.
Fix a topological generator $\gamma$ of $\Gamma$, and regard $\ZZ_p[[\Gamma]]$ as the ring of power series $\ZZ_p[[T]]$ via the identification $\gamma=1+T$.
Let $f$ be the characteristic polynomial of the $\ZZ_p[[T]]$-module $A_S^T(L_\infty)$.
By Lemma \ref{uniformizer}, for sufficiently large $n$, $\gamma^{p^n}-1$ is a
uniformizer of $R$.
On the other hand, by the assumption (F), we see that $f$ is prime to
$\gamma^{p^n}-1$. This implies (ii).

We prove (iii). Proving that $H^1(C)_{\rm tors}$ is isomorphic to $\cX_{K_\infty,S\setminus V}\otimes_\Lambda R$, it is sufficient to show that
$$\cX_{K_\infty,S}\otimes_\Lambda Q(\Lambda_{\chi_1}) \simeq \cY_{K_\infty,V}\otimes_{\Lambda}Q(\Lambda_{\chi_1}),$$
by (ii).
This has been shown in the proof of Theorem \ref{lemisom}. We prove that $\cX_{K_\infty,S\setminus V}\otimes_\Lambda R$ is annihilated by $P$. Note that
$$\cX_{K_\infty,S\setminus V}\otimes_\Lambda R=\cX_{K_\infty,S \setminus (V\cup S_\infty)}\otimes_\Lambda R,$$
since the complex conjugation $c$ at $v \in S_\infty\setminus (V\cap S_\infty)$ is non-trivial in $G_{\chi_1}$, and hence $c-1\in R^\times$. Hence, it is sufficient to show that, for every $v \in S\setminus (V\cup S_\infty)$, there exists $\sigma \in G_v \cap \Gamma$ such that $\sigma -1$ is a uniformizer of $R$, where $G_v\subset \G$ is the decomposition group at a place of $K_\infty$ lying above $v$. Thanks to the assumption (S), we find such $\sigma$ by Lemma \ref{uniformizer}.

The assertion (iv) is immediate from the above argument.

The assertion (v) follows from (iii), (iv), and that
$$\cX_{K_\infty,S}\otimes_\Lambda E\simeq \cX_{L,S}\otimes_{\ZZ_p[G_\chi]}\QQ_p(\chi)\simeq e_\chi\QQ_p(\chi)\cX_{L,S}\simeq e_\chi\QQ_p(\chi)\cY_{L,V'}$$
is an $r'$-dimensional $E$-vector space.
\end{proof}

In the following for any $R$-module $M$ we often denote $M\otimes_R E$ by $M_E$. Also, we assume that (F) is satisfied.

\begin{definition}
The `Bockstein map' is the map
\begin{eqnarray}
\beta: H^0(C_E) &\to& H^1(C\otimes_R P) \nonumber \\
&=& H^1(C)\otimes_R P \nonumber \\
&\to & H^1(C_E)\otimes_E P/P^2\nonumber
\end{eqnarray}
induced by the exact triangle
$$C\otimes_R P \to C \to C_E.$$
\end{definition}

Note that there are canonical isomorphisms
$$H^0(C_E)\simeq U_{L,S,T}\otimes_{\ZZ_p[G_\chi]}\QQ_p(\chi) \simeq e_\chi \QQ_p(\chi) U_{L,S,T},$$
$$H^1(C_E)\simeq \cX_{L,S}\otimes_{\ZZ_p[G_\chi]}\QQ_p(\chi) \simeq e_\chi \QQ_p(\chi)\cX_{L,S}\simeq e_\chi \QQ_p(\chi) \cY_{L,V'},$$
where $\QQ_p(\chi)$ is regarded as a $\ZZ_p[G_\chi]$-algebra via $\chi$. Note also that $P$ is generated by $\gamma^{p^n}-1$ with sufficiently large $n$, where $\gamma$ is a fixed topological generator of $\Gamma$ (see Lemma \ref{uniformizer}). There is a canonical isomorphism
$$I(\Gamma_\chi)/I(\Gamma_\chi)^2\otimes_{\ZZ_p}\QQ_p(\chi) \simeq P/P^2,$$
where $I(\Gamma_\chi)$ denotes the augmentation ideal of $\ZZ_p[[\Gamma_\chi]].$ (Note that $\Gamma=\Gal(K_\infty/K)$ and $\Gamma_\chi=\Gal(L_\infty/L)$.)
Thus, the Bockstein map is regarded as the map
$$\beta: e_\chi\QQ_p(\chi)U_{L,S,T} \to e_\chi \QQ_p(\chi)(\cX_{L,S} \otimes_{\ZZ_p}I(\Gamma_\chi)/I(\Gamma_\chi)^2)\simeq e_\chi \QQ_p(\chi)(\cY_{L,V'} \otimes_{\ZZ_p}I(\Gamma_\chi)/I(\Gamma_\chi)^2).$$

\begin{proposition} \label{proprec}
%Let $I$ be the augmentation ideal of $\ZZ_p[[\Gamma]]$.
The Bockstein map $\beta$ is induced by the map
$$U_{L,S,T} \to \cX_{L,S}\otimes_{\ZZ_p} I(\Gamma_\chi)/I(\Gamma_\chi)^2$$
given by $a \mapsto \sum_{w \in S_L} w \otimes ({\rm rec}_w(a)-1)$.
\end{proposition}

\begin{proof}
The proof is the same as that of \cite[Lemma 5.8]{flachsurvey}. We sketch the proof given in loc. cit.

Take $n$ so that the image of $\gamma^{p^n} \in \Gal(K_\infty/L)$ in $\Gal(L_\infty/L)=\Gamma_\chi$ is a generator. We regard $\gamma^{p^n} \in \Gamma_\chi$. Define $\theta\in H^1(L,\ZZ_p)=\Hom(G_L,\ZZ_p)$ by $\gamma^{p^n}\mapsto 1$. Define
$$\beta' : e_\chi\QQ_p(\chi)U_{L,S,T} \to e_\chi \QQ_p(\chi)(\cX_{L,S} \otimes_{\ZZ_p}I(\Gamma_\chi)/I(\Gamma_\chi)^2) \stackrel{\sim}{\to} e_\chi \QQ_p(\chi)\cX_{L,S}$$
by $\beta(a)=\beta'(a)\otimes (\gamma^{p^n}-1)$. Then, $\beta'$ is induced by the cup product
$$\cdot \cup \theta : \QQ_p U_{L,S}\simeq H^1(\mathcal{O}_{L,S},\QQ_p(1)) \to H^2(\mathcal{O}_{L,S},\QQ_p(1))\simeq \QQ_p \cX_{L,S\setminus S_\infty}.$$
By class field theory we see that $\beta$ is induced by the map $a \mapsto \sum_{w \in S_L\setminus S_\infty(L)} w \otimes ({\rm rec}_w(a)-1)$. Since ${\rm rec}_w(a)=1\in \Gamma_\chi$ for all $w \in S_\infty(L)$, the proposition follows.
\end{proof}

\begin{proposition} \label{ker coker}
Then we have canonical isomorphisms
$$\ker \beta\simeq H^0(C)_E$$
and
$$\coker \beta\simeq H^1(C)_{\rm tf}\otimes_R P/P^2.$$
\end{proposition}
\begin{proof}
Let  $\delta$ be the boundary map $H^0(C_E) \to H^1(C\otimes_R P)=H^1(C) \otimes_R P$. We have
$$\ker \delta\simeq \coker (H^0(C\otimes_R P)\to H^0(C)) = H^0(C)_E$$
and
$$\im \delta =\ker (H^1(C)\otimes_R P \to H^1(C))=H^1(C)[P] \otimes_R P,$$
where $H^1(C)[P]$ is the submodule of $H^1(C)$ which is annihilated by $P$. By Proposition \ref{keylemma} (iii), we know $H^1(C)[P]=H^1(C)_{\rm tors}$. Hence, the natural map
$$H^1(C)\otimes_R P \to H^1(C) \otimes_R P/P^2 \simeq  H^1(C)_E \otimes_E P/P^2\simeq H^1(C_E) \otimes_E P/P^2$$
is injective on $H^1(C)_{\rm tors}\otimes_R P$. From this we see that $\ker \beta \simeq H^0(C)_E$. We also have
$$\coker \beta \simeq \coker (H^1(C)_{\rm tors} \otimes_R P \to H^1(C)\otimes_R P/P^2)\simeq H^1(C)_{\rm tf} \otimes_R P/P^2.$$
Hence we have completed the proof.
\end{proof}

By Lemma \ref{keylemma}, we see that there are canonical isomorphisms
$$H^0(C)_E \simeq U_{K_\infty,S,T} \otimes_\Lambda \QQ_p(\chi),$$
$$H^1(C)_E\simeq \cX_{K_\infty,S} \otimes_\Lambda \QQ_p(\chi),$$
$$H^1(C)_{{\rm tf}, E} \simeq \cY_{K_\infty,V} \otimes_\Lambda \QQ_p(\chi).$$
Hence, by Proposition \ref{ker coker}, we have the exact sequence
\begin{multline*}
0 \to U_{K_\infty,S,T}\otimes_{\Lambda}\QQ_p(\chi) \to e_\chi\QQ_p(\chi)U_{L,S,T} \\
\stackrel{\beta}{\to} e_\chi \QQ_p(\chi) (\cY_{L,V'}\otimes_{\ZZ_p}I(\Gamma_\chi)/I(\Gamma_\chi)^2)\to \cY_{K_\infty,V} \otimes_\Lambda P/P^2 \to 0.
\end{multline*}
This induces an isomorphism
$$\widetilde \beta: e_\chi \QQ_p(\chi)(\bigwedge^{r'}U_{L,S,T} \otimes \bigwedge^{r'} \cY_{L,V'}^\ast )\stackrel{\sim}{\to} \bigwedge^r (U_{K_\infty,S,T}\otimes_\Lambda \QQ_p(\chi)) \otimes \bigwedge^r (\cY_{K_\infty,V}^\ast \otimes_{\Lambda}\QQ_p(\chi))\otimes P^e/P^{e+1}.$$
We have isomorphisms
$$\bigwedge^{r'}\cY_{L,V'}^\ast \stackrel{\sim}{\to} \ZZ_p[G_\chi]; \ w_1^\ast \wedge \cdots \wedge w_{r'}^\ast \mapsto 1,$$
$$\bigwedge^r (\cY_{K_\infty,V}^\ast \otimes_\Lambda \QQ_p(\chi)) \stackrel{\sim}{\to} \QQ_p(\chi); \ w_1^\ast \wedge \cdots \wedge w_r^\ast \mapsto 1.$$
By these isomorphisms, we see that $\widetilde \beta$ induces an isomorphism
$$e_\chi \QQ_p(\chi) \bigwedge^{r'} U_{L,S,T} \stackrel{\sim}{\to} \bigwedge^r(U_{K_\infty,S,T}\otimes_{\Lambda} \QQ_p(\chi)) \otimes P^e /P^{e+1},$$
which we denote also by $\widetilde \beta$. Note that we have a natural injection
$$\bigwedge^r(U_{K_\infty,S,T}\otimes_{\Lambda} \QQ_p(\chi)) \otimes P^e /P^{e+1} \hookrightarrow e_\chi \QQ_p(\chi) (\bigwedge^r U_{L,S,T} \otimes_{\ZZ_p}I(\Gamma_\chi)^e/I(\Gamma_\chi)^{e+1}).$$
Composing this with $\widetilde \beta$, we have an injection
$$\widetilde \beta: e_\chi \QQ_p(\chi) \bigwedge^{r'} U_{L,S,T} \hookrightarrow e_\chi \QQ_p(\chi) (\bigwedge^r U_{L,S,T} \otimes_{\ZZ_p}I(\Gamma_\chi)^e/I(\Gamma_\chi)^{e+1}).$$
By Proposition \ref{proprec}, we obtain the following

\begin{proposition} \label{bock rec}
Let
$${\rm Rec}_\infty: \CC_p \bigwedge^{r'} U_{L,S,T} \rightarrow \CC_p(\bigwedge^r U_{L,S,T} \otimes_{\ZZ_p}I(\Gamma_\chi)^e/I(\Gamma_\chi)^{e+1})$$
be the map defined in \S \ref{formulate mrs}. Then we have
$$(-1)^{re}e_\chi{\rm Rec}_\infty=\widetilde \beta.$$
In particular, $e_\chi{\rm Rec}_\infty$ is injective.
\end{proposition}

\subsection{The proof of the main result} \label{proof main result}
In this section we prove Theorem \ref{mainthm}.

We start with an important technical observation. Let $\Pi_n$ denote the free $\ZZ_p[\G_{\chi,n}]$-module $\Pi_{K_\infty}\otimes_{\Lambda} \ZZ_p[\G_{\chi,n}]$, and $I(\Gamma_{\chi,n})$ denote the augmentation ideal of $\ZZ_p[\Gamma_{\chi,n}]$.

We recall from \cite[Lemma 5.19]{bks1} that the image of
$$\pi_{L_n/k,S,T}^V: {\det}_{\ZZ_p[\G_{\chi,n}]}(C_{L_n,S,T})\to \bigwedge^r \Pi_n$$ is contained in $I(\Gamma_{\chi,n})^e\cdot \bigwedge^r \Pi_n$ (see Proposition \ref{explicit projector}(iii)) and also from \cite[Proposition 4.17]{bks1} that $\nu_n^{-1} \circ \mathcal{N}_n$ induces the map
$$I(\Gamma_{\chi,n})^e \cdot \bigwedge^r \Pi_n \to  \bigwedge^r\Pi_0 \otimes_{\ZZ_p}I(\Gamma_{\chi,n})^e/I(\Gamma_{\chi,n})^{e+1} .$$
\begin{lemma} \label{lemcomm1}
There exists a commutative diagram
$$\xymatrix{
{\det}_{\ZZ_p[\G_{\chi,n}]}(C_{L_n,S,T}) \ar[r] \ar[d]_{\pi_{L_n/k,S,T}^V} & {\det}_{\ZZ_p[G_\chi]}(C_{L,S,T}) \ar[d]^{\pi_{L/k,S,T}^{V'}} \\
I(\Gamma_{\chi,n})^e \cdot\bigwedge^r \Pi_n \ar[d]_{\nu_{n}^{-1} \circ \cN_{n}} & \bigcap^{r'} U_{L,S,T} \ar[d]^{(-1)^{re}{\rm Rec}_{n}} \\
 \bigwedge^r\Pi_0 \otimes_{\ZZ_p}I(\Gamma_{\chi,n})^e/I(\Gamma_{\chi,n})^{e+1} & \bigcap^r U_{L,S,T}\otimes_{\ZZ}I(\Gamma_{\chi,n})^e/I(\Gamma_{\chi,n})^{e+1}. \ar[l]_\supset}
$$
\end{lemma}

\begin{proof}
This follows from Proposition \ref{explicit projector}(iii) and \cite[Lemma 5.21]{bks1}.
\end{proof}

For any intermediate field $F$ of $K_\infty/k$, we denote by $\mathcal{L}_{F/k,S,T}$ the image of the (conjectured) element $\mathcal{L}_{K_\infty/k,S,T}$ of $ {\det}_\Lambda(C_{K_\infty,S,T})$ under the isomorphism
$$\ZZ_p[[\Gal(F/k)]]\otimes_{\La}{\det}_\Lambda(C_{K_\infty,S,T}) \simeq {\det}_{\ZZ_p[[\Gal(F/k)]]}(C_{F,S,T}). $$
Note that, by the proof of Theorem \ref{imcrs}, we have
$$\pi_{L_n/k,S,T}^{V}(\mathcal{L}_{L_n/k,S,T})=\epsilon_{L_n/k,S,T}^V.$$
Hence, Lemma \ref{lemcomm1} implies that
$$(-1)^{re}{\rm Rec}_{n}(\pi_{L/k,S,T}^{V'}(\mathcal{L}_{L/k,S,T}))=\nu_{n}^{-1}\circ \cN_{n}(\epsilon_{L_{n}/k,S,T}^V)=:\kappa_n.$$
We set
\[ \kappa:=(\kappa_n)_n \in \bigcap^r U_{L,S,T} \otimes_{\ZZ_p} \varprojlim_n I(\Gamma_{\chi,n})^e/I(\Gamma_{\chi,n})^{e+1}.\]

%By Lemma \ref{lemcomm2} and Proposition \ref{proppi},
%we see that the image of $\mathcal{L}_{L/k,S,T}$ under the map
%\begin{eqnarray}
%{\det}_{\ZZ_p[G]}(C_{L,S,T})&\to& {\det}_E(C_E) \nonumber \\
%&\simeq& {\bigwedge}_E^{r'}H^0(C_E) \otimes_E {\bigwedge}_E^{r'}H^1(C_E)^\ast \nonumber \\
%&\stackrel{\widetilde \beta}{\to} &{\bigwedge}_E^rH^0(C)_E \otimes_E {\bigwedge}_E^r H^1(C)_{{\rm tf},E}^\ast \otimes_E P^e/P^{e+1} \nonumber \\
%&\subset& {\bigwedge}_{\QQ_p(\chi)}^r e_\chi\QQ_p U_{L,S} \otimes_{\QQ_p(\chi)} {\bigwedge}_{\QQ_p(\chi)}^{r} (e_\chi \QQ_p \cX_{L,S})^\ast \otimes_{\QQ_p(\chi)} P^e/P^{e+1} \nonumber
%\end{eqnarray}
%is equal to $e_\chi i_{L_\chi,S,V}( \kappa_\infty)$.

Then the validity of Conjecture ${\rm MRS}(K_\infty/k,S,T,\chi,V')$ implies that
$$e_\chi \kappa= (-1)^{re} e_\chi{\rm Rec}_\infty( \epsilon_{L/k,S,T}^{V'}).$$
%Since $\widetilde \beta$ is an isomorphism, this shows that the image of $z_{L/k,S,T}$ under the map
%$${\det}_{\ZZ_p[G]}(C_{L,S,T})\to {\det}_E(C_E) \simeq {\bigwedge}_E^{r'}H^0(C_E) \otimes_E {\bigwedge}_E^{r'}H^1(C_E)^\ast$$
%is equal to $e_\chi \eta_{L_\chi/k,S,T}$.

In addition, by Proposition \ref{bock rec}, we know that $e_\chi{\rm Rec}_\infty$ is injective, and so
$$\pi_{L/k,S,T}^{V'}(e_\chi \mathcal{L}_{L/k,S,T})=e_\chi \epsilon_{L/k,S,T}^{V'}.$$

Hence, by Proposition \ref{etncbyrs}, we see that Conjecture ${\rm eTNC}(h^0(\Spec L),\ZZ_p[G])$ is valid, as claimed.

\end{document}